\documentclass[reqno,11pt]{amsart}

\usepackage{a4wide}
\usepackage{color}
\usepackage{mathrsfs}
\usepackage{mathtools}
\usepackage{amsmath}
\usepackage{amssymb}
\usepackage{esint}
\numberwithin{equation}{section}
\usepackage[colorlinks,citecolor=green,linkcolor=red]{hyperref}
\usepackage{enumitem}

\usepackage[latin1]{inputenc}

\newcommand{\nchi}{{\raise.3ex\hbox{\(\chi\)}}}
\newcommand{\N}{\mathbb{N}}
\newcommand{\R}{\mathbb{R}}

\newcommand{\D}{{\rm D}}

\newcommand{\sfd}{{\sf d}}
\renewcommand{\d}{{\mathrm d}}
\newcommand{\e}{{\rm e}}
\newcommand{\X}{{\rm X}}
\newcommand{\Y}{{\rm Y}}

\newcommand{\mm}{\mathfrak{m}}

\newcommand{\Mod}{{\rm Mod}}

\newcommand{\LIP}{{\rm LIP}}
\newcommand{\Lip}{{\rm Lip}}
\newcommand{\lip}{{\rm lip}}

\newcommand{\ppi}{{\mbox{\boldmath\(\pi\)}}}

\newcommand{\sppi}{{\mbox{\scriptsize\boldmath\(\pi\)}}}

\newcommand{\limi}{\varliminf}
\newcommand{\lims}{\varlimsup}

\newcommand{\eps}{\varepsilon}

\newcommand{\distd}{{\rm dist}_{\sfd}}
\newcommand{\AMMod}{{\rm Mod}_{\rm AM}}
\newcommand{\AM}{{\rm AM}}
\def\vint_#1{\mathchoice
	{\mathop{\vrule width 6pt height 3 pt depth -2.5pt
			\kern -8pt \intop}\nolimits_{#1}}%
	{\mathop{\vrule width 5pt height 3 pt depth -2.6pt
			\kern -6pt \intop}\nolimits_{#1}}%
	{\mathop{\vrule width 5pt height 3 pt depth -2.6pt
			\kern -6pt \intop}\nolimits_{#1}}%
	{\mathop{\vrule width 5pt height 3 pt depth -2.6pt
			\kern -6pt \intop}\nolimits_{#1}}}

\newcommand{\fr}{\penalty-20\null\hfill\(\blacksquare\)}

\newtheorem{theorem}{Theorem}[section]
\newtheorem{corollary}[theorem]{Corollary}
\newtheorem{lemma}[theorem]{Lemma}
\newtheorem{proposition}[theorem]{Proposition}
\newtheorem{definition}[theorem]{Definition}
\newtheorem{example}[theorem]{Example}
\newtheorem{remark}[theorem]{Remark}

\title[Equivalence of BV notions]
{On the equivalence of BV notions \\ in metric measure spaces}


\author[L. Ambrosio]{Luigi Ambrosio}
\address{Scuola Normale Superiore, Piazza dei Cavalieri 7, 56126 Pisa}
\email{luigi.ambrosio@sns.it}

\author[E. Pasqualetto]{Enrico Pasqualetto}
\address{Department of Mathematics and Statistics,
P.O.\ Box 35 (MaD), FI-40014 University of Jyvaskyla}
\email{enrico.e.pasqualetto@jyu.fi}

\author[N. Shanmugalingam]{Nageswari Shanmugalingam}
\address{Department of Mathematical Sciences, P.O.~Box 210025, University of Cincinnati, Cincinnati, OH~45221-0025, U.S.A.}
\email{shanmun@uc.edu}
\begin{document}
\date{\today}

\begin{abstract} The aim of the paper is to compare in detail several notions of $BV$ space of
functions of bounded variation in metric measure spaces $({\rm X},\mathsf d,\mathfrak m)$. Informally,
they can be grouped in two classes, either by a relaxation procedure starting from a class of nicer functions
(and with different notions of pseudo-gradient in the relaxation procedure) or by requiring good
behaviour along a rich class of absolutely continuous curves. In the second approach, richness
can be understood according to the notion of approximation modulus of \cite{Martio} or according to the
notion of test plan introduced in \cite{Inventiones}. Extending \cite{Amb:DiMa:14}, we prove that all these
approaches are isometrically equivalent in any locally complete metric measure space.
\end{abstract}

\maketitle
\tableofcontents
\section{Introduction}

In this paper we study in depth several notions of $BV$ classes in metric measure spaces $(\X,\sfd,\mm)$. 
After the seminal paper \cite{Miranda}, which provided a definition based on the relaxation of the local Lipschitz
constant for Lipschitz functions, in \cite{Amb:DiMa:14} an equivalent notion was provided, based on the behaviour
of the function along almost every curve, in the sense provided by the theory of transport plans introduced in
\cite{Inventiones}.
This equivalence result, based on ideas first developed in the Sobolev setting (see for instance
\cite{Density,HKST}) was established in a very strong sense in complete metric measure spaces, 
independently of any other structural assumption, with exact coincidence of the norms arising from the two
definitions. 

After these papers, more recently other notions of $BV$ classes based on good behaviour along curves
appeared, relying on the notion of AM-modulus, see \cite{Martio,HMMartio-1} and also \cite{DuCa:ErBiq:Ko:Shan:19}, where an equivalence result
with other consolidated definitions has been established under doubling and Poincar\'e assumptions
on the metric measure structure. The results of \cite{DuCa:ErBiq:Ko:Shan:19}, which also provide equivalence between different
notions of Poincar\'e inequalities, are in part reproduced in Section~\ref{sec:doubling}, with the aim to rectify a typo
in the definition of the so-called AM-bounding sequence, see also the discussion after Definition~\ref{defAMbounding} and
Example~\ref{ex:strong_am_bound}.\footnote{A similar issue occurs in \cite[Appendix A]{Nob:Pas:Sch:22}:
due to an imprecise interpretation of the definition of AM-upper bound, the proof of \cite[Theorem A.6]{Nob:Pas:Sch:22}
is incorrect, and it is still unclear whether its statement is in fact true or not.}

The aim of this paper is to improve the result of \cite{DuCa:ErBiq:Ko:Shan:19}, getting an equivalence result between
various (and basically all) definitions independently of doubling and Poincar\'e assumptions, in any locally complete
metric measure space (hence, in any open subset of a complete metric space). Our result, which relies essentially on a refinement and adaptation of the tools
developed in \cite{Amb:DiMa:14}, provides also exact coincidence of the norms arising from the different
definitions. 

First, we discuss informally the sequence of inclusions below, which hold in {\it any} metric measure
space, thanks to a direct comparison of the definitions (see Theorem~\ref{thm:incl_BV_spaces}):
\begin{equation}\label{chain1}
BV^\star(\X)\subseteq BV(\X)\subseteq BV_N(\X)
\end{equation}
\begin{equation}\label{chain2}
BV_N(\X)\subseteq BV_{am}^0(\X)\subseteq BV_{\AM}(\X)\cap BV_{am}^1(\X)
\end{equation}
\begin{equation}\label{chain3}
BV_{\AM}(\X)\cup BV_{am}^1(\X)\subseteq BV_{tp}^\star(\X).
\end{equation}

The space $BV^\star(\X)$ provides the most restrictive notion, as it involves a relaxation procedure applied to
a small class of functions (the Lipschitz functions with bounded support) and the largest pseudo-gradient 
(the asymptotic Lipschitz constant \eqref{asylip}). If we consider, instead, the larger class of locally Lipschitz functions with bounded support, we are led to the larger space $BV(\X)$ or to the even larger space $BV_N(\X)$, if in the approximation
functions in the Newtonian space $N^{1,1}(\X)$ are considered, together with their (minimal) weak upper gradients.
In connection with the invariance of the choice of the pseudo gradient in the approximation procedure,
see also the more general discussion in Appendix~\ref{sec:genrel}, essentially based on Lemma~\ref{lem:lip_a_vs_lip}.

In \eqref{chain2} we can compare $BV_N(\X)$ with notions based on the AM-modulus of curves, denoted in this paper by
${\rm AM}_0$. The space corresponding to this notion of modulus, denoted by
$BV_\AM(\X)$, was first proposed by Martio in~\cite{Martio, Plans_modulus_duality} and subsequently also
considered in~\cite{DuCa:ErBiq:Ko:Shan:19}. 
We discuss three different notions, which can be naturally ordered when different exceptions are considered in the basic inequality
of the theory, namely
$$
|f(\gamma_b)-f(\gamma_a)|\leq\varliminf_k\int_\gamma\rho_k\qquad\text{with $\gamma:[a,b]\to\X$}.
$$
Indeed, we can measure exceptional curves either with the ${\rm AM}_0$-modulus of with the ${\rm AM}_1$-modulus
(see Definition~\ref{AML}), but we can also measure exceptional endpoints either using an $\mm$-negligible set $N\subseteq\X$, or a $\gamma$-dependent null set $N_\gamma$ with $\mathscr{H}^1(\gamma(N_\gamma))=0$, 
see Section~\ref{sec:BVAM} for the formal definitions. 

The final inclusion in \eqref{chain2} and the initial one in \eqref{chain3} are written in this way because,
at this very general level, we have not been able to compare $BV_{\AM}(\X)$ with $BV_{am}^1(\X)$. In any
case both are proved to be always included in the space $BV_{tp}^\star(\X)$, whose definition is very close
to the one proposed in Remark~7.2 of \cite{Amb:DiMa:14} (see the beginning of Section~\ref{sec:tpp} and Appendix~\ref{sec:tppp}
for a more precise discussion). 

Eventually, in Theorem~\ref{thm:equiv_BV_complete}, adapting ideas of \cite{Amb:DiMa:14} recalled in Appendix~\ref{sec:Cheeger1}, we prove that the ``largest''
space $BV_{tp}^\star(\X)$ coincides with the ``smallest'' space $BV^\star(\X)$ in complete metric measure
spaces, so that in the end \textit{all} spaces coincide with equality of the total variation norms. By a localization procedure, this result can be extended to locally complete metric measure spaces,
see Theorem~\ref{thm:equiv_BV_loc-complete}; here we dropped the smaller space $BV^\star(\X)$ because we cannot expect equality with $BV(\X)$ in the locally but not globally complete context, see for instance Example~\ref{ex:different_BV}.
\medskip

We close this introduction with the discussion of three open questions.

\smallskip

$\bullet$ \textbf{Question 1.} We prove in Lemma~\ref{lem:Nages2} (see also Proposition \ref{prop:bdry_reg_tp}) that
$BV$ functions have {\it essential} bounded variation along ${\rm AM}_0$-almost every curve.
A more difficult question would be to prove the existence of a {\it precise} representative $f^*$ in the Lebesgue equivalence class with the
property that $f^*\circ\gamma$ has {\it pointwise} finite variation 
along ${\rm AM}_0$-almost every curve, as it happens in the Sobolev theory
(see \cite{HKST} for instance), see the discussion in Subsection~\ref{subsec:BV-real},
where the essential variation of $f$ is denoted ${\rm eV}(f)$ while the pointwise
finite variation is denoted ${\rm pV}(f)$. This problem seems to be open even if the doubling and Poincar\'e assumptions are added.
\smallskip

$\bullet$ \textbf{Question 2.} Do we have an example of metric measure space where \(BV(\X)\neq BV_{tp}^\star(\X)\)?
In view of Theorem~\ref{thm:equiv_BV_loc-complete}, such a space cannot be locally complete.

\smallskip

$\bullet$ \textbf{Question 3.} According to the chosen axiomatisation in this note, the reference measure \(\mm\) of a metric measure space is assumed to be Radon
(and also test plans are assumed to be Radon). However, since (in general) we are not assuming separability nor completeness of the metric space,
it is not clear whether the total variation measure \(|Df|\) of a BV function, though $\sigma$-additive on Borel sets, is Radon. 
Are there examples where \(|Df|\) is not Radon?
\subsubsection*{Acknowledgements}
We thank Francesco Nobili and Timo Schultz for the useful discussions on the topics of this paper.
L.A.'s work is partially supported by a grant of the Balzan Foundation. E.P.'s work is partially supported by
the Research Council of Finland, grant 362898. N.S.'s work is partially supported by a grant from the National
Science Foundation (USA), DMS~\#2348748.
\section{Preliminaries}
Let us fix general notations and conventions, which will be used throughout the whole paper. Given two sets \(E\subseteq\X\),
we denote by \(\nchi_E\colon\X\to\{0,1\}\) the characteristic function of \(E\). Given a map \(\varphi\colon\X\to\Y\)
between two sets \(\X\), \(\Y\) and a subset \(E\subseteq\X\), we denote by \(\varphi_{|E}\colon E\to\Y\) the restriction of
\(\varphi\) to \(E\). The one-dimensional Lebesgue measure on a Euclidean interval
will be denoted by \(\mathscr L^1\) and the one-dimensional
Hausdorff measure on a metric space $(\X, \sfd)$ is denoted by \(\mathscr H^1\). Moreover, if $A,B\subseteq\X$
are non-empty sets, then
the distance $\distd(A,B)$ between the two sets is the number
\[
\distd(A,B):=\inf\{\sfd(x,y)\, :\, x\in A,\, y\in B\}.
\]
\subsection{Functions of bounded variation on open intervals}\label{subsec:BV-real}
While some classical textbooks on functions of bounded variation on an interval usually look at closed intervals
(see for example~\cite[Section~6.3]{RoyFitz}), 
in this section we follow the text~\cite[Section~3.2]{Amb:Fus:Pal:00} and discuss functions of bounded variation on an open interval.
For $a,b\in[-\infty,\infty]$ with $a<b$, the pointwise variation (also called total variation in some classical textbooks) 
of a function $f:(a,b)\to\R$ is the number
\[
{\rm pV}(f,(a,b))=\sup\bigg\lbrace \sum_{k=1}^{n-1}|f(x_{k+1})-f(x_k)|\, :\, n\ge 2\text{ and }a<x_1<\cdots<x_n<b\bigg\rbrace.
\]
A function of bounded variation on $(a,b)$ can be extended to the closure (in $\R$) of the interval $(a,b)$ as a function of bounded
variation in the classical sense.
The \emph{essential variation} ${\rm eV}(f,(a,b))$ is the number
\[
{\rm eV}(u,(a,b))\coloneqq\inf\{{\rm pV}(v,(a,b))\, :\, v=u\  \  \mathscr L^1{\rm -a.e.~in }\, (a,b)\}.
\]
This notion of essential variation has a natural extension to open subsets of $\R$, for such open sets are countable disjoint unions of open intervals,
and so the essential variation of the function on the open set is the sum of the essential variations on these open intervals.
Given a function \(u\colon(a,b)\to\R\) on an interval $(a,b)$
and an open set \(\Omega\subseteq(a,b)\), we denote by \({\rm eV}(u,\Omega)\) the \emph{essential variation}
of \(u\) in \(\Omega\). When \({\rm eV}(u,(a,b))<\infty\), the function $u$ belongs to
$L^\infty(a,b)$ and the set-function
\(\Omega\mapsto{\rm eV}(u,\Omega)\), defined on open subsets $\Omega$ of $(a,b)$,
can be uniquely extended to a finite Borel measure \({\rm eV}(u,\cdot)\) on \((a,b)\). This Borel measure 
coincides with $|Du|$, the total variation of the distributional derivative of $u$ (cf.\ \cite[Section~3.2]{Amb:Fus:Pal:00}).
\subsection{Lipschitz maps}
Let \((\X,\sfd_\X)\), \((\Y,\sfd_\Y)\) be given metric spaces. We denote by
\[
B_r(x)\coloneqq\{y\in\X\;|\;\sfd_\X(x,y)<r\},\qquad\bar B_r(x)\coloneqq\{y\in\X\;|\;\sfd_\X(x,y)\leq r\}
\]
the \emph{open ball} and the \emph{closed ball} with
center \(x\in\X\) and radius \(r>0\), respectively
(occasionally denoted also by $B(x,r)$ and $\bar B(x,r)$). We denote by \(C(\X;\Y)\) the space
of continuous maps from \(\X\) to \(\Y\), and by \(C_b(\X;\Y)\) its subspace consisting of those maps \(\varphi\in C(\X;\Y)\) whose image
\(\varphi(\X)\) is a bounded subset of \(\Y\). For any \(\varphi\in C(\X;\Y)\), we denote by \(\Lip(\varphi)\in[0,\infty]\) its \emph{Lipschitz constant}.
Recall that \(\varphi\) is a Lipschitz map if and only if \(\Lip(\varphi)<\infty\). Moreover, we say that a map \(\varphi\in C(\X;\Y)\) is
\emph{locally Lipschitz continuous} if 
for every \(x\in\X\) there exists \(r_x>0\) such that \(\varphi_{|B_{r_x}(x)}\) is Lipschitz continuous. We define
\begin{align*}
\LIP_{loc}(\X;\Y)&\coloneqq\big\{\varphi\in C(\X;\Y)\;\big|\;\varphi\text{ is locally Lipschitz}\big\},\\
\LIP_{b,loc}(\X;\Y)&\coloneqq C_b(\X;\Y)\cap\LIP_{loc}(\X;\Y),\\
\LIP(\X;\Y)&\coloneqq\big\{\varphi\in C(\X;\Y)\;\big|\; \Lip(\varphi)<\infty\big\},\\
\LIP_b(\X;\Y)&\coloneqq C_b(\X;\Y)\cap\LIP(\X;\Y).
\end{align*}
In the specific case where \((\Y,\sfd_\Y)\) is the real line \(\R\) together with the Euclidean distance, we write \(C(\X)\), \(C_b(\X)\), \(\LIP_{loc}(\X)\), \(\LIP_{b,loc}(\X)\),
\(\LIP(\X)\) and \(\LIP_b(\X)\) instead. Recall that \(C_b(\X)\) is a Banach space with respect to the supremum norm \(\|f\|_{C_b(\X)}\coloneqq\sup_\X|f|\).
The \emph{support} \({\rm spt}(f)\) of a function \(f\in C(\X)\) is defined as the closure of \(\{f\neq 0\}\) in \(\X\). We then define
\[
C_{bs}(\X)\coloneqq\big\{f\in C(\X)\;\big|\;{\rm spt}(f)\text{ is bounded}\big\},\quad\LIP_{bs}(\X)\coloneqq C_{bs}(\X)\cap\LIP(\X)\subseteq\LIP_b(\X).
\]
Given any \(f\in\LIP_{loc}(\X)\), we define the \emph{local Lipschitz constant} \(\lip(f)\colon\X\to[0,\infty)\) of \(f\) as
\begin{equation}\label{eq:def_slope}
\lip(f)(x)\coloneqq\lims_{y\to x}\frac{|f(y)-f(x)|}{\sfd(y,x)}\quad\text{ for every accumulation point }x\in\X
\end{equation}
and \(\lip(f)(x)\coloneqq 0\) for every isolated point \(x\in\X\). The function \(\lip(f)\) is Borel, see for 
instance~\cite[Lemma~6.2.5]{HKST}.
We define its \emph{asymptotic Lipschitz constant} \(\lip_a(f)\colon\X\to[0,\infty)\) as
\begin{equation}\label{asylip}
\lip_a(f)(x)\coloneqq\inf_{r>0} \, \sup_{z,w\in B_r(x); z\ne w}\ \frac{|f(z)-f(w)|}{d(z,w)}
\quad\text{ for every }x\in\X.
\end{equation}
The function \(\lip_a(f)\) is upper semicontinuous, thus in particular it is a Borel function.
The asymptotic Lipschitz constant satisfies several elementary calculus rules, among which the following Leibniz-type inequality:
for any \(f,g\in\LIP_{loc}(\X)\), it holds that \(fg\in\LIP_{loc}(\X)\) and
\begin{equation}\label{eq:Leibniz_lip_a}
\lip_a(fg)\leq|f|\lip_a(g)+|g|\lip_a(f).
\end{equation}

A Lipschitz function on a subset of a metric space can be extended to be a Lipschitz function on the entire metric space
without increasing the Lipschitz constant of the function; this result is originally from the work of McShane~\cite{McShane}, and a simple
truncation also ensures that such an extension has the same maximum and minimum values on the entire metric space as the original
function does on its domain of definition. However, such an extension may not preserve the asymptotic Lipschitz constant even on
the domain of definition of the original function. Indeed,~\cite{DiMa:Gig:Pra:21} gives an example which demonstrates the possibility that
such an extension cannot happen.
The following relaxation allows us to preserve the asymptotic Lipschitz constant at the 
expense of slightly increasing the global Lipschitz constant.
\begin{theorem}[Di Marino--Gigli--Pratelli \cite{DiMa:Gig:Pra:21}]\label{thm:DiMa_Gig_Pra}
Let \((\X,\sfd)\) be a metric space and \(E\subseteq\X\) an arbitrary set. Fix any \(f\in\LIP_b(E)\) and \(\varepsilon>0\).
Then there exists an extension \(\bar f\in\LIP_b(\X)\) of the function \(f\) such that \(\inf_E f\leq\bar f\leq\sup_E f\)
on \(\X\), as well as \(\Lip(\bar f)\leq\Lip(f)+\varepsilon\) and
\[
\lip_a(\bar f)(x)=\lip_a(f)(x)\quad\text{ for every }x\in E.
\]
If in addition \({\rm spt}(f)\) is bounded, then \(\bar f\) can be also chosen so that \({\rm spt}(\bar f)\) is bounded.
\end{theorem}
\begin{proposition}[Gutev {\cite[Proposition 4.4]{Gut:20}}]\label{prop:Gutev}
Let \((\X,\sfd)\) be a metric space and \(C\subseteq\X\) a closed set. Fix any \(f\in\LIP_{b,loc}(C)\). Then \(f\) can be extended
to a function \(\bar f\in\LIP_{b,loc}(\X)\).
\end{proposition}

The closedness assumption in Proposition \ref{prop:Gutev} cannot be dropped (e.g.\ the locally Lipschitz function
\((0,1]\ni t\mapsto\sin(1/t)\in[-1,1]\) cannot be extended to a continuous function on \([0,1]\)).
Combining Theorem \ref{thm:DiMa_Gig_Pra} with Proposition \ref{prop:Gutev}, we obtain the following result.
\begin{theorem}[Global \(\lip_a\)-preserving extensions]\label{thm:loc_Lip_extension}
Let \((\X,\sfd)\) be a metric space and \(C\subseteq\X\) a closed set. Fix any \(f\in\LIP_{b,loc}(C)\).
Then there exists an extension \(\bar f\in\LIP_{b,loc}(\X)\) of the function \(f\) such that \(\inf_C f\leq\bar f\leq\sup_C f\) on \(\X\) and
\[
\lip_a(\bar f)(x)=\lip_a(f)(x)\quad\text{ for every }x\in C.
\]
\end{theorem}
\begin{proof}
If \(C=\X\), there is nothing to prove. Thus, assume that \(U\coloneqq\X\setminus C\neq\varnothing\), and let \(\bar x\in U\).
For any \(x\in C\), take some \(r_x>0\) for which \(f_{|B_r(x)\cap C}\) is Lipschitz continuous and \(\bar x\notin B_{r_x}(x)\). Each metric
space is paracompact~\cite{Sto:48}, and \(\{U\}\cup\{B_{r_x}(x):x\in C\}\) is an open cover of \(\X\).
Thus we can find a family of points \(\{x_i:i\in I\}\subseteq C\) such that \(\{U\}\cup\{B_i:i\in I\}\) is a
locally-finite cover of \(\X\), where we set \(B_i\coloneqq B_{r_{x_i}}(x_i)\) for brevity; note that the extracted subcover must
contain \(U\), as \(\bar x\notin\bigcup_{x\in C}B_{r_x}(x)\). However, the balls $B_i$ together forms an open cover of the closed
set $C$. The plan is to first extend $f$ to $\bigcup_{i\in I}B_i$ using Theorem~\ref{thm:DiMa_Gig_Pra} on each $B_i$, and then 
to modify using
a partition of unity subordinate to the cover $\{B_i\}_{i\in I}$. Finally, we use Proposition~\ref{prop:Gutev} to obtain a global extension, and
modify this extension using the function from the previous step so that the asymptotic Lipschitz constant function on $C$ is not modified from the
previous step.

Given any \(i\in I\), from Theorem~\ref{thm:DiMa_Gig_Pra} there
exists an extension \(\bar f_i\in\LIP_b(\X)\) of \(f_{|B_i\cap C}\) such that \(\lip_a(\bar f_i)(x)=\lip_a(f_{|B_i\cap C})(x)\) for
every \(x\in B_i\cap C\) and \(\inf_{B_i\cap C}f\leq\bar f_i\leq\sup_{B_i\cap C}f\) on \(\X\). Moreover, we know from Proposition \ref{prop:Gutev}
that there exists an extension \(\tilde f\in\LIP_{loc}(\X)\) of \(f_{|C}\). Up to a truncation, we can also assume that \(\inf_C f\leq\tilde f\leq\sup_C f\) on \(\X\).
Now, fix a Lipschitz partition of unity \(\{\eta\}\cup\{\eta_i:i\in I\}\subseteq\LIP(\X;[0,1])\) subordinated to \(\{U\}\cup\{B_i:i\in I\}\), i.e.\ \({\rm spt}(\eta_i)\subseteq B_i\)
for every \(i\in I\), \({\rm spt}(\eta)\subseteq U\) and \(\eta(x)+\sum_{i\in I}\eta_i(x)=1\) for all \(x\in\X\);
see e.g.\ \cite[Theorem 2.6.5]{Cob:Mic:Nic:19}. Let us then define
\[
\bar f\coloneqq\eta\tilde f+\sum_{i\in I}\eta_i\bar f_i\in\LIP_{loc}(\X).
\]
Note that \(\inf_C f\leq\bar f\leq\sup_C f\) on \(\X\) and \(\bar f|_C=f\). Given any \(x\in C\), there exists
a finite set \(F_x\subseteq I\) such that some neighbourhood of \(x\) does not intersect \(\bigcup_{i\in I\setminus F_x}B_i\),
thus in particular we have \(\eta+\sum_{i\in F_x}\eta_i=1\) and \(\bar f=\eta\tilde f+\sum_{i\in F_x}\eta_i\bar f_i\) on some
neighbourhood of \(x\). Therefore,
\begin{align*}
&\lip_a(\bar f)(x)=\lip_a\bigg(\eta\tilde f+\sum_{i\in F_x}\eta_i\bar f_i\bigg)(x)
=\lip_a\bigg(\eta(\tilde f-f(x))+\sum_{i\in F_x}\eta_i(\bar f_i-f(x))\bigg)(x)\\
\leq\,&\lip_a\big(\eta(\tilde f-f(x))\big)(x)+\sum_{i\in F_x}\lip_a\big(\eta_i(\bar f_i-f(x))\big)(x)\\
\leq\,&\eta(x)\lip_a(\tilde f)(x)+|\tilde f(x)-f(x)|\lip_a(\eta)(x)+\sum_{i\in F_x}\big(\eta_i(x)\lip_a(\bar f_i)(x)+|\bar f_i(x)-f(x)|\lip_a(\eta_i)(x)\big)\\
=\,&\sum_{i\in F_x}\eta_i(x)\lip_a(\bar f_i)(x)=\sum_{i\in F_x}\eta_i(x)\lip_a(f_{|B_i\cap C})(x)\leq\lip_a(f)(x).
\end{align*}
Since \(\bar f\) extends \(f\), we also have \(\lip_a(f)(x)\leq\lip_a(\bar f)(x)\), thus the statement is proved.
\end{proof}
\subsection{Metric measure spaces}
Given a measure space \((\X,\Sigma,\mu)\) and an exponent \(p\in[1,\infty]\), we denote by \(\mathcal L^p(\X)\) (or \(\mathcal L^p(\mu)\))
the space of all Borel functions \(f\colon\X\to\R\) that are \(p\)-integrable with respect to \(\mu\) (to be interpreted as `bounded' for \(p=\infty\)).
Recall that \((\mathcal L^p(\X),\|\cdot\|_{\mathcal L^p(\X)})\) is a complete seminormed space, where we define \(\|f\|_{\mathcal L^p(\X)}\coloneqq(\int_{\X}|f|^p\,\d\mu)^{1/p}\)
when \(p<\infty\), or \(\|f\|_{\mathcal L^\infty(\X)}\coloneqq\sup_\X|f|\). Taking the quotient up to \(\mu\)-a.e.\ equality, we obtain
the \emph{\(p\)-Lebesgue space} \(L^p(\X)\) (or \(L^p(\mu)\)), which is Banach with respect to the quotient norm \(\|\cdot\|_{L^p(\X)}\).
We denote by \([f]_\mu\in L^p(\X)\) the equivalence class of a function \(f\in\mathcal L^p(\X)\). 
\medskip

Let \((\X,\sfd)\) be a metric space. We denote by \(\tau(\X)\) the topology induced by \(\sfd\), and by \(\mathscr B(\X)\) the Borel \(\sigma\)-algebra
of \((\X,\sfd)\). Any $\sigma$-additive set-function \(\mu\colon\mathscr B(\X)\to[0,\infty]\) with \(\mu(\varnothing)=0\)
is said to be a \emph{Borel measure} on \(\X\). Given a set \(E\in\mathscr B(\X)\),
we define the restricted measure \(\mu\llcorner E\colon\mathscr B(\X)\to[0,\infty]\) as
\[
(\mu\llcorner E)(F)\coloneqq\mu(E\cap F)\quad\text{ for every }F\in\mathscr B(\X).
\]
We say that \(\mu\) is \emph{locally finite} if for any \(x\in\X\) there exists \(r_x>0\) such that \(\mu(B_{r_x}(x))<\infty\), while it is said to be
\emph{boundedly finite} if \(\mu(B)<\infty\) for every bounded Borel set \(B\subseteq\X\). Recall that \(\mu\) is called a \emph{Radon measure} provided
it satisfies the following properties:
\begin{enumerate}[label=(\roman*)]
\item \(\mu(K)<\infty\) for every compact set \(K\subseteq\X\).
\item \(\mu\) is \emph{inner regular on open sets}, meaning that
\[
\mu(U)=\sup\big\{\mu(K)\;\big|\;K\subseteq\X\text{ compact, }K\subseteq U\big\}\quad\text{ for every }U\in\tau(\X).
\]
\item \(\mu\) is \emph{outer regular (on Borel sets)}, meaning that
\[
\mu(E)=\inf\big\{\mu(U)\;\big|\;U\subseteq\X\text{ open, }E\subseteq U\big\}\quad\text{ for every }E\in\mathscr B(\X).
\]
\end{enumerate}

It is well known that if \((\X,\sfd)\) is a complete separable metric space, then every locally-finite Borel measure on \(\X\) is a Radon measure,
see for instance~\cite[Theorem~3.4.19]{Srivastava}.
If \(\varphi\colon\X\to\Y\) is a Borel map between two metric spaces \((\X,\sfd_\X)\), \((\Y,\sfd_\Y)\) and \(\mu\) is a Borel measure on \(\X\), then we
denote by \(\varphi_\#\mu\) the \emph{pushforward measure} of \(\mu\) under \(\varphi\), which is the Borel measure on \(\Y\) that is defined as
\[
\varphi_\#\mu(E)\coloneqq\mu(\varphi^{-1}(E))\quad\text{ for every }E\in\mathscr B(\Y).
\]
Clearly, if \(\mu\) is a finite measure, then \(\varphi_\#\mu\) is a finite measure
and \(\varphi_\#\mu(\Y)=\mu(\X)\). Moreover, if \(\mu\) is a finite Radon measure and
\(\varphi\) is continuous, then also \(\varphi_\#\mu\) is a Radon measure.
\medskip

In this paper, we adopt the following definition of a metric measure space.
\begin{definition}[Metric measure space]\label{defn:mms}
We say that \((\X,\sfd,\mm)\) is a \emph{metric measure space} provided \((\X,\sfd)\) is a metric space
and \(\mm\) is a (non-negative) \(\sigma\)-finite Radon measure on \(\X\).
\end{definition}

Given a metric measure space \((\X,\sfd,\mm)\) and a Borel set \(E\subseteq\X\), it is easy to check that both
\((E,\sfd_{|E\times E},\mm_{|\mathscr B(E)})\) and \(\X_E\coloneqq(\X,\sfd,\mm\llcorner E)\) are metric measure spaces as well.
\medskip

The \emph{support} of the measure \(\mm\) is defined as
\[
{\rm spt}(\mm)\coloneqq\big\{x\in\X\;\big|\;\mm(B_r(x))>0\text{ for every }r>0\big\}.
\]
In the following result, we collect the main properties of the support \({\rm spt}(\mm)\).
\begin{proposition}\label{prop:main_Radon}
Let \((\X,\sfd)\) be a metric space and \(\mm\) a  Borel measure on \(\X\). Then:
\begin{enumerate}[label=(\roman*)]
\item\label{it:spt_closed} The support \({\rm spt}(\mm)\) is closed.
\item\label{it:spt_separable} If \(\mm\) is \(\sigma\)-finite, then \({\rm spt}(\mm)\) is separable.
\item\label{it:spt_concentr} If \(\mm\) is a Radon measure, then \(\mm\) is concentrated on \({\rm spt}(\mm)\), i.e.\ \(\mm(\X\setminus{\rm spt}(\mm))=0\).
\end{enumerate} 
\end{proposition}
\begin{proof} \ref{it:spt_closed}
Given any \(x\in\X\setminus{\rm spt}(\mm)\), we have \(\mm(B_r(x))=0\) for some \(r>0\). For any \(y\in B_r(x)\), we have that
\(B_{r-\sfd(x,y)}(y)\subseteq B_r(x)\) and thus \(\mm(B_{r-\sfd(x,y)}(y))\leq\mm(B_r(x))=0\). This shows that \(B_r(x)\subseteq\X\setminus{\rm spt}(\mm)\),
which means that \(\X\setminus{\rm spt}(\mm)\) is open (and thus \({\rm spt}(\mm)\) is closed).\\
\ref{it:spt_separable} We argue by contradiction: assume \({\rm spt}(\mm)\) is not separable. Given any \(n\in\N\),
by using Zorn's lemma we obtain a maximal \(\frac{2}{n}\)-separated subset \(M_n\) of \({\rm spt}(\mm)\). Since \(\bigcup_{n\in\N}M_n\) is dense in \({\rm spt}(\mm)\),
and the latter set is non-separable, we deduce that \(M_{n_0}\) is uncountable for some \(n_0\in\N\). By the \(\sigma\)-finiteness of \(\mm\),
we can find a Borel covering \((E_k)_{k\in\N}\) of \(\X\) such that \(\mm(E_k)<+\infty\) for every \(k\in\N\). Hence, for any \(x\in M_{n_0}\) there exists
an index \(k_x\in\N\) such that \(\mm(E_{k_x}\cap B_{1/n_0}(x))>0\). Let us now define
\[
S_{k,j}\coloneqq\big\{x\in M_{n_0}\;\big|\;k_x=k,\,\mm(E_k\cap B_{1/n_0}(x))\geq 1/j\big\}\quad\text{ for every }k,j\in\N.
\]
Since \(B_{1/n_0}(x)\cap B_{1/n_0}(y)=\varnothing\) if \(x,y\in M_{n_0}\) are distinct, for any \(k,j\in\N\) we can estimate
\[
\frac{{\rm card}(F)}{j}\leq\sum_{x\in F}\mm(E_k\cap B_{1/n_0}(x))\leq\mm(E_k)\quad\text{ for every finite subset }F\text{ of }S_{k,j},
\]
thus \({\rm card}(S_{k,j})\leq j\mm(E_k)\). Then \(M_{n_0}=\bigcup_{k,j\in\N}S_{k,j}\) is countable, leading to a contradiction.\\
\ref{it:spt_concentr} Since \(\X\setminus{\rm spt}(\mm)\) is open and \(\mm\) is inner regular on open sets,
we find a sequence \((K_n)_n\) of compact subsets of \(\X\setminus{\rm spt}(\mm)\) such that \(\mm(\X\setminus{\rm spt}(\mm))=\sup_{n\in\N}\mm(K_n)\).
Fix \(n\in\N\). Then we can find \(x_1,\ldots,x_k\in K_n\) and \(r_1,\ldots,r_k>0\) such that \(\mm(B_{r_i}(x_i))=0\) for every \(i=1,\ldots,k\) and
\(K_n\subseteq\bigcup_{i=1}^k B_{r_i}(x_i)\). It follows that \(\mm(K_n)=0\), thus accordingly \(\mm(\X\setminus{\rm spt}(\mm))=0\).
\end{proof}
\begin{lemma}[Density of Lipschitz functions in \(\mathcal L^p(\X)\)]\label{lem:dens_Lip}
Let \((\X,\sfd,\mm)\) be a metric measure space and \(p\in[1,\infty)\). Then 
\(\LIP_{bs}(\X)\cap\mathcal L^p(\X)\) is dense in \(\mathcal L^p(\X)\).
\end{lemma}
\begin{proof}
In the case where \(\mm\) is finite, the statement follows from \cite[Lemma 2.1.27]{Sav:22}. The general case can be deduced as follows.
Given any \(f\in\mathcal L^p(\X)\) and \(\varepsilon>0\), we know (by the dominated convergence theorem) that there exists \(\delta\in(0,1)\)
such that the Borel set
\(E\coloneqq\{\delta\leq|f|^p\leq\delta^{-1}\}\) satisfies \(\int_{\X\setminus E}|f|^p\,\d\mm\leq\varepsilon^p\).
Since \(f\geq\delta\) on \(E\), we have \(\mm(E)\leq\frac{1}{\delta}\int|f|^p\,\d\mm<+\infty\), thus \(\mm\llcorner E\) is a finite
Radon measure. Fix a compact set \(K\subseteq E\) such that \(\mm(E\setminus K)\leq\delta\varepsilon^p\). 
By~\cite[Lemma 2.1.27]{Sav:22}, we find a function
\(\tilde h\in\LIP_b(K)\) such that \(\|\tilde h-f_{|K}\|_{\mathcal L^p(K)}\leq\varepsilon\). Now, fix an open set \(\Omega\subseteq\X\) such that \(\mm(\Omega\setminus K)\leq\varepsilon^p/M^p\),
where \(M\coloneqq\max_K|\tilde h|\). 
Such an open set can be found by the fact that $\mm$ is a Radon measure. Indeed, we can find an open set $\Omega\supset E$ such that
$\mm(\Omega\setminus E)<\varepsilon/2$.
Let \(h\colon\X\to[-M,M]\) be a Lipschitz function that extends \(\tilde h\), whose existence is guaranteed by the 
McShane--Whitney extension theorem \cite{McShane}.
Fix a Lipschitz cut-off function \(\eta\in\LIP_{bs}(\X;[0,1])\) such that \(\eta=1\) in \(K\) and \(\eta=0\) in \(\X\setminus\Omega\). Define
\[
g\coloneqq\eta h\in\LIP_{bs}(\X).
\]
Since \(\mm(\Omega)\leq\mm(K)+\varepsilon<\infty\) and \(|g|\leq M\), we have \(g\in\mathcal L^p(\X)\). Moreover, we can estimate
\[\begin{split}
\|g-f\|_{\mathcal L^p(\X)}&\leq\|g_{|K}-f_{|K}\|_{\mathcal L^p(K)}+\|\nchi_{\X\setminus K}g\|_{\mathcal L^p(\X)}+\|\nchi_{\X\setminus K}f\|_{\mathcal L^p(\X)}\\
&\leq\|\tilde h-f_{|K}\|_{\mathcal L^p(K)}+\mm(\Omega\setminus K)^{1/p}M+\|\nchi_{\X\setminus E}f\|_{\mathcal L^p(\X)}+\|\nchi_{E\setminus K}f\|_{\mathcal L^p(\X)}\\
&\leq 3\varepsilon+\frac{1}{\delta^{1/p}}\mm(E\setminus K)^{1/p}\leq 4\varepsilon,
\end{split}\]
which shows that \(\LIP_{bs}(\X)\cap\mathcal L^p(\X)\) is dense in \(\mathcal L^p(\X)\).
\end{proof}
\begin{remark}\label{rmk:good_cover}{\rm
Given a metric measure space \((\X,\sfd,\mm)\), there exists an increasing sequence \((\Omega_n)_n\)
of open sets in \(\X\) such that \({\rm spt}(\mm)\subseteq\bigcup_{n\in\N}\Omega_n\) and
\(\mm(\overline\Omega_n)<\infty\) for all \(n\in\N\). 
The $\sigma$-finiteness of $\mm$ together with the Radon property 
immediately give us the existence of such a sequence of open sets. However, we give an explicit construction
that also helps us in establishing additional properties of the sets $\Omega_n$, see the last paragraph of this remark below.
Indeed, for any \(x\in{\rm spt}(\mm)\) we can take
\(r_x>0\) such that \(\mm(\bar B_{r_x}(x))<\infty\); existence of such $r_x>0$
follows from the outer regularity
of \(\mm\) (as \(\{x\}\) is compact and thus \(\mm(\{x\})<\infty\)). Since \({\rm spt}(\mm)\)
is separable by Proposition \ref{prop:main_Radon}, we can find \((x_n)_n\subseteq{\rm spt}(\mm)\)
such that \({\rm spt}(\mm)\subseteq\bigcup_{n\in\N}B_{r_{x_n}}(x_n)\). Therefore, the sought sequence
\((\Omega_n)_n\) can be taken as \(\Omega_n\coloneqq\bigcup_{i=1}^n B_{r_{x_i}}(x_i)\) for every \(n\in\N\).

If in addition \((\X,\sfd)\) is locally complete, then we can further require that \((\overline\Omega_n,\sfd)\)
is complete for all \(n\in\N\). To this aim, it suffices to choose \(r_x>0\) so that \((\bar B_{r_x}(x),\sfd)\)
is complete.
\fr}\end{remark}
\begin{theorem}\label{thm:cor_of_DeGiorgi-Letta}
Let \((\X,\sfd)\) be a metric space. Assume that \(\phi\colon\tau(\X)\to[0,\infty)\) is a set-function that satisfies the following conditions:
\begin{enumerate}[label=(\roman*)]
\item\label{it:phi_empty} \(\phi(\varnothing)=0\).
\item\label{it:isoton} \(\phi(U)\leq\phi(V)\) for every \(U,V\in\tau(\X)\) with \(U\subseteq V\).
\item\label{it:sigma-subadd} \(\phi\big(\bigcup_{n\in\N}U_n\big)\leq\sum_{n=1}^\infty\phi(U_n)\) 
for every sequence \((U_n)_n\subseteq\tau(\X)\).
\item\label{it:superadd_on_disj} \(\phi(U\cup V)\geq\phi(U)+\phi(V)\) for every \(U,V\in\tau(\X)\) with \(\distd(U,V)>0\).
\end{enumerate}
Then there exists a unique outer regular Borel measure \(\mu\) on \(\X\) that extends \(\phi\).
\end{theorem}
\begin{proof}
Let us define the set-function \(\mu\colon\mathscr B(\X)\to[0,\infty)\) as
\begin{equation}\label{eq:def_ext_phi}
\mu(E)\coloneqq\inf\big\{\phi(U)\;\big|\;U\in\tau(\X),\,E\subseteq U\big\}\quad\text{ for every }E\in\mathscr B(\X).
\end{equation}
Choosing the sequence \(U,V,\varnothing,\varnothing,\ldots\) in \ref{it:sigma-subadd}, we obtain that
\begin{equation}\label{eq:subadd}
\phi(U\cup V)\leq\phi(U)+\phi(V)\quad\text{ for every }U,V\in\tau(\X).
\end{equation}
Next, fix an increasing sequence of sets \((U_n)_n\subseteq\tau(\X)\) and denote \(U\coloneqq\bigcup_{n\in\N}U_n\).
Given \(n\in\N\), consider the open set \(V_n\coloneqq\{x\in\X:\sfd(x,\X\setminus U_n)>1/n\}\subseteq U_n\). 
Note that \(V_n\subseteq V_{n+1}\) and \(\sfd(V_n,\X\setminus V_{n+1})>0\)
for every \(n\in\N\), and \(U=\bigcup_{n\in\N}V_n\). We claim that \(\phi(U)\leq\sup_{n\in\N}\phi(V_n)\), 
so that \(\phi(U)=\lim_n\phi(U_n)\) by \ref{it:isoton}. To see the truth of this claim, we define the
sequences of open sets \((A_j)_j\) and \((B_j)_j\) as
\[
A_j\coloneqq V_{4j}\setminus\overline{V_{4j-3}},\quad B_j\coloneqq V_{4j+2}\setminus\overline{V_{4j-1}}\quad\text{ for every }j\in\N.
\]
Since \(\sfd(A_j,\X\setminus A_{j+1})>0\) for all \(j\in\N\) by construction, we deduce from \ref{it:isoton} and \ref{it:superadd_on_disj} that \(\sum_{j=1}^\infty\phi(A_j)\leq\phi(\X)<\infty\). Note that as $\X\in\tau(\X)$, by hypothesis $\phi(\X)<\infty$.
Similarly, we also have that \(\sum_{j=1}^\infty\phi(B_j)<\infty\).
Hence, for any given \(\varepsilon>0\), we can find \(j_\varepsilon\in\N\) such that
\(\sum_{j=j_\varepsilon}^\infty\phi(A_j)+\phi(B_j)\leq\varepsilon\). Since we have that
\[
U=V_{4j_\varepsilon-2}\cup\bigcup_{j\geq j_\varepsilon}A_j\cup\bigcup_{j\geq j_\varepsilon}B_j
\]
by construction, it follows from \ref{it:sigma-subadd} that
\[
\phi(U)\leq\phi(V_{4j_\varepsilon-2})+\sum_{j=j_\varepsilon}^\infty\phi(A_j)+\sum_{j=j_\varepsilon}^\infty\phi(B_j)\leq\sup_{n\in\N}\phi(V_n)+\varepsilon.
\]
By the arbitrariness of \(\varepsilon>0\), we conclude that \(\phi(U)\leq\sup_{n\in\N}\phi(V_n)\). Hence, we proved that
\begin{equation}\label{eq:cont_from_below}
\phi\bigg(\bigcup_{n\in\N}U_n\bigg)=\lim_{n\to\infty}\phi(U_n)\quad\text{ for every increasing sequence }(U_n)_n\subseteq\tau(\X).
\end{equation}
In view of \ref{it:phi_empty}, \ref{it:isoton}, \ref{it:superadd_on_disj}, \eqref{eq:subadd} and \eqref{eq:cont_from_below}, we know  
(see for instance \cite[Theorem~1.53]{Amb:Fus:Pal:00})
that the set-function \(\mu\) defined in \eqref{eq:def_ext_phi} is the unique outer regular Borel measure on \(\X\) that extends \(\phi\).
\end{proof}
\subsection{Curves and path integrals}
Let \((\X,\sfd)\) be a metric space. By a \emph{curve} in \(\X\) we mean a continuous map \(\gamma\colon I\to\X\) defined on
some real interval \(I\subseteq\R\). We denote by \(\Gamma(\X)\) the space of all curves in \(\X\) whose domain is a non-empty
compact interval, namely
\[
\Gamma(\X)\coloneqq\bigcup_{\substack{a,b\in\R:\\a\leq b}}C([a,b];\X).
\]
Given any curve \(\gamma\in\Gamma(\X)\), we denote by \([a_\gamma,b_\gamma]\) its domain. The \emph{length} of \(\gamma\) is defined as
\[
\ell(\gamma)\coloneqq\sup\bigg\{\sum_{i=1}^n\sfd(\gamma_{t_i},\gamma_{t_{i-1}})\;\bigg|\;n\in\N,\,a_\gamma=t_0\leq t_1<\cdots\leq t_n=b_\gamma\bigg\}\in[0,\infty].
\]
We say that \(\gamma\) is \emph{rectifiable} if
\(\ell(\gamma)<\infty\). We denote by \(\mathscr R(\X)\subseteq\Gamma(\X)\) the space of all
rectifiable curves in \(\X\). 
In particular, given any \(\gamma\in\mathscr R(\X)\), there exists a unique Lipschitz curve
\[
\hat\gamma\colon[0,\ell(\gamma)]\to\X,
\]
which we call the \emph{arc-length parameterisation} of \(\gamma\), such that
\[
\gamma_t=\hat\gamma_{\ell(\gamma_{|[a_\gamma,t]})}\quad\text{ for every }t\in[a_\gamma,b_\gamma].
\]
Such a parametrization also satisfies
\(\ell(\hat\gamma_{|[r,s]})=s-r\) for every \(r,s\in[0,\ell(\gamma)]\) with \(r\leq s\), thus in particular \(\hat\gamma\) is \(1\)-Lipschitz
and \(\ell(\hat\gamma)=\ell(\gamma)\). 

Henceforth, we will often consider continuous rectifiable curves with their absolutely continuous
and, when necessary, arc-length parameterisation.

For any Borel function \(g\colon\X\to[0,\infty]\), we define the \emph{path integral} of \(g\) over a curve \(\gamma\in\mathscr R(\X)\) as
\[
\int_\gamma g\coloneqq\int_0^{\ell(\gamma)}g(\hat\gamma_s)\,\d s\quad\text{ if }\gamma\in\mathscr R(\X).
\]
More generally, for any absolutely continuous $\gamma\colon[a,b]\to\X$ and any Borel function \(g\colon\X\to[0,\infty]\), we have
\begin{equation}\label{eq:intrinsic}
\int_\gamma g=\int_a^b g(\gamma_t)|\dot\gamma_t|\,\d t=\int_{\gamma([a,b])}g(x){\rm card}(\gamma^{-1}(x))\,\d\mathscr{H}^1(x),
\end{equation}
where \(|\dot\gamma_t|\coloneqq\lims_{h\to 0}\sfd(\gamma_{t+h},\gamma_t)/|h|\in[0,\infty)\) denotes the \emph{metric speed} of \(\gamma\) at \(t\), which
exists for \(\mathscr L^1\)-a.e.\ \(t\in (a,b)\) and we recall that $\mathscr{H}^1$ stands for the 1-dimensional Hausdorff measure in $(\X,\sfd)$
(see for instance \cite{Ambrosio_Tilli}).
Observe that \(\int_{\hat\gamma}g=\int_\gamma g\) for every \(\gamma\in\mathscr R(\X)\).

\medskip

We endow the space \(C([0,1];\X)\) with the supremum distance \(\sfd_{C([0,1];\X)}\), which we define as
\[
\sfd_{C([0,1];\X)}(\gamma,\sigma)\coloneqq\sup_{t\in[0,1]}\sfd(\gamma_t,\sigma_t)\quad\text{ for every }\gamma,\sigma\in C([0,1];\X).
\]
The space \(C([0,1];\X)\) is complete (resp.\ separable) if and only if \(\X\) is complete (resp.\ separable). The \emph{evaluation map} \(\e_t\colon C([0,1];\X)\to\X\)
at time \(t\in[0,1]\) is defined as \(\e_t(\gamma)\coloneqq\gamma_t\) for every \(\gamma\in C([0,1];\X)\). Note that \(\e_t\) is \(1\)-Lipschitz. The space
\(\LIP([0,1];\X)\) is a countable union of closed subsets of \(C([0,1];\X)\), thus in particular it is a Borel subset of \(C([0,1];\X)\). %
\section{Measures on curves}
\subsection{1-modulus}

We recall the classical definition of modulus
in a metric measure space \((\X,\sfd,\mm)\) (we do not emphasize here the dependence of this concept on the integrability exponent
$p$, as only the case $p=1$ is relevant for this paper). 

\begin{definition}[Modulus]
Given \(\Gamma\subseteq\Gamma(\X)\), we denote by $\mathcal{A}(\Gamma)$
the class of Borel functions \(\rho\colon\X\to[0,\infty]\) such that $\int_\gamma\rho\geq 1$ for all 
$\gamma\in\Gamma\cap\mathscr R(\X)$. Note
that if $\Gamma$ contains even one constant path, then $\mathcal{A}(\Gamma)$ is empty. Then, we set
\[
\Mod(\Gamma):=\inf_{\rho\in\mathcal{A}(\Gamma)}\, \int_{\X}\rho\, \d\mm\in [0,\infty],
\]
so that $\Mod(\Gamma)=\infty$ if $\Gamma$ contains a constant path.
\end{definition}

By construction, $\Mod(\Gamma)=\Mod(\Gamma\cap\mathscr R(\X))$. It is well known (see for instance
\cite{HKST}) that $\Mod$ is an outer measure on $\Gamma(X)$. 

The following lemma is originally due to Fuglede, see for instance \cite{Koskela_MacManus}.

\begin{lemma}\label{lem:Fuglede1}
Suppose that $\Gamma\subseteq\Gamma(\X)$. Then $\Mod(\Gamma)=0$ if and only if there is a Borel function $\rho_\Gamma:X\to[0,\infty]$
such that $\int_{\X}\rho_\Gamma\,\d\mm<\infty$ but $\int_\gamma\rho_\Gamma=\infty$ for each $\gamma\in\Gamma\cap\mathscr R(\X)$.
\end{lemma}

For $N\subseteq\X$, we set $\Gamma_N$ to be the collection of all curves $\gamma\in \mathscr R(\X)$ that intersect $N$.
In addition, 
we denote by $\Gamma_N^+\subseteq\Gamma_N$ the collection of all curves $\gamma\in \mathscr R(\X)$ for which
\[
\mathscr{L}^1(\hat{\gamma}^{-1}(N))>0.
\]
Here, by $\mathscr{L}^1(\hat{\gamma}^{-1}(N))$ we mean the Lebesgue outer measure of $\hat{\gamma}^{-1}(N)$.
Let $N$ be a Borel subset of $\X$.
By \eqref{eq:intrinsic} with $g=\nchi_N$, we have that $\gamma\in\Gamma_N^+$ if and only if 
$\mathscr{H}^1(\gamma([a_\gamma,b_\gamma])\cap N)>0$.
If $\mathscr{L}^1(\hat{\gamma}^{-1}(N))=0$, then for every $\eps>0$ we can find a pairwise disjoint (relatively) open
cover $\{(a_i,b_i)\}_i$ of $\hat{\gamma}^{-1}(N)$ such that $\sum_i(b_i-a_i)<\eps$. Note that $\{\hat{\gamma}((a_i,b_i))\}_i$ is a
cover of $N\cap \gamma([a_\gamma,b_\gamma])$ by subcurves, and as $\hat{\gamma}$ is an arc-length parametrization,
the diameter of each $\hat{\gamma}((a_i,b_i))$ is at most $b_i-a_i$, and so $\mathscr{H}^1(\hat{\gamma}([a_\gamma,b_\gamma])\cap N)$
is at most $\sum_i2\, (b_i-a_i)\le 2\eps$; that is, $\mathscr{H}^1(\hat{\gamma}([a_\gamma,b_\gamma])\cap N)=0$. This argument shows that
if $\mathscr{H}^1(\gamma([a_\gamma,b_\gamma])\cap N)>0$, then 
$\mathscr{L}^1(\hat{\gamma}^{-1}(N))>0$.
On the other hand, if $\mathscr{L}^1(\hat{\gamma}^{-1}(N))>0$, then, choosing
$g=\nchi_{N}$, we have $\int_\gamma g=\int_0^{\ell(\gamma)}g\circ\hat{\gamma}(t)\,\d t>0$. 
Now from~\eqref{eq:intrinsic} it follows that 
$\mathscr{H}^1(\gamma([a_\gamma,b_\gamma])\cap N)>0$.

\begin{lemma}\label{Nages1}
Suppose that $N\subseteq\X$ is such that $\mm(N)=0$. Then $\Mod(\Gamma_N^+)=0$.
\end{lemma} 

\begin{proof}
Since $\mu$ is Borel regular, we can find a Borel set $N_0\subseteq\X$ such that $N\subseteq N_0$ and $\mm(N_0)=0$. Moreover,
we have that $\Gamma_N^+\subseteq\Gamma_{N_0}^+$. Therefore, without loss of generality, we may assume that $N$ itself is a Borel set.
Then the function $\rho=\infty\, \nchi_N$ satisfies the criterion set forth in Lemma~\ref{lem:Fuglede1}, and the conclusion follows now from that lemma.
\end{proof}

\begin{definition}[Approximation \(\lambda\)-modulus]\label{AML}
Let \((\X,\sfd,\mm)\) be a metric measure space and \(\lambda\in\{0,1\}\). Let \(\Gamma\subseteq \Gamma(\X)\) be given.
Then a sequence \((\rho_i)_i\) of Borel functions \(\rho_i\colon\X\to[0,\infty]\) is said to be
\emph{\({\rm AM}_\lambda\)-admissible} for \(\Gamma\) provided it holds that
\[
\limi_{i\to\infty}\left(\lambda\big(\rho_i(\gamma_{a_\gamma})+\rho_i(\gamma_{b_\gamma})\big)+\int_\gamma\rho_i\right)\geq 1
\quad\text{ for every }\gamma\in\Gamma\cap\mathscr R(\X).
\]
We denote by \(\mathcal{AM}_\lambda(\Gamma)\) the collection of all \({\rm AM}_\lambda\)-admissible sequences for \(\Gamma\).
Moreover, we define the \emph{approximation \(\lambda\)-modulus} of \(\Gamma\) as
\[
{\rm AM}_\lambda(\Gamma)
 \coloneqq\inf\left\{\limi_{i\to\infty}\int_{\X}\rho_i\,\d\mm\;\middle|\;(\rho_i)_i\in{\mathcal{AM}}_\lambda(\Gamma)\right\}\in[0,\infty].
\]
\end{definition}
By construction, \({{\rm AM}}_\lambda(\Gamma)={{\rm AM}}_\lambda(\Gamma\cap\mathscr R(\X))\).
Notice also that \({\rm AM}_0\) coincides with Martio's approximation modulus $\AMMod$, as defined in \cite{Martio},
but we use a different notation due to the $\lambda$-dependence we introduced.

If $\rho\in\mathcal{A}(\Gamma)$, then the constant sequence $(\rho)_i$ belongs to ${\mathcal{AM}}_\lambda(\Gamma)$, and so it follows that
\[
{\rm AM}_0(\Gamma)\le \Mod(\Gamma).
\]
In the following result, we collect important properties of the approximation \(\lambda\)-modulus.
\begin{proposition}[Some properties of \({\rm AM}_\lambda\)]\label{prop:properties_AM}
Let \((\X,\sfd,\mm)\) be a metric measure space. Then:
\begin{enumerate}[label=(\roman*)]
\item\label{it:AM_outer_meas} \({\rm AM}_\lambda\) is an outer measure on \(\Gamma(\X)\) for both \(\lambda\in\{0,1\}\), namely
\begin{align}
\label{eq:AM_zero}
&{\rm AM}_\lambda(\varnothing)=0,\\
\label{eq:AM_monot}
&{\rm AM}_\lambda(\Gamma)\leq{\rm AM}_\lambda(\tilde\Gamma)\quad\text{ for every }\Gamma\subseteq\tilde\Gamma\subseteq\Gamma(\X),\\
\label{eq:AM_subadd}
&{\rm AM}_\lambda\bigg(\bigcup_{n=1}^\infty\Gamma_n\bigg)\leq\sum_{n=1}^\infty{\rm AM}_\lambda(\Gamma_n)
\quad\text{ if }\Gamma_n\subseteq\Gamma(\X)\text{ for all }n\in\N.
\end{align}
\item\label{it:ineq_AM} \({\rm AM}_1(\Gamma)\leq{\rm AM}_0(\Gamma)\le\Mod(\Gamma)\) 
for every \(\Gamma\subseteq\Gamma(\X)\).
\item\label{it:wlog_const_speed_AM} \({\rm AM}_\lambda(\Gamma)={\rm AM}_\lambda(\{\hat\gamma:\gamma\in\Gamma\cap\mathscr R(\X)\})\)
for all \(\lambda\in\{0,1\}\) and \(\Gamma\subseteq\Gamma(\X)\). 
\item\label{it:ineq_AM_mm} Given any Borel set \(E\subseteq\X\), we have that
\begin{equation}\label{eq:ineq_AM_mm}
\max\bigg\lbrace{\rm AM}_1\left(\left\{\gamma\in\Gamma(\X)\;\middle|\;\gamma_{a_\gamma}\in E\right\}\right),\
{\rm AM}_1\left(\left\{\gamma\in\Gamma(\X)\;\middle|\;\gamma_{b_\gamma}\in E\right\}\right)\bigg\rbrace\leq\mm(E).
\end{equation}
\item\label{subcurves} If $\Gamma,\tilde{\Gamma}\subseteq\mathscr R(\X)$ and
any $\gamma\in\Gamma$ contains a subcurve $\tilde\gamma\in\tilde\Gamma$, then ${\rm AM}_0(\Gamma)\leq {\rm AM}_0(\tilde\Gamma)$.
The same holds for ${\rm AM}_1$ if the subcurve $\tilde\gamma$ has the same endpoints of $\gamma$.
\item\label{Fuglede2}
Given any $\Gamma\subseteq\Gamma(\X)$, we have that ${\rm AM}_0(\Gamma)=0$ if and only if there is a sequence of Borel functions $\rho_i\colon\X\to[0,\infty]$
such that $\sup_i\int\rho_i\,\d\mm<\infty$  but $\varliminf_i\int_\gamma\rho_i=\infty$ for each $\gamma\in\Gamma\cap\mathscr R(\X)$.
\end{enumerate}
\end{proposition}
\begin{proof} \ref{it:AM_outer_meas} Fix \(\lambda\in\{0,1\}\). Taking \(\rho_i\coloneqq 0\) for all \(i\in\N\), we have
\((\rho_i)_i\in\mathcal {AM}_\lambda(\varnothing)\), which gives \eqref{eq:AM_zero}. Given any
\(\Gamma\subseteq\tilde\Gamma\subseteq\Gamma(\X)\), we have \(\mathcal {AM}_\lambda(\tilde\Gamma)\subseteq\mathcal {AM}_\lambda(\Gamma)\) and thus
\eqref{eq:AM_monot} holds. To prove that \({\rm AM}_\lambda\) is an outer measure, 
it only remains to check its \(\sigma\)-subadditivity \eqref{eq:AM_subadd}.
Without loss of generality, we can assume that \(\sum_{n=1}^\infty{\rm AM}_\lambda(\Gamma_n)<\infty\).
Fix any \(\varepsilon>0\). For any \(n\in\N\), we can find a sequence \((\rho^n_i)_i\subseteq\mathcal {AM}_\lambda(\Gamma_n)\)
such that \(\sup_i\int\rho^n_i\,\d\mm\leq{\rm AM}_\lambda(\Gamma_n)+\varepsilon 2^{-n}\) (since \(\mathcal {AM}_\lambda(\Gamma_n)\)
is closed under passing to subsequences). Define \(\rho_i\coloneqq\sum_{n=1}^\infty\rho^n_i\) for every \(i\in\N\). Note that
\((\rho_i)_i\in\mathcal {AM}_\lambda\big(\bigcup_{n\in\N}\Gamma_n\big)\). Moreover, by the monotone convergence theorem we get
\begin{align*}
{\rm AM}_\lambda\bigg(\bigcup_{n=1}^\infty\Gamma_n\bigg)&\leq\limi_{i\to\infty}\int\rho_i\,\d\mm
=\limi_{i\to\infty}\sum_{n=1}^\infty\int_{\X}\rho^n_i\,\d\mm
\leq\limi_{i\to\infty}\sum_{n=1}^\infty({\rm AM}_\lambda(\Gamma_n)+\varepsilon 2^{-n})\\
&\leq\varepsilon+\sum_{n=1}^\infty{\rm AM}_\lambda(\Gamma_n).
\end{align*}
Letting \(\varepsilon\searrow 0\) we conclude that \eqref{eq:AM_subadd} holds, thus accordingly
\({\rm AM}_\lambda\) is an outer measure.\\
\ref{it:ineq_AM} It trivially follows from the observation that \(\mathcal {AM}_0(\Gamma)\subseteq\mathcal {AM}_1(\Gamma)\).\\
\ref{it:wlog_const_speed_AM}
Since \(\int_\gamma\rho_i=\int_{\hat\gamma}\rho_i\), and \(\big(\rho_i(\gamma_{a_\gamma}),\rho_i(\gamma_{b_\gamma})\big)
=\big(\rho_i(\hat\gamma_0),\rho_i(\hat\gamma_{\ell(\gamma)})\big)\) for every \(\gamma\in\Gamma\cap\mathscr R(\X)\) and \(i\in\N\),
we have that \(\mathcal {AM}_\lambda(\Gamma\cap\mathscr R(\X))=\mathcal {AM}_\lambda(\{\hat\gamma:\gamma\in\Gamma\cap\mathscr R(\X)\})\).\\
\ref{it:ineq_AM_mm} Letting \(\rho^E_i\coloneqq\nchi_E\) for every \(i\in\N\), we clearly have
\((\rho_i)_i\in\mathcal {AM}_1(\{\gamma\in\Gamma(\X):\gamma_{a_\gamma}\in E\})\), thus
\[
{\rm AM}_1\left(\left\{\gamma\in\Gamma(\X)\;\middle|\;\gamma_{a_\gamma}\in E\right\}\right)
\leq\limi_{i\to\infty}\int_{\X}\rho^E_i\,\d\mm=\mm(E).
\]
Similarly for \({\rm AM}_1(\{\gamma\in\Gamma(\X):\gamma_{b_\gamma}\in E\})\).\\
\ref{subcurves} It is a simple consequence of the fact that, by monotonicity, any sequence $(\rho_i)_i$
admissible for $\tilde\gamma$ is admissible for $\gamma$, with the additional constraint that $\tilde\gamma$
must have the same endpoints of $\gamma$ when $\lambda=1$.\\
\ref{Fuglede2} If there is a sequence $(\rho_i)_i$ satisfying the second criterion stated in (vi),
then for each $\eps>0$ we know that the sequence $(g_i)_i$ given by $g_i=\eps\rho_i$ satisfies $(g_i)_i\in {\mathcal{AM}}_0(\Gamma)$,
and so  ${\rm Mod}_0(\Gamma)\le \eps\, \varliminf_i\int\rho_i\,\d\mm$. Letting $\eps\to 0^+$ tells us that ${\rm Mod}_0(\Gamma)=0$.
Now suppose that $\Gamma\subseteq\Gamma(\X)$ with ${\rm AM}_0(\Gamma)=0$. 
We now use the fact that subsequences of sequences in ${\mathcal{AM}}_0(\Gamma)$ also belong to ${\mathcal{AM}}_0(\Gamma)$.
Then for each positive integer $k$ we can find a sequence
$(\rho_{k,i})_i\in {\mathcal{AM}}_0(\Gamma)$ such that $\sup_i\int\rho_{k,i}\,\d\mm< 2^{-k}$. Now the sequence
$(\rho_i)_i$ given by $\rho_i\coloneqq\sum_{k=1}^\infty \rho_{k,i}$ satisfies the second criterion stated in (vi),
completing the proof.
\end{proof}

The following simple example shows that the inequality ${\rm AM}_1\leq {\rm AM}_0$ can be strict.
\begin{example}\label{ex:AM_segment}{\rm
As a metric measure space, consider \(\R\) with the Euclidean distance and the Lebesgue measure. Let \(a,b\in\R\) with \(a<b\) be given.
Let \(\sigma\colon[0,b-a]\to\R\) denote the arc-length parameterisation of the interval \([a,b]\),
i.e.\ \(\sigma_t=a+t\) for every \(t\in[0,b-a]\). Then
\[
{\rm AM}_0(\{\sigma\})=1,\qquad{\rm AM}_1(\{\sigma\})=0.
\]
Indeed, for any \({\rm AM}_0\)-admissible sequence \((\rho_i)_i\) for \(\{\sigma\}\) we have that
\[
\varliminf_{i\to\infty}\int_\R\rho_i(t)\,\d t\geq\varliminf_{i\to\infty}\int_\sigma\rho_i\geq 1,
\]
thus \({\rm AM}_0(\{\sigma\})\geq 1\). Choosing the constant sequence \(((b-a)^{-1}\nchi_{[a,b]})_i\) as an \({\rm AM}_0\)-admissible sequence
for \(\{\sigma\}\), we then see that \({\rm AM}_0(\{\sigma\})=1\). Moreover, the constant sequence \((\nchi_{\{a,b\}}/2)_i\) is \({\rm AM}_1\)-admissible
for \(\{\sigma\}\), thus \({\rm AM}_1(\{\sigma\})\leq\frac{1}{2}\int_\R\nchi_{\{a,b\}}(t)\,\d t=0\).\fr}
\end{example}
Given any rectifiable curve \(\gamma\in\mathscr R(\X)\) in a metric space \((\X,\sfd)\) and \(\lambda\in\{0,1\}\),
we define the finite Borel measure \(\theta^\lambda_\gamma\) on the interval \([0,\ell(\gamma)]\subseteq\R\) as
\[
\theta^\lambda_\gamma\coloneqq\lambda(\delta_0+\delta_{\ell(\gamma)})+\mathscr L^1\llcorner(0,\ell(\gamma)).
\]
\begin{corollary}\label{cor:exceptional_AM}
Let \((\X,\sfd,\mm)\) be a metric measure space and \(\lambda\in\{0,1\}\). Let \(N\subseteq\X\) be a Borel set
such that \(\mm(N)=0\). Then it holds that
\[
\theta^\lambda_\gamma\big(\big\{t\in[0,\ell(\gamma)]\;\big|\;\hat\gamma_t\in N\big\}\big)=0
\quad\text{ for }{\rm AM}_\lambda\text{-a.e.\ }\gamma\in\mathscr R(\X).
\]
In particular, for \({\rm AM}_\lambda\)-a.e.\ curve \(\gamma\in\mathscr R(\X)\) it holds that
\[
\gamma\big([a,b]\cap\gamma^{-1}(\X\setminus N)\big)\text{ is dense in }\gamma([a,b])
\quad\text{ for every }a,b\in[a_\gamma,b_\gamma]\text{ with }a<b.
\]
\end{corollary}

\begin{proof} 
The first part of the statement is a direct consequence of 
a combination of
Lemma~\ref{Nages1} (because
${\rm AM}_1\leq{\rm AM}_0\le \Mod$) and~\eqref{eq:ineq_AM_mm} with the choice of $E=N$.
To prove the last part of the statement, fix a curve \(\gamma\in\mathscr R(\X)\) such that
$\theta_\gamma^\lambda(\{t\in [0,\ell(\gamma)]\, :\, \hat{\gamma}_t\in N\})=0$; then
\(\mathscr L^1(\{t\in (0,\ell(\gamma)):\hat\gamma_t\in N\})=0\). For $a,b\in[a_\gamma,b_\gamma]$ with
$a_\gamma\le a<b\le b_\gamma$, letting
\(\hat a\coloneqq\ell(\gamma_{|[a_\gamma,a]})\) and \(\hat b\coloneqq\ell(\gamma_{|[a_\gamma,b]})\),
we have that \(\gamma_a=\hat\gamma_{\hat a}\) and \(\gamma_b=\hat\gamma_{\hat b}\). Since
\(\mathscr L^1([\hat a,\hat b]\cap\hat\gamma^{-1}(N))=0\), we deduce that \([\hat a,\hat b]\cap\hat\gamma^{-1}(\X\setminus N)\)
is dense in \([\hat a,\hat b]\). In view of the continuity of \(\hat\gamma\), we finally conclude that
\(\gamma([a,b]\cap\gamma^{-1}(\X\setminus N))=\hat\gamma([\hat a,\hat b]\cap\hat\gamma^{-1}(\X\setminus N))\)
is dense in \(\hat\gamma([\hat a,\hat b])=\gamma([a,b])\).
\end{proof}
\subsection{Test plans}\label{SubSec:TestPlans}
Now we recall the concept of test plans, first given in~\cite{Inventiones}, somewhat
reminiscent of Semmes' pencil of curves, introduced in \cite{Amb:DiMa:14}, but
adapting it to the limiting case $p=\infty$. 
Following the nomenclature of~\cite{Inventiones},
this should be called an $\infty$-test plan, but as we are not interested in $p$-test plans for other values of $p$, we will simply
refer to them as test plans.
\begin{definition}[Test plan]\label{def:test_plan}
Let \((\X,\sfd,\mm)\) be a metric measure space. Then a Radon probability measure
\(\ppi\) on \(C([0,1];\X)\) is called a \emph{test plan} on \(\X\) if the following conditions hold:
\begin{enumerate}[label=(\roman*)]
\item \textsc{Bounded compressibility.} There exists a constant \(C>0\) such that
\begin{equation}\label{eq:compr_const}
(\e_t)_\#\ppi\leq C\mm\quad\text{ for every }t\in[0,1].
\end{equation}
\item\label{it:Lip_plan} \textsc{Lipschitzianity.} There exists a constant \(L\geq 0\) such that
\begin{equation}\label{eq:Lip_const_plan}
\Lip(\gamma)\leq L\quad\text{ for }\ppi\text{-a.e.\ }\gamma\in C([0,1];\X).
\end{equation}
\end{enumerate}
The minimal constant $C$ satisfying~\eqref{eq:compr_const} is 
denoted \({\rm Comp}(\ppi)>0\) and is called the \emph{compression constant}
of \(\ppi\), while the minimal constant $L$ satisfying \eqref{eq:Lip_const_plan} is 
denoted \(\Lip(\ppi)\geq 0\) and is called the
\emph{Lipschitz constant} of \(\ppi\). We denote by \(\Pi(\X)\) the collection of all test plans on \(\X\).
\end{definition}

Given the requirement of absolute continuity as described in~\eqref{eq:compr_const}, the constant 
${\rm Comp}(\ppi)$ must be positive, for if not, then $\ppi$ cannot be a probability measure.
Note that Definition \ref{def:test_plan}\ref{it:Lip_plan} implies that \(\ppi\) is concentrated on \(\LIP([0,1];\X)\),
and the quantity \(\Lip(\ppi)\) coincides with the \(L^\infty\)-norm in \((\LIP([0,1];\X),\ppi)\) of
\(\LIP([0,1];\X)\ni\gamma\mapsto\Lip(\gamma)\).

In the following proposition we state a weak duality between $\text{AM}_0$-modulus
and test plans, see \cite{ADS} and \cite{Plans_modulus_duality}
for full duality statements between modulus, $\text{AM}_0$-modulus and plans.
\begin{proposition}[Relation between \({\rm AM}_\lambda\) and test plans]\label{prop:AM_vs_tp}
Let \((\X,\sfd,\mm)\) be a metric measure space and \(\lambda\in\{0,1\}\). Fix any \(\Gamma\subseteq C([0,1];\X)\).
Then there exists a Borel set \(\tilde\Gamma\subseteq C([0,1];\X)\) such that \(\Gamma\subseteq\tilde\Gamma\) and
\begin{equation}\label{eq:tildeAM_vs_tp}
\ppi(\tilde\Gamma)\leq{\rm Comp}(\ppi)(\Lip(\ppi)+2\lambda){\rm AM}_\lambda(\Gamma)\quad\text{ for every }\ppi\in\Pi(\X).
\end{equation}
In particular, if a set 
\(\mathcal N\subseteq C([0,1];\X)\) satisfies \({\rm AM}_\lambda(\mathcal N)=0\), then there
exists a Borel set \(\tilde{\mathcal N}\subseteq C([0,1];\X)\) such that \(\mathcal N\subseteq\tilde{\mathcal N}\) and
\[
\ppi(\tilde{\mathcal N})=0\quad\text{ for every }\ppi\in\Pi(\X).
\]
\end{proposition}

\begin{proof}
Given any \(n\in\N\), we take an \({\rm AM}_\lambda\)-admissible sequence \((\rho_i^n)_i\) for \(\Gamma\) such that
\[
\limi_{i\to\infty}\int_{\X}\rho_i^n\,\d\mm\leq{\rm AM}_\lambda(\Gamma)+\frac{1}{n}.
\]
Thanks to the Vitali--Carath\'{e}odory theorem, given a Borel function $\rho$ with
$\int_{\X} \rho\, \d\mm$ finite, we can find sequences of
lower semicontinuous functions that are larger than $\rho$ pointwise and approximate $\rho$ both pointwise $\mm$-a.e.\ and
in $L^1(\mm)$; hence we can assume that the functions $\rho_i^n$ are lower semicontinuous. It follows that the maps
$\mathscr{R}(X)\ni \gamma\mapsto \int_\gamma \rho_i^n$ and 
$\mathscr{R}(X)\ni \gamma\mapsto \rho_i^n\circ e_1(\gamma)+\rho_i^n\circ e_0(\gamma)$ are Borel maps, see 
also~\cite[Corollary~2.23]{AmbrosioIkonenPasqualetto}, where a stronger statement, that for every non-negative Borel function $g$
the map $\gamma\mapsto\int_\gamma g$ is Borel, is established using the above-mentioned lower semicontinuous approximations.
We then define the family of curves \(\Gamma_n\) as
\[
\Gamma_n\coloneqq\left\{\gamma\in C([0,1];\X)\;\middle|\;
\limi_{i\to\infty}\left(\lambda\big(\rho_i^n(\gamma_0)+\rho_i^n(\gamma_1)\big)+\int_\gamma\rho_i^n\right)\geq 1\right\}.
\]
Note that each set \(\Gamma_n\) is Borel and $\Gamma\subseteq\Gamma_n$, 
and so \(\tilde\Gamma\coloneqq\bigcap_{n\in\N}\Gamma_n\subseteq C([0,1];\X)\)
is Borel as well with $\Gamma\subseteq\tilde\Gamma$.
Making use of the inclusion \(\Gamma\subseteq\tilde\Gamma\) and of Fatou's lemma, for any \(\ppi\in\Pi(\X)\)
and \(n\in\N\) we obtain that
\begin{align*}
\ppi(\tilde\Gamma)&\leq\ppi(\Gamma_n)=\int_{C([0,1];\X)}\nchi_{\Gamma_n}\,\d\ppi
\leq\int_{C([0,1];\X)}\,\limi_{i\to\infty}\left(\lambda\big(\rho_i^n(\gamma_0)
+\rho_i^n(\gamma_1)\big)+\int_\gamma\rho_i^n\right)\,\d\ppi(\gamma)\\
&\leq\limi_{i\to\infty}\lambda\left(\int_{C([0,1];\X)}\rho_i^n\circ\e_0\,\d\ppi
+\int_{C([0,1];\X)}\rho_i^n\circ\e_1\,\d\ppi\right)+\int_{C([0,1];\X)}\,
\int_0^1\rho_i^n(\gamma_t)|\dot\gamma_t|\,\d t\,\d\ppi(\gamma)\\
&\leq\limi_{i\to\infty}\lambda\left(\int_{C([0,1];\X)}\rho_i^n\,\d(\e_0)_\#\ppi
+\lambda\int_{C([0,1];\X)}\rho_i^n\,\d(\e_1)_\#\ppi\right)+\Lip(\ppi)\int_0^1\!\!\!\int_{C([0,1];\X)}\rho_i^n\,\d(\e_t)_\#\ppi\,\d t\\
&\leq{\rm Comp}(\ppi)\big(2\lambda+\Lip(\ppi)\big)\limi_{i\to\infty}\int_{\X}\rho_i^n\,\d\mm\\
&\leq{\rm Comp}(\ppi)\big(2\lambda+\Lip(\ppi)\big)\left({\rm AM}_\lambda(\Gamma)+\frac{1}{n}\right),
\end{align*}
which gives \eqref{eq:tildeAM_vs_tp} by letting \(n\to\infty\). The last part of the claim immediately follows.
\end{proof}

Given any \(r,s\in[0,1]\) with \(r<s\), we define the map \({\rm restr}_r^s\colon C([0,1];\X)\to C([0,1];\X)\) as
\[
{\rm restr}_r^s(\gamma)_t\coloneqq\gamma_{(1-t)r+ts}\quad\text{ for every }\gamma\in C([0,1];\X)\text{ and }t\in[0,1].
\]
Note that \({\rm restr}_r^s\) is \(1\)-Lipschitz. Moreover, it is easy to check that for any \(\ppi\in\Pi(\X)\) we have
\begin{equation}\label{eq:restr_tp}
\ppi_r^s\coloneqq({\rm restr}_r^s)_\#\ppi\in\Pi(\X),\quad{\rm Comp}(\ppi_r^s)\leq{\rm Comp}(\ppi),
\quad\Lip(\ppi_r^s)\leq(s-r)\Lip(\ppi).
\end{equation}
\begin{remark}\label{rmk:tp_conc_on_spt}{\rm
Let \((\X,\sfd,\mm)\) be a metric measure space and \(\ppi\in\Pi(\X)\). Then we claim that
\[
\gamma([0,1])\subseteq{\rm spt}(\mm)\quad\text{ for }\ppi\text{-a.e.\ }\gamma\in C([0,1];\X).
\]
To see this, we argue as follows. By Corollay~\ref{cor:exceptional_AM} with the choice
of $N=X\setminus{\rm spt}(\mm)$, we know that $\AM_\lambda$-almost every $\gamma$ lies entirely inside ${\rm spt}(\mm)$ because the
support of $\mm$ is closed and $\mm(N)=0$. Now the claim follows from 
Proposition~\ref{prop:AM_vs_tp}. By Proposition~\ref{prop:AM_vs_tp} we know that
$\AM_\lambda$-negligible collections are $\ppi$-negligible.
\fr}\end{remark}
\subsection{Newtonian spaces \texorpdfstring{$\widetilde{N}^{1,1}(\X,\sfd,\mm)$}{tildeN11}, \texorpdfstring{$N^{1,1}(\X,\sfd,\mm)$}{N11} and upper gradients}

Given an $\mm$-measurable function $u:\X\to[-\infty,\infty]$ with $\int |u|\,\d\mm<\infty$, 
we say that $u\in \widetilde{N}^{1,1}(\X,\sfd,\mm)$ if there is a nonnegative Borel function $g\colon \X\to [0,\infty]$
with $\int_{\X} g \,\d\mm<\infty$ such that 
\begin{equation}\label{eq:upper-grad}
|u(\gamma_{b_\gamma}))-u(\gamma_{a_\gamma})|\le \int_\gamma g\quad\text{ for every }\gamma\in \mathscr R(\X).
\end{equation}
Since subcurves of $\gamma\in \mathscr R(\X)$ also belong to $\mathscr R(\X)$, we can also state the above condition as
\[
|u(\gamma_t)-u(\gamma_s)|\le \int_{\gamma_{|[s,t]}}\hskip -5pt g
\]
whenever $s,\,t\in[a_\gamma,b_\gamma]$ with $s\leq t$. Such a function $g$ is said to be an \emph{upper gradient} of $u$.
It is well known that $\lip_a$ (or even smaller notions of pseudo-modulus of the gradient, as the local Lipschitz
constant in \eqref{eq:def_slope}, see also Appendix~\ref{sec:genrel}) is an upper gradient for locally Lipschitz functions $u$,
see for instance~\cite[Lemma~6.2.6]{HKST}.
The above condition also should be interpreted to mean that $\int_\gamma g=\infty$ when
at least one of $u(\gamma_t)$ and $u(\gamma_s)$ belongs to the $\mm$-negligible set $N\coloneqq\{x\in\X\, :\, |u(x)|=\infty\}$.
This requirement avoids the issue of interpreting $\infty-\infty$  and, at the same time, ensures together with 
Lemma~\ref{lem:Fuglede1} that $\Mod(\Gamma_N)=0$ for any $u\in\tilde{N}^{1,1}(\X,\sfd,\mm)$, where we recall that $\Gamma_N$
stands for the collection of rectifiable curves that intersect $N$. 

Although the property of being in 
$\widetilde{N}^{1,1}(\X,\sfd,\mm)$ is not stable
under modification of the function on $\mm$-negligible sets,
thanks to Lemma~\ref{lem:Fuglede1},
we can modify the function on $\mm$-negligible sets $N$ for which $\Mod(\Gamma_N)=0$. See
Lemma~\ref{lem:exceptional} below for more on this.

For functions $u\in\widetilde{N}^{1,1}(\X,\sfd,\mm)$ we set
\[
\|u\|_{N^{1,1}(\X,\sfd,\mm)}:=\int_{\X}|u|\, \d\mm+\inf_g\int_{\X} g\,\d\mm,
\]
where the infimum is over all upper gradients $g$ of $u$. 
The collection $\widetilde{N}^{1,1}(\X,\sfd,\mm)$ is a vector space, and $\|\cdot\|_{N^{1,1}(\X,\sfd,\mm)}$ is a seminorm on this collection.
We set
\[
N^{1,1}(\X,\sfd,\mm)\coloneqq\widetilde{N}^{1,1}(\X,\sfd,\mm)/\hskip -3pt\sim,
\]
where $\sim$ is the 
equivalence relation given by $u\sim v$ when $\|u-v\|_{N^{1,1}(\X,\sfd,\mm)}=0$. The following lemma gives us a way to understand
the  equivalence relation $\sim$ (see for instance \cite{HKST}).
\begin{lemma}\label{lem:exceptional}
Suppose that $u\in \widetilde{N}^{1,1}(\X,\sfd,\mm)$ and $v:X\to[-\infty,\infty]$ is $\mm$-measurable. Then 
$v\in \widetilde{N}^{1,1}(\X,\sfd,\mm)$ with $v\sim u$ if only if the set $K\coloneqq\{u\neq v\}$ is $\mm$-negligible and
$\Mod(\Gamma_K)=0$.
\end{lemma}

Another useful property of functions in $\widetilde{N}^{1,1}(\X,\sfd,\mm)$ is the following.

\begin{theorem}\label{thm:identification_gf}
For any $f\in\widetilde{N}^{1,1}(\X,\sfd,\mm)$
there is a unique, up to $\mm$-negligible sets,
Borel function $g_u:\X\to[0,\infty]$ such that the 
collection of all non-constant compact rectifiable curves in $X$ for which the
upper gradient inequality~\eqref{eq:upper-grad} 
with $g=g_u$ fails is $\Mod$-negligible, and 
\[
\|u\|_{N^{1,1}(\X,\sfd,\mm)}=\int_{\X}|u|\, \d\mm+\int_{\X} g_u\, \d\mm.
\]
Furthermore, $g_u\leq g$ $\mm$-a.e. for any upper gradient $g$ of $u$.
\end{theorem}

The function $g_u$ is called the \emph{minimal weak upper gradient} of $u$. In this context, the existence of such a function was
established in \cite[Theorem~7.16]{Hajlasz}. 

In the sequel, when there is no ambiguity we use also the simpler notation $\widetilde{N}^{1,1}(\X)$, $N^{1,1}(\X)$.
\section{Notions of metric BV functions}
\subsection{BV functions via relaxation}

\begin{definition}[The space \(BV(\X)\)]
Let \((\X,\sfd,\mm)\) be a metric measure space and \(f\in L^1(\X)\). Given any open
set \(\Omega\subseteq\X\), we define \(|Df|(\Omega)\) by
\begin{equation}\label{eq:def_|Df|}
|Df|(\Omega)\coloneqq\inf\left\{\limi_{n\to\infty}\int_\Omega\lip_a(f_n)\,\d\mm\;\middle|
\;(f_n)_n\subseteq\LIP_{b,loc}(\Omega)\cap\mathcal L^1(\Omega),\,f_n\to f\text{ in }L^1(\Omega)\right\}.
\end{equation}
Then we declare that \(f\) belongs to \(BV(\X)\), provided \(|Df|(\X)<\infty\).
\end{definition}

\begin{remark}
{\rm Note that if $A\subseteq \X$ is not an open set, then $|Df|(A)$ is not defined using approximations $f_n\in \LIP_{b,loc}(A)$
unless $A$ is seen as a metric measure space in its own right and $f$ is in $BV(A)$ but is not expected to be in
$BV(\X)$. Given this ambiguity, in situations where we wish to emphasize that $f\in BV(A)$, we represent its BV-gradient by
$|D_Af|$.

Moreover, the notation $BV(\X)$ is itself ambiguous, as the $BV$ property depends on the
full metric measure structure. So, the
correct notation should be $BV(\X,\sfd,\mm)$, as we did for the
Newtonian spaces, and we use $BV(\X)$ and $|Df|(\X)$ for notational
simplicity. When we want to compare
the $BV$ property and the total variation with respect to different
metric measure structures we emphasize the full
dependence, as in Lemma~\ref{lem:localisation_BV_star}. Thus, the reference to $\X$ implicitly assumes
the metric measure space $(\X,\sfd,\mm)$.
\fr}\end{remark}

A non-negative Borel function $g$ on a metric measure space $(Z,\sfd_Z,\mu_Z)$ induces a weighted measure $\nu_g$ on $Z$
given by $\d\nu_g=g\, \d\mu_Z$. We will use the more descriptive notation $g\, \mu_Z$ to denote this weighted measure $\nu_g$.

Given a function \(f\in L^1(\X)\), arguing as in \cite{Miranda} (see also the proof of
Theorem~\ref{thm:cor_of_DeGiorgi-Letta}), it can be proved that the set-function \(|Df|\colon\tau(\X)\to[0,\infty]\)
defined in \eqref{eq:def_|Df|} can be uniquely extended to an outer regular Borel measure on \(\X\). We denote this
Borel measure also
by \(|Df|\) and we call this measure the \emph{total variation measure} of \(f\). If a 
sequence \((f_n)_n\subseteq\LIP_{b,loc}(\X)\cap L^1(\X)\) is chosen so that \(f_n\to f\) in \(L^1(\X)\) and
\(\int_{\X}\lip_a(f_n)\,\d\mm\to|Df|(\X)\), then we have that
\begin{equation}\label{eq:weak_conv_to_|Df|}
\lip_a(f_n)\mm\rightharpoonup|Df|\;\text{ weakly,}\quad\text{i.e.\ }\int_{\X} h\,\lip_a(f_n)\,\d\mm\to\int_{\X} h\,\d|Df|
\text{ for every }h\in C_b(\X).
\end{equation}
Indeed, we have that \(|Df|(U)\leq\limi_n(\lip_a(f_n)\mm)(U)=\limi_n\int_U\lip_a(f_n)\,\d\mm\) for every open set \(U\subseteq\X\),
and \((\lip_a(f_n)\mm)(\X)\to|Df|(\X)\), whence \eqref{eq:weak_conv_to_|Df|} follows by the Portmanteau theorem.
\begin{definition}[The space \(BV_N(\X)\)]\label{BV_Newtonian}
Let \((\X,\sfd,\mm)\) be a metric measure space and \(f\in L^1(\X)\). Given any open
set \(\Omega\subseteq\X\), we define \(|Df|_N(\Omega)=|D_\X f|_N(\Omega)\) as
\begin{equation}\label{eq:def_|Df|_N}
|Df|_N(\Omega)\coloneqq\inf\left\{\limi_{n\to\infty}\int_\Omega g_{f_n}\,\d\mm\;\middle|
\;(f_n)_n\subseteq N^{1,1}(\Omega),\,f_n\to f\text{ in }L^1(\Omega)\right\},
\end{equation}
where $g_{f_n}$ is the minimal weak upper gradient of $f_n$ in $\Omega$ in the sense of
Theorem~\ref{thm:identification_gf}.
Then we declare that \(f\) belongs to \(BV_N(\X)\), provided \(|Df|_N(\X)<\infty\).
\end{definition}

As in the case of $BV(\X)$, for $f\in L^1(\X)$ the above 
construction provides an outer regular Borel measure in $\X$,
denoted $|Df|_N$. Furthermore, since $\lip_a$ is an upper gradient for locally Lipschitz
functions, it is immediately seen that 
\begin{equation}\label{Mir_in_New}
BV(\X)\subseteq BV_N(\X)\qquad
\text{with}\qquad \text{$|Df|_N(\X)\leq|Df|(\X)$ for all $f\in BV(\X)$.}
\end{equation}
\begin{remark}[Locality properties] \label{rmk:local_property_BV}{\rm
Let \(\Omega\subseteq\X\) be open and \(f\in L^1(\X)\) be given. Note that both
\((\Omega,\sfd_{|\Omega\times\Omega},\mm_{|\mathscr B(\Omega)})\) and \(\X_\Omega\coloneqq(\X,\sfd,\mm\llcorner\Omega)\)
are metric measure spaces. We also have \(f_{|\Omega}\in L^1(\Omega)\) and \(f\in L^1(\X_\Omega)\).
It is easy to deduce from the definition that \(|D_\X f|(\Omega)=|D_\Omega f_{|\Omega}|(\Omega)=|D_{\X_\Omega}f|(\X)\).
Since the notion of minimal weak upper gradient is local on open sets, the same remarks above apply also
to the spaces $BV_N(X)$ of Definition~\ref{BV_Newtonian}.

In addition, given a Borel set \(E\subseteq\X\) and a function \(f\in BV(\X)\), it is easy to check that
\begin{equation}\label{eq:restr_BV}
f_{|E}\in BV(E),\qquad|Df_{|E}|(E)\leq|Df|(\X).
\end{equation}
Note however that the above inequality can be strict if $E$ is not an open set.
\fr}
\end{remark}
\begin{lemma}[A Leibniz formula for \(BV(\X)\)]
Let \((\X,\sfd,\mm)\) be a metric measure space. Fix \(f\in BV(\X)\) and \(\eta\in\LIP_b(\X)\).
Then it holds that \(\eta f\in BV(\X)\) and
\begin{equation}\label{eq:Leibniz_BV}
|D(\eta f)|\leq|\eta|\, |Df|+|f|\lip_a(\eta)\, \mm.
\end{equation}
\end{lemma}
\begin{proof}
Take \((f_n)_n\subseteq\LIP_{b,loc}(\X)\cap L^1(\X)\) with \(f_n\to f\) in \(L^1(\X)\)
and \(\int_{\X}\lip_a(f_n)\,\d\mm\to|Df|(\X)\). Recall
that \(\lip_a(f_n)\mm\rightharpoonup|Df|\) by~\eqref{eq:weak_conv_to_|Df|}.
By~\eqref{eq:Leibniz_lip_a}, we have that
\((\eta f_n)_n\subseteq\LIP_{b,loc}(\X)\cap L^1(\X)\) and
\(\lip_a(\eta f_n)\leq|\eta|\lip_a(f_n)+|f_n|\lip_a(\eta)\) for every \(n\in\N\). 

Next, fix an open set \(U\subseteq\X\).
Given that the measure
\(\mu\coloneqq|\eta|\, |Df|+|f|\lip_a(\eta)\, \mm\) is a finite measure, for any \(k\in\N\)
we can find \(\varepsilon_k\in(0,1/k)\) such that \(\mu(\partial U_k)=0\), where the open set \(U_k\) is given by
\(U_k\coloneqq\{x\in U:\ \sfd(x,\X\setminus U)>\varepsilon_k\big\}\). Note that \(U=\bigcup_{k\in\N}U_k\). Now, fix \(k\in\N\).
We can find a sequence \((\psi_j)_j\) of continuous functions \(\psi_j\colon\X\to[0,1]\) satisfying
\(\psi_j(x)\searrow\nchi_{\bar U_k}(x)\) for all \(x\in\X\). Since \(\eta f_n\to\eta f\) in \(L^1(\X)\), we have
\begin{align*}
|D(\eta f)|(U_k)&\leq\limi_{n\to\infty}\int_U\lip_a(\eta f_n)\,\d\mm\leq\limi_{n\to\infty}\int_{\X}\psi_j\lip_a(\eta f_n)\,\d\mm\\
&\leq\lim_{n\to\infty}\int_{\X}\psi_j|\eta|\lip_a(f_n)\,\d\mm+\lim_{n\to\infty}\int_{\X}\psi_j|f_n|\lip_a(\eta)\,\d\mm\\
&=\int_{\X}\psi_j|\eta|\,\d|Df|+\int_{\X}\psi_j|f|\lip_a(\eta)\,\d\mm=\int_{\X}\psi_j\,\d\mu\quad\text{ for every }j\in\N.
\end{align*}
By the dominated convergence theorem, we deduce that
\[
|D(\eta f)|(U_k)\leq\lim_{j\to\infty}\int_{\X}\psi_j\,\d\mu=\mu(\bar U_k)=\mu(U_k)\leq\mu(U)\quad\text{ for every }k\in\N,
\]
whence it follows that \(|D(\eta f)|(U)=\sup_{k\in\N}|D(\eta f)|(U_k)\leq\mu(U)\). Thanks to the outer regularity of
\(|D(\eta f)|\) and \(\mu\), we conclude that \(|D(\eta f)|\leq\mu\). In particular, \(\eta f\in BV(\X)\).
\end{proof}

In order to compare the space $BV_N(\X)$
with other $BV$-spaces, the following result will be useful.

\begin{lemma}\label{lem:BV_in_BV_tildeam}
Let \((\X,\sfd,\mm)\) be a metric measure space, \(\Omega\subseteq\X\) be open,
and \(f\in \mathcal L^1(\Omega)\) with \([f]_\mm\in BV_N(\Omega)\). Then there exists a 
sequence \((f_n)_n\subseteq N^{1,1}(\Omega)\) and a Borel set \(N\subseteq\Omega\) with
\(\mm(N)=0\) such that \(\int_\Omega g_{f_n}\,\d\mm\to|\D f|_N(\Omega)\) and
\begin{equation}\label{eq:compitino}
|f(\gamma_{b_\gamma})-f(\gamma_{a_\gamma})|\leq\varliminf_{n\to\infty}\int_\gamma g_{f_n}
\quad\text{ if }\gamma\in\mathscr R(\X)\text{, }\gamma([a_\gamma,b_\gamma])\subseteq\Omega\text{ and }
\gamma_{a_\gamma},\gamma_{b_\gamma}\notin N.
\end{equation}
\end{lemma}
\begin{proof} 
Note that the condition $\gamma\in\mathscr R(\X)$ together with $\gamma([a_\gamma,b_\gamma])\subseteq\Omega$
is equivalent to the condition that $\gamma\in\mathscr R(\Omega)$, and so we will continue to use the simpler notation
$\gamma\in\mathscr R(\Omega)$.

Take any \((f_n)_n\subseteq\tilde{N}^{1,1}(\Omega)\) such that \(\int_\Omega|f_n-f|\,\d\mm\to 0\) 
and \(\int_\Omega g_{f_n}\,\d\mm\to|\D f|_N(\Omega)\). Up to passing to a non-relabelled subsequence, we can
assume that there exists an \(\mm\)-null Borel set \(N\subseteq\Omega\) such that \(f(x)=\lim_n f_n(x)\) for every
\(x\in\Omega\setminus N\). Also, we can find a
set $\Gamma\subseteq\mathscr R(\Omega)$ with $\Mod(\Gamma)=0$ such
that the upper gradient property holds for $(f_n,g_{f_n})$ for any 
$\gamma\in\mathscr R(\Omega)\setminus\Gamma$, see Theorem~\ref{thm:identification_gf} above.
Now, fix a curve \(\gamma\in\mathscr R(\Omega)\setminus\Gamma\) satisfying \(\gamma([a_\gamma,b_\gamma])\subseteq\Omega\)
and \(\gamma_{a_\gamma},\gamma_{b_\gamma}\notin N\). For \(n\in\N\),  the upper gradient property of $g_{f_n}$ gives
\(|f_n(\gamma_{a_\gamma})-f_n(\gamma_{b_\gamma})|
\leq\int_\gamma g_{f_n}\).
Therefore, we can conclude that
\[
|f(\gamma_{b_\gamma})-f(\gamma_{a_\gamma})|=\lim_{n\to\infty}|f_n(\gamma_{b_\gamma})-f_n(\gamma_{a_\gamma})|
\leq\varliminf_{n\to\infty}\int_\gamma g_{f_n},
\]
thus proving the statement.
\end{proof}
In the final part of this section we discuss one more relaxed notion of $BV$ function. 
Even though this notion is hard to localize, it provides the smaller space in the chain of inclusions \eqref{chain1}, \eqref{chain2},
\eqref{chain3}, and is equal to
the larger space \(BV_{tp}^\star(\X)\) in complete metric measure spaces.

Given a Borel function $f$ on $\X$, by $[f]_{\mm}$ we mean the class of all functions $f_0$ on $\X$ that agree
with $f$ $\mm$-a.e.~in $\X$. If $f\in\mathcal{L}^1(\X)$, then $[f]_{\mm}\in L^1(X)$. If we have a sequence of functions
$f_n\in\mathcal{L}^1(X)$, by the less cumbersome notation
$f_n\to f$ in $L^1(X)$ we mean that $[f_n]_{\mm}\to[f]_{\mm}$ in $L^1(\X)$, or more specifically,
$\lim_{n\to\infty}\, \int_{\X}|f-f_n|\,\d\mm=0$.

\begin{definition}[The space \(BV^\star(\X)\)]\label{def:BV-star}
Let \((\X,\sfd,\mm)\) be a metric measure space and \(f\in \mathcal{L}^1(\X)\). We define the
quantity \(|Df|^\star(\X)=|D_\X f|^\star(\X)\in[0,\infty]\) as
\[
|Df|^\star(\X)\coloneqq\inf\left\{\limi_{n\to\infty}\int_{\X}\lip_a(f_n)\,\d\mm\;\middle|
\;(f_n)_n\subseteq\LIP_{bs}(\X)\cap \mathcal{L}^1(\X),\,f_n\to f\text{ in } L^1(\X)\right\}.
\]
Then we declare that \([f]_{\mm}\) belongs to \(BV^\star(\X)\) provided \(|Df|^\star(\X)<\infty\). 
\end{definition}
Unlike in the setting of Newton-Sobolev spaces $N^{1,1}(\X)$, here we are allowed to perturb BV functions arbitrarily on
sets of $\mm$-measure zero. Hence, to simplify notation, we will continue to refer to $[f]_{\mm}\in BV(\X)$ also as
$f\in BV(\X)$ without ambiguity.

Given that \(\LIP_{bs}(\X)\subseteq\LIP_{b,loc}(\X)\), it can be readily checked that
\begin{equation}\label{eq:triv_ineq_BV}
|Df|(\X)\leq|Df|^\star(\X)\quad\text{ for every }f\in L^1(\X),
\end{equation} 
in particular \(BV^\star(\X)\subseteq BV(\X)\). The following example shows that the inclusion can be strict
and even that, in general, the definition of \(|Df|^\star\)
cannot be localised in order to obtain a total variation measure (in fact, not even a finitely-additive measure).
\begin{example}\label{ex:different_BV}{\rm
Let us equip the space \(\X\coloneqq\bigcup_{k\in\N}(k-2^{-k},k+2^{-k})\setminus\{k\}\) with the Euclidean distance \(\sfd\)
and the restricted Lebesgue measure \(\mm\). Note that \((\X,\sfd)\) is not complete, but it is locally complete.
Let us consider the function
\[
f\coloneqq\sum_{k\in\N}\nchi_{(k,k+2^{-k})}\in BV(\X).
\]
As it is proved in \cite[Example 3.6]{Pas:Sod:25}, we have that \(|Df|(\X)=0\) 
and \(|Df|^\star(\X)=\infty\). Thus in particular,
\(f\notin BV^\star(\X)\). Furthermore,
letting \(\Omega\coloneqq(1,3/2)\) and \(g\coloneqq\nchi_\Omega\in L^1(\X)\), we have
\(\Omega,\X\setminus\Omega\in\tau(\X)\) with
\(|Dg_{|\Omega}|^\star(\Omega)=|Dg_{|\X\setminus\Omega}|^\star(\X\setminus\Omega)=0\),
whereas \(|Dg|^\star(\X)=1\).
\fr}\end{example}
\begin{lemma}\label{lem:localisation_BV_star}
Let \((\X,\sfd,\mm)\) be a metric measure space and \(f\in \mathcal L^1(\X)\), with $\tilde{f}=[f]_\mm$. Assume that \(\Omega\subseteq\X\)
is an open set such that \(\distd(\{f\neq 0\},\X\setminus\Omega)>0\)
and \(\tilde{f}\in BV^\star(\X_\Omega)\),
where \(\X_\Omega\coloneqq(\X,\sfd,\mm\llcorner\Omega)\). Then it holds that \(\tilde{f}\in BV^\star(\X)\) and
\(|D_\X \tilde{f}|^\star(\X)=|D_{\X_\Omega}\tilde{f}|^\star(\X)\).
\end{lemma}

\begin{proof}
Fix $r>0$ such that $r<\distd(\{f\neq 0\},\X\setminus\Omega)$,
and set \(\eta\coloneqq(r^{-1}\, \distd(\cdot,\X\setminus\Omega))\wedge 1\).
Note that \(\{\eta=1\}\) contains a neighbourhood of \(\{f\neq 0\}\). Take any \((f_n)_n\subseteq\LIP_{bs}(\X)\cap L^1(\X_\Omega)\)
such that \(f_n\to \tilde{f}\) in \( L^1(\X_\Omega)\) and \(\int_{\X}\lip_a(f_n)\,\d(\mm\llcorner\Omega)\to|D_{\X_\Omega}\tilde{f}|^\star(\X)\).
Letting \(g_n\coloneqq\eta f_n\in\LIP_{bs}(\X)\cap L^1(\X)\) for every \(n\in\N\), we clearly have that
\(g_n\to \eta \tilde{f}=\tilde{f}\) in \(L^1(\X)\). Furthermore, \eqref{eq:Leibniz_lip_a} yields
\[
\lip_a(g_n)\leq\eta\,\lip_a(f_n)+|f_n|\lip_a(\eta)\leq\nchi_\Omega\,\lip_a(f_n)+\frac{1}{r}\nchi_{\Omega\cap\{f=0\}}|f_n|\quad\text{ on }\X.
\]
Integrating with respect to \(\mm\) and then letting \(n\to\infty\), we conclude that
\begin{align*}
|D_\X \tilde{f}|^\star(\X)&\leq\limi_{n\to\infty}\int_{\X}\lip_a(g_n)\,\d(\mm\llcorner\Omega)\\
&\leq\lim_{n\to\infty}\int_{\X}\lip_a(f_n)\,\d(\mm\llcorner\Omega)+\frac{1}{r}\lim_{n\to\infty}\int_{\{f=0\}}|f_n|\,\d(\mm\llcorner\Omega)
=|D_{\X_\Omega}\tilde{f}|^\star(\X),
\end{align*}
thus in particular \(\tilde{f}\in BV^\star(\X)\). Finally, the converse inequality
\(|D_{\X_\Omega}\tilde{f}|^\star(\X)\leq|D_\X \tilde{f}|^\star(\X)\) easily follows from the fact that \(\mm\llcorner\Omega\leq\mm\).
\end{proof}
\begin{theorem}[Invariance properties of \(BV(\X)\) and \(BV^\star(\X)\)]\label{thm:invar_BV}
Let \((\X,\sfd,\mm)\) be a metric measure space and let \(f\in L^1(\X)\). Then \(f\in BV(\X)\) if
and only if \(f_{|{\rm spt}(\mm)}\in BV({\rm spt}(\mm))\) and, analogously, \(f\in BV^\star(\X)\) 
if and only if \(f_{|{\rm spt}(\mm)}\in BV^\star({\rm spt}(\mm))\). 
Moreover, it holds that
\begin{align}
\label{eq:inv_BV}
|Df|(\X)=|Df_{|{\rm spt}(\mm)}|({\rm spt}(\mm))&\quad\text{ for every }f\in BV(\X),\\
\label{eq:inv_BV_star}
|Df|^\star(\X)=|Df_{|{\rm spt}(\mm)}|^\star({\rm spt}(\mm))&\quad\text{ for every }f\in BV^\star(\X).
\end{align}
\end{theorem}

\begin{proof}
It follows from \eqref{eq:restr_BV} that if \(f\in BV(\X)\), then \(f_{|{\rm spt}(\mm)}\in BV({\rm spt}(\mm))\)
and that the inequality \(|Df_{|{\rm spt}(\mm)}|({\rm spt}(\mm))\leq|Df|(\X)\) holds. Conversely, assume
\(f_{|{\rm spt}(\mm)}\in BV({\rm spt}(\mm))\), 
and take a sequence \((f_n)_n\subseteq\LIP_{b,loc}({\rm spt}(\mm))\cap L^1({\rm spt}(\mm))\)
such that \(f_n\to f_{|{\rm spt}(\mm)}\) in \(L^1({\rm spt}(\mm))\) and
\(\int_{{\rm spt}(\mm)}\lip_a(f_n)\,\d\mm\to|Df_{|{\rm spt}(\mm)}|({\rm spt}(\mm))\). By Theorem \ref{thm:loc_Lip_extension},
we can extend each \(f_n\) to a function \(\bar f_n\in\LIP_{b,loc}(\X)\) satisfying
\(\lip_a(\bar f_n)(x)=\lip_a(f_n)(x)\) for every \(x\in{\rm spt}(\mm)\). Since \(\mm\) is concentrated on \({\rm spt}(\mm)\)
by Proposition ~\ref{prop:main_Radon}, we have that \((\bar f_n)_n\subseteq\LIP_{b,loc}(\X)\cap L^1(\X)\),
\(\bar f_n\to f\) in \(L^1(\X)\) and
\[
|Df|(\X)\leq\lim_{n\to\infty}\int_{\X}\lip_a(\bar f_n)\,\d\mm=\lim_{n\to\infty}\int_{{\rm spt}(\mm)}\lip_a(f_n)\,\d\mm
=|Df_{|{\rm spt}(\mm)}|({\rm spt}(\mm)).
\]
Repeating the same arguments, but using Theorem \ref{thm:DiMa_Gig_Pra} instead of Theorem \ref{thm:loc_Lip_extension},
one can prove that \(f\in BV^\star(\X)\) if and only if \(f_{|{\rm spt}(\mm)}\in BV^\star({\rm spt}(\mm))\),
and that \eqref{eq:inv_BV_star} holds.
\end{proof}
\subsection{BV functions via the approximation modulus}\label{sec:BVAM}
Now we recall the definition of AM-bounding sequence 
from~\cite{DuCa:ErBiq:Ko:Shan:19} and the corresponding notion of BV functions. This definition was motivated by the
paper~\cite{Martio-ConfGeomDyn}, which is based on the earlier works~\cite{Martio, HMMartio-1}. The definition adopted
in~\cite{DuCa:ErBiq:Ko:Shan:19} is motivated by~\cite[Lemma~2]{Martio-ConfGeomDyn}.

\begin{definition}[AM-bounding sequence]\label{defAMbounding}
Let $f:\X\to [-\infty,\infty]$ be an $\mm$-measurable function with $\int_{\X}|f|\,\d\mm<\infty$ and $\Omega\subseteq\X$ be an open set. 
We say that a sequence $(\rho_j)_j$ of nonnegative Borel functions from $\X$ to $[0,\infty]$ is an \emph{AM-bounding sequence} for $f$ in
$\Omega$ if there is a family $\Gamma\subseteq\mathscr R(\X)$
with ${\rm AM}_0(\Gamma)=0$ such that for each $\gamma\in \mathscr R(\X)\setminus\Gamma$ there exists a set 
$N_\gamma\subseteq[a_\gamma,b_\gamma]$ with $\mathscr{H}^1(\gamma(N_\gamma))=0$ satisfying
\begin{equation}\label{eq:BVAM0-1}
|f(\gamma_b)-f(\gamma_a)|\le \varliminf_j\int_{\gamma_{|[a,b]}}\, \rho_j
\end{equation}
for all $a,\,b\in[a_\gamma,b_\gamma]\setminus N_\gamma$ with $a<b$ and $\gamma([a,b])\subseteq\Omega$.
\end{definition}
In the above definition, we can equivalently consider only curves in $\mathscr R(\Omega)\subseteq\mathscr R(\X)$. Indeed,
if the exceptional family $\Gamma$ is only a subfamily of $\mathscr R(\Omega)$, let $\Gamma_{\X}$
be the collection of all $\gamma\in\mathscr R(\X)$ such that $\gamma$ has a sub-curve in $\Gamma$; then 
${\rm AM}_0(\Gamma_{\X})=0$ (see for instance
Proposition~\ref{prop:properties_AM}). Any curve $\gamma\in\mathscr R(\X)\setminus\Gamma_{\X}$ satisfies the condition that
for each $a_0,b_0\in[a_\gamma,b_\gamma]\cap\mathbb{Q}$ with $a_0<b_0$ and $\gamma([a_0,b_0])\subseteq\Omega$ we have
$\gamma\vert_{[a_0,b_0]}\in\mathscr R(\Omega)\setminus\Gamma$, and so there is a set $N[a_0,b_0]\subseteq[a_0,b_0]$ with
$\mathscr{H}^1(\gamma(N[a_0,b_0]))=0$ such that  for each $a,b\in[a_0,b_0]\setminus N[a_0,b_0]$ with $a<b$
the inequality~\eqref{eq:BVAM0-1} holds. With $\mathbb{Q}_\gamma$ the collection of all such pairs $(a_0,b_0)$, 
we know that $\mathscr{H}^1(N_\gamma)=0$ where $N_\gamma=\bigcup_{(a_0,b_0)\in \mathbb{Q}_\gamma}N[a_0,b_0]$
serving the role delineated in the above definition for $\gamma$.

Also this definition is invariant under modification of $f$ in $\mm$-negligible sets, as the following simple lemma shows.

\begin{lemma}\label{lem:BVAM-measurable}
If $(\rho_j)_j$ is an AM-bounding sequence for $f$ in $\Omega$, then it is an AM-bounding sequence in $\Omega$ for any
$\tilde{f}$ such that $\mm(\Omega\cap\{\tilde{f}\neq f\})=0$.
\end{lemma}
\begin{proof}
If $N\subseteq\Omega$ is an $\mm$-negligible set, then by Lemma~\ref{Nages1} we know that
$0\leq {\rm AM}_0(\Gamma_N^+)\leq \Mod(\Gamma_N^+)=0$, and if $\gamma$ is a curve in
$\mathscr R(\X)\setminus\Gamma_N^+$, then by definition $\mathscr{H}^1(\gamma([a_\gamma,b_\gamma])\cap N)=0$. Hence, we may take 
$\tilde{N}_\gamma\coloneqq N_\gamma\cup\gamma^{-1}(N)$ with $N\coloneqq\Omega\cap\{f\neq\tilde{f}\}$.
\end{proof}

\begin{definition}[The space $BV_{\AM}(\X)$]
We say that $f\in L^1(\X)$ belongs to $BV_{\AM}(\X)$ if there exists an AM-bounding sequence $(\rho_j)_j$ for $f$ in $\X$ such that
$\varliminf_j\int_{\X}\rho_j\,\d\mm$ is finite. For any open set $\Omega\subseteq\X$, we define
\[
|Df|_{\AM}(\Omega):=\inf\bigg\{\varliminf_j\, \int_\Omega\rho_j\, \d\mm\;\bigg|\;\text{$(\rho_j)_j$ 
AM-bounding sequence for $f$ in $\Omega$}\bigg\}.
\]
\end{definition}

A suitable modification of
the proofs of Lemma~\ref{lem:sum_am-bounds} and Proposition~\ref{prop:ext_|Df|_am_to_meas} below
can be used to verify that given any \(f\in BV_{\AM}(\X)\), the set-function $|Df|_{\AM}$
has a unique extension to an outer regular Borel measure on \(\X\). 

We highlight that in Definition~\ref{defAMbounding} it is necessary to describe the exceptional 
sets $N_\gamma$ in intrinsic (i.e.\ invariant under reparameterization) terms, either using 
the Hausdorff measure $\mathscr{H}^1$ in $\X$ as we did, or working with the arc-length parameterization \(\hat\gamma\) of \(\gamma\) 
(in this case the condition simply becomes 
$\mathscr{H}^1(N_{\hat\gamma})=0$ with $N_{\hat\gamma}\subseteq [0,\ell_\gamma]$), as 
Example~\ref{ex:strong_am_bound} below illustrates. 

First, we fix some auxiliary notation. We say that \((\rho_j)_j\subseteq\mathcal L^1(X)^+\) is a \emph{strong} AM-bound for \(f\) on \(\Omega\)
if for \({\rm AM}_0\)-a.e.\ \(\gamma\in\mathscr R(\X)\)
there exists a Lebesgue null set \(N_\gamma\subseteq[a_\gamma,b_\gamma]\) such that
\[
|f(\gamma_b)-f(\gamma_a)|\leq\varliminf_{n\to\infty}\int_{\gamma_{|[a,b]}}g_n
\quad\text{ if }a,b\in[a_\gamma,b_\gamma]\setminus N_\gamma\text{, }a<b\text{ and }\gamma([a,b])\subseteq\Omega.
\]
Then we define the quantity \(|Df|_{s}(\Omega)\in[0,\infty]\) as
\[
|Df|_{s}(\Omega)\coloneqq\inf_{(\rho_j)_j}\limi_{j\to\infty}\int_\Omega\rho_j\,\d\mm,
\]
the infimum being taken over all strong AM-bounds \((\rho_j)_j\) for \(f\) on \(\Omega\). It is easy to check
that \(|Df|_{s}\colon\tau(\X)\to[0,\infty]\) is monotone nondecreasing and superadditive on disjoint sets.
\begin{example}\label{ex:strong_am_bound}{\rm
Let us consider \(\R\) as a metric measure space (together with the Euclidean distance and the Lebesgue measure)
and the function \(\nchi_{\mathbb Q}\in\mathcal L^1(\R)\). Since \(\nchi_{\mathbb Q}\) is a.e.\ equivalent to
the null function, we have that \(\nchi_{\mathbb Q}\in BV_{\AM}(\R)\) with $|D\nchi_{\mathbb Q}|_{\AM}(\R)=0$. 
On the other hand, we claim that
\[
|D\nchi_{\mathbb Q}|_{s}(\Omega)=\infty\quad\text{ for every non-empty open set }\Omega\subseteq\R.
\]
To prove it, fix a non-empty open set \(\Omega\subseteq\R\) and \(k\in\N\). Take \(a_1<b_1<\cdots<a_k<b_k\) such that \([a_1,b_1],\ldots,[a_k,b_k]\)
are contained in \(\Omega\). Take pairwise disjoint open subsets \(\Omega_1,\ldots,\Omega_k\) of \(\Omega\) such that \([a_j,b_j]\subseteq\Omega_j\)
for every \(j=1,\ldots,k\). Now, fix any \(j=1,\ldots,k\) and denote by \(\sigma^j\colon[0,b_j-a_j]\to\R\) the arc-length parameterisation of \([a_j,b_j]\).
Pick any rational number \(q_j\in(a_j,b_j)\). We then define \(\gamma^j\colon[0,3]\to\R\) as the unique curve such that
\begin{itemize}
\item \(\gamma^j_{|[0,1]}\) is the constant-speed parameterisation of \([a_j,q_j]\),
\item \(\gamma^j_t=q_j\) for every \(t\in[1,2]\),
\item \(\gamma^j_{|[2,3]}\) is the constant-speed parameterisation of \([q_j,b_j]\).
\end{itemize}
Note that \(\gamma^j\) is rectifiable and \(\hat\gamma^j=\sigma^j\), thus \({\rm AM}_0(\{\gamma^j\})={\rm AM}_0(\{s^j\})=1\) by Example \ref{ex:AM_segment}.
In particular, fixed a strong pointwise AM-bound \((\rho_n)_n\) for \(\nchi_{\mathbb Q}\) on \(\Omega_j\), we can find a Lebesgue null set
\(N_j\subseteq[0,3]\) such that
\[
|\nchi_{\mathbb Q}(\gamma^j_t)-\nchi_{\mathbb Q}(\gamma^j_{t'})|\leq\varliminf_{n\to\infty}\int_{\gamma^j_{|[t',t]}}\rho_n
\quad\text{ for every }t',t\in[0,3]\setminus N_j\text{ with }t'<t.
\]
There exist \(t_1\in(0,1)\setminus N_j\) and \(t_3\in(2,3)\setminus N_j\) with \(\gamma^j_{t_1},\gamma^j_{t_3}\notin\mathbb Q\). Pick \(t_2\in(1,2)\setminus N_\gamma\). Then
\[\begin{split}
\varliminf_{n\to\infty}\int_{\Omega_j}\rho_n(t)\,\d t&\geq\varliminf_{n\to\infty}\int_{\gamma^j_{|[t_1,t_2]}}\rho_n+\varliminf_{n\to\infty}\int_{\gamma^j_{|[t_2,t_3]}}\rho_n\\
&\geq|\nchi_{\mathbb Q}(\gamma^j_{t_2})-\nchi_{\mathbb Q}(\gamma^j_{t_1})|+|\nchi_{\mathbb Q}(\gamma^j_{t_3})-\nchi_{\mathbb Q}(\gamma^j_{t_2})|=\nchi_{\mathbb Q}(q_j)+\nchi_{\mathbb Q}(q_j)=2.
\end{split}\]
By the arbitrariness of \((\rho_n)_n\), we deduce that \(|D\nchi_{\mathbb Q}|_{s}(\Omega_j)\geq 2\). Therefore, we have that
\[
|D\nchi_{\mathbb Q}|_{s}(\Omega)\geq|D\nchi_{\mathbb Q}|_s(\Omega_1)+\cdots+|D\nchi_{\mathbb Q}|_s(\Omega_k)\geq 2k.
\]
By the arbitrariness of \(k\in\N\), we conclude that \(|D\nchi_{\mathbb Q}|_{s}(\Omega)=\infty\).
\fr}\end{example}

In order to compare $BV_{\AM}(\X)$ with other spaces the following lemma will be useful. Recall the definition of
essential variation of functions on intervals from Subsection~\ref{subsec:BV-real}.

\begin{lemma}\label{lem:Nages2}
Suppose that $u\in BV_{\AM}(\X)$. Then there is an ${\rm AM}_0$-negligible family $\Gamma\subseteq\mathscr R(\X)$ such that 
$u\circ\hat\gamma\in BV([0,\ell(\gamma)])$ 
whenever $\gamma\in \mathscr R(\X)\setminus\Gamma$ and
\begin{equation}\label{eq:Nages3}
|D(u\circ\hat{\gamma})|([0,\ell(\gamma)]))\leq\varliminf_k\int_\gamma\rho_k
\qquad\text{for ${\rm AM}_0$-a.e.\ $\gamma$}
\end{equation}
for any AM-bounding sequence $(\rho_k)_k$.
In particular, the essential variation ${\rm eV}\bigl(f\circ\hat\gamma, (0,\ell(\gamma))\bigr)$ of
$f\circ\hat\gamma$ in $(0,\ell(\gamma))$ is finite for ${\rm AM}_0$-a.e.\ curve $\gamma$.
\end{lemma}
\begin{proof}
Let $u\in BV_{\AM}(\X)$. Let $(\rho_k)_k$ be an AM-bounding sequence for $u$ such that $\sup_k\int_{\X}\rho_k\,\d\mm$ is finite.
Let $\Gamma_0$ be the exceptional family for $(\rho_k)_k$ as given by the definition of AM-bounding sequences, and let
$\Gamma_1\subseteq \mathscr R(\X)$ be the collection of all curves $\gamma$ for which 
$\varliminf_k\int_\gamma\rho_k=\infty$.
Then by Lemma~\ref{lem:Fuglede1} and by the sub-additivity of ${\rm AM}_0$, we have that ${\rm AM}_0(\Gamma_0\cup\Gamma_1)=0$.

Now let $\gamma\in\mathscr R(\X)\setminus(\Gamma_0\cup\Gamma_1)$.
Since $\varliminf_k\int_{\hat{\gamma}}\rho_k(\hat{\gamma}(\tau))\,\d\tau=\varliminf_k\int_\gamma\rho_k<\infty$, by passing to a subsequence if necessary
(not reindexed for notational simplicity),
the nonnegative measures $\nu_k\coloneqq(\rho_k\circ\hat\gamma){\mathscr L}^1$ weakly converge on $[0,\ell(\gamma)]$ 
to $\nu_\infty$ with $\nu_\infty([0,\ell(\gamma)])=\varliminf_k\int_\gamma\rho_k$.
Let $E_\gamma\subseteq [0,\ell(\gamma)]$ be the at most countable collection of all $t$
for which $\nu_\infty(\{t\})>0$, and let $K_\gamma:=\hat{\gamma}^{-1}(\gamma(N_\gamma))$.
Then, since by assumption $\mathscr{H}^1(\gamma(N_\gamma))=0$, we have $\mathscr{H}^1(K_\gamma)=0$ 
and, for $s,t\in[0,\ell(\gamma)]\setminus K_\gamma$ with $s<t$, we have that 
\[
|u(\hat{\gamma}(s))-u(\hat{\gamma}(t))|
\le \varliminf_k\int_{\hat{\gamma}_{|[s,t]}}\rho_k(\hat{\gamma}(\tau))\, d\tau
\le \nu_\infty([s,t]).
\]
Set $w:[0,\ell(\gamma)]\to\R$ by $w(t):=\nu_\infty([0,t])$; then $w$ is a bounded, monotone increasing function on $[0,\ell(\gamma)]$, hence a BV function on that interval. Let $v:[0,\ell(\gamma)]\setminus (K_\gamma\cup E_\gamma)\, \to \R$ be given by 
$v(t)=u(\hat{\gamma}(t))+w(t)$. If $t_1,t_2\in[0,\ell(\gamma)]\setminus(K_\gamma\cup E_\gamma)$ with $t_1<t_2$, we have
\begin{align*}
v(t_2)&=u(\hat{\gamma}(t_2))+w(t_2)-u(\hat{\gamma}(t_1))-w(t_1)+v(t_1)\\
    &=u(\hat{\gamma}(t_2))-u(\hat{\gamma}(t_1))+\nu_\infty([t_1,t_2])+v(t_1)\ge v(t_1),
\end{align*}
that is, $v$ is monotone nondecreasing on $[0,\ell(\gamma)]\setminus(K_\gamma\cup E_\gamma)$. As $\mathscr{H}^1(K_\gamma\cup E_\gamma)=0$,
it follows that there is at least one monotone increasing function on $[0,\ell(\gamma)]$ that agrees with $v$ on 
$[0,\ell(\gamma)]\setminus(K_\gamma\cup E_\gamma)$. We choose one such extension and denote that extension also by $v$.
Now we have two monotone increasing functions $w,v$ on $[0,\ell(\gamma)]$, and hence the function $\widetilde{u}=v-w$ is a
function of bounded variation
on $[0,\ell(\gamma)]$. Since $u\circ\hat{\gamma}=\widetilde{u}$ on 
$[0,\ell(\gamma))\setminus(K_\gamma\cup E_\gamma)$
and $\mathscr{H}^1(K_\gamma\cup E_\gamma)=0$, it follows that
$u\circ\hat{\gamma}$ is a BV function on $[0,\ell(\gamma)]$, with total variation bounded
from above as in \eqref{eq:Nages3} thanks to $\nu_\infty([0,\ell(\gamma)])=\varliminf_k\int_\gamma\rho_k<\infty$.
\end{proof}

Now we introduce two more variants, where the exceptional points are measured in terms of the intersection of $\gamma$
with a given null set in $\X$ (but, possibly dependent on the sequence $(\rho_j)_j$). The advantage of this formulation is
that we need only to consider the endpoints of $\gamma$, as in the definitions based on test plans (also in analogy with the
Newtonian space definition). Exceptional curves can be measured either in terms of ${\rm AM}_0$ or in terms
of ${\rm AM}_1$, and this leads to two definitions.
\begin{definition}[\({\rm am}_0\)-bounding sequence]
Let \((\X,\sfd,\mm)\) be a metric measure space. Fix \(f\in\mathcal L^1(\X)\) and let \(\Omega\subseteq\X\) be open.
Then we say that \((\rho_j)_j\subseteq\mathcal L^1(\Omega)^+\) is an \emph{\({\rm am}_0\)-bounding sequence} for \(f\) on \(\Omega\)
if there exists an \(\mm\)-null Borel set \(N\subseteq\Omega\) such that for \({\rm AM}_0\)-a.e.\ \(\gamma\in\mathscr R(\Omega)\)
it holds that
\begin{equation}\label{eq:tildeam-bound-alt_w}
|f(\gamma_{b_\gamma})-f(\gamma_{a_\gamma})|\leq\varliminf_{j\to\infty}\int_{\gamma}\rho_j
\quad\text{ whenever }
\gamma_{a_\gamma},\gamma_{b_\gamma}\notin N.
\end{equation}
\end{definition}

\begin{remark}[Equivalent formulation of ${\rm am}_0$-bounding sequence]\label{rem:localize_ambound}{\rm
Since the class of curves $\gamma$ in $\mathscr R(\X)$ satisfying $\gamma([a_\gamma,b_\gamma])\subseteq\Omega$ and
\[
|f(\gamma_{b_\gamma})-f(\gamma_{a_\gamma})|>\varliminf_{j\to\infty}\int_{\gamma} \rho_j\quad\text{and}\quad
\gamma_{a_\gamma},\gamma_{b_\gamma}\notin N
\]
is, by definition of ${\rm am}_0$-bound on $\Omega$, ${\rm AM}_0$-negligible, Proposition \ref{prop:properties_AM}\ref{subcurves}
ensures that an equivalent formulation of ${\rm am}_0$-bounding sequence is based on
\begin{equation}\label{eq:tildeam-bound-alt}
|f(\gamma_b)-f(\gamma_a)|\leq\varliminf_{j\to\infty}\int_{\gamma_{|[a,b]}}\rho_j
\quad\text{ if }a_\gamma\leq a<b\leq b_\gamma\text{, }\gamma([a,b])\subseteq\Omega\text{ and }\gamma_a,\gamma_b\notin N
\end{equation}
for ${\rm AM}_0$-a.e.\ $\gamma$.
\fr}\end{remark}

\begin{definition}[\({\rm am}_1\)-bounding sequence]\label{def:AM-lambda-bounding}
Let \((\X,\sfd,\mm)\) be a metric measure space. Fix \(f\in\mathcal L^1(\X)\) and let \(\Omega\subseteq\X\) be open.
Then we say that \((g_n)_n\subseteq\mathcal L^1(\Omega)^+\) is an \emph{\({\rm am}_1\)-bounding sequence} for \(f\) on \(\Omega\)
if for \({\rm AM}_1\)-a.e.\ \(\gamma\in\mathscr R(\X)\) it holds that
\begin{equation}\label{eq:tildeam-bound-alt_w_1}
|f(\gamma_{b_\gamma})-f(\gamma_{a_\gamma})|\leq\varliminf_{j\to\infty}\int_{\gamma}\rho_j
\quad\text{whenever}\quad\gamma([a_\gamma,b_\gamma])\subseteq\Omega.
\end{equation}
\end{definition}
Note that, as ${\rm AM}_1$ provides a better control on the measure
at the endpoints, for the notion of ${\rm am}_1$-bounding sequence we do not need
to postulate an exceptional $\mm$-negligible set as we did for ${\rm am}_0$-bounding sequences.
Clearly, every subsequence of an \({\rm am}_\lambda\)-bounding sequence for \(f\) on \(\Omega\) is still an \({\rm am}_\lambda\)-bounding sequence for \(f\) on \(\Omega\).
It is easily seen that, as for all the other notions we introduced so far, for both notions of ${\rm am}$-bounding sequences there is stability with respect to modifications of $f$ in $\mm$-negligible sets (for the case of the \({\rm am}_0\)-bound simply enlarge $N$, for the case of the \({\rm am}_1\)-bound
see Proposition \ref{prop:properties_AM}\ref{it:ineq_AM_mm}). Hence, the following definitions are well posed
even though we conflate functions $f$ with equivalence classes $[f]_{\mm}$.
\begin{definition}[The spaces \(BV_{am}^0(\X)\), \(BV_{am}^1(\X)\)]\label{def:BV_am}
Let \((\X,\sfd,\mm)\) be a metric measure space, \(\lambda\in\{0,1\}\) and \(f\in \mathcal{L}^1(\X)\). Then we declare that \(f\) belongs to \(BV_{am}^\lambda(\X)\)
if there exists an \({\rm am}_\lambda\)-bounding sequence \((\rho_j)_j\) for \(f\) on \(\X\) such that \(\sup_{j\in\N}\|\rho_j\|_{L^1(\X)}<\infty\). Given any function
\(f\in BV_{am}^\lambda(\X)\), we define
\begin{equation}\label{eq:def_|Df|_am}
|Df|_{am}^\lambda(\Omega)\coloneqq\inf_{(\rho_j)_j}\limi_{j\to\infty}\int_\Omega\rho_j\,\d\mm\quad\text{ for every open set }\Omega\subseteq\X,
\end{equation}
where the infimum is taken over all \({\rm am}_\lambda\)-bounding sequences \((\rho_j)_j\) for \(f\) on \(\Omega\).
\end{definition}

Based on the discussion in the paragraph preceding Definition~\ref{def:AM-lambda-bounding}, we know that
if $f\in BV_{am}^\lambda(\X)$ and $\widehat{f}:\X\to[-\infty,\infty]$ such that $f=\widehat{f}$ $\mm$-a.e.~in $\X$, then
$\widehat{f}\in BV_{am}^\lambda(\X)$ with $|Df|_{am}^\lambda(\Omega)=|D\widehat{f}|_{am}^\lambda(\Omega)$ for all
$\Omega\in\tau(\X)$; hence we can consider $f\in L^1(\X)$ rather than focusing on specific representative $f\in\mathcal{L}^1(\X)$
in the above definition.
The inequality ${\rm AM}_1\leq {\rm AM}_0$, combined with Corollary \ref{cor:exceptional_AM}, 
provides the inclusion
\[
BV_{am}^0(\X)\subseteq BV_{am}^1(\X),
\]
as well as the inequality $|Df|_{am}^1(\Omega)\leq |Df|_{am}^0(\Omega)$ for every \(f\in BV_{am}^0(\X)\) and \(\Omega\in\tau(\X)\).
\medskip

Our next goal is to show that, given \(f\in BV_{am}^0(\X)\), the set-function \(|Df|_{am}^0\colon\tau(\X)\to[0,\infty)\)
defined in \eqref{eq:def_|Df|_am} can be uniquely extended to an outer regular Borel measure \(|Df|_{am}^0\) on \(\X\).

\begin{lemma}\label{lem:sum_am-bounds}
Let \((\X,\sfd,\mm)\) be a metric measure space. Let \(f\in\mathcal L^1(\X)\) be fixed and let \((\Omega_k)_k\) be a sequence of open subsets of \(\X\).
For any \(k\in\N\), let \((g_n^k)_n\) be an \({\rm am}_0\)-bounding sequence for \(f\) on \(\Omega_k\). Assume also that
\begin{equation}\label{eq:sum_tildeam-bounds_alt}
g_n\coloneqq\sum_{k=1}^\infty\nchi_{\Omega_k}\, g_n^k\in\mathcal L^1(\X)\quad\text{ for every }n\in\N.
\end{equation}
Then \((g_n)_n\) is an \({\rm am}_0\)-bounding sequence for \(f\) on \(\bigcup_{k\in\N}\Omega_k\).
In particular, if \(f\in BV_{am}^0(\X)\), then
\[
|\D f|_{am}^0\bigg(\bigcup_{k\in\N}\Omega_k\bigg)\leq\sum_{k=1}^\infty|\D f|_{am}^0(\Omega_k).
\]
\end{lemma}

\begin{proof}
Given any \(k\in\N\), we can find an \(\mm\)-null Borel set \(N_k\subseteq\Omega_k\) such that
for \({\rm AM}_0\)-a.e.\ curve \(\gamma\in\mathscr R(\X)\) it holds that
\(|f(\gamma_b)-f(\gamma_a)|\leq\varliminf_n\int_{\gamma_{|[a,b]}}g_n^k\) whenever \(a_\gamma\leq a<b\leq b_\gamma\)
with \(\gamma([a,b])\subseteq\Omega_k\) and 
\(\gamma_a,\gamma_b\notin N_k\). Define \(N\coloneqq\bigcup_{k\in\N}N_k\)
and note that \(\mm(N)=0\). Taking also Corollary \ref{cor:exceptional_AM} into account, we can thus find a set
\(\mathcal N\subseteq\mathscr R(\X)\) with \({\rm AM}_0(\mathcal N)=0\) having the following properties:
for any \(\gamma\in\mathscr R(\X)\setminus\mathcal N\) and \(k\in\N\), we have that
\begin{equation}\label{eq:choice_mathcalN_1}
|f(\gamma_b)-f(\gamma_a)|\leq\varliminf_{n\to\infty}\int_{\gamma_{|[a,b]}}g_n^k
\quad\text{ if }a_\gamma\leq a<b\leq b_\gamma\text{, }\gamma([a,b])\subseteq\Omega_k\text{ and }\gamma_a,\gamma_b\notin N,
\end{equation}
along with
\begin{equation}\label{eq:choice_mathcalN_2}
\gamma\big([a,b]\cap\gamma^{-1}(\X\setminus N)\big)\text{ is dense in }\gamma([a,b])
\quad\text{ for every }a,b\in[a_\gamma,b_\gamma]\text{ with }a<b.
\end{equation}
Now, fix \(\gamma\in\mathscr R(\X)\setminus\mathcal N\) and
\(a,b\in[a_\gamma,b_\gamma]\) with \(a<b\) such that \(\gamma([a,b])\subseteq\Omega\coloneqq\bigcup_{k\in\N}\Omega_k\)
and \(\gamma_a,\gamma_b\notin N\). By the compactness of \(\gamma([a,b])\) and by~\eqref{eq:choice_mathcalN_2}, we can
find \(m\in\N\), \(a_1,\ldots,a_m\in(a,b)\) with \(a\eqqcolon a_0<a_1<\cdots<a_m<a_{m+1}\coloneqq b\)
and \(k_0,\ldots,k_m\in\N\) such that \(\gamma([a_j,a_{j+1}])\subseteq\Omega_{k_j}\) and \(\gamma_{a_j}\notin N\)
for every \(j=0,\ldots,m\). Hence, using \eqref{eq:choice_mathcalN_1} and Fatou's lemma we can estimate
\begin{align*}
|f(\gamma_b)-f(\gamma_a)|&\leq\sum_{j=0}^m|f(\gamma_{a_{j+1}})-f(\gamma_{a_j})|\leq
\sum_{j=0}^m\varliminf_{n\to\infty}\int_{\gamma_{|[a_j,a_{j+1}]}}g_n^{k_j}\\
&\leq\varliminf_{n\to\infty}\sum_{j=0}^m\int_{\gamma_{|[a_j,a_{j+1}]}}\nchi_{\Omega_{k_j}}g_n^{k_j}
\leq\varliminf_{n\to\infty}\sum_{j=0}^m\int_{\gamma_{|[a_j,a_{j+1}]}}g_n=\varliminf_{n\to\infty}\int_{\gamma_{|[a,b]}}g_n,
\end{align*}
which shows that \((g_n)_n\) is an \({\rm am}_0\)-bounding sequence for \(f\) on \(\Omega\). 

To prove the last part of the statement,
assume that \(f\in BV_{am}^0(\X)\). If \(\sum_{k=1}^\infty|\D f|_{am}^0(\Omega_k)=\infty\), then there is nothing to prove. Thus, assume
also that \(\sum_{k=1}^\infty|\D f|_{am}^0(\Omega_k)<\infty\). Given \(\varepsilon>0\) and \(k\in\N\), we can take an \({\rm am}_0\)-bounding sequence \((g_n^k)_n\)
for \(f\) on \(\Omega_k\) such that \(\int_{\Omega_k}g_n^k\,\d\mm\leq|\D f|_{am}^0(\Omega_k)+\varepsilon2^{-k}\) for every \(n\in\N\). We know from the first part
of the statement that \((g_n)_n\) is an \({\rm am}_0\)-bounding sequence for \(f\) on \(\Omega\), thus
\[
|\D f|_{am}^0(\Omega)\leq\varliminf_{n\to\infty}\int_\Omega g_n\,\d\mm=\varliminf_{n\to\infty}\sum_{k=1}^\infty\int_{\Omega_k}g_n^k\,\d\mm
\leq\varepsilon+\sum_{k=1}^\infty|\D f|_{am}^0(\Omega_k).
\]
By the arbitrariness of \(\varepsilon>0\), we conclude that \(|\D f|_{am}^0(\Omega)\leq\sum_{k=1}^\infty|\D f|_{am}^0(\Omega_k)\).
\end{proof}
Combining Lemma \ref{lem:sum_am-bounds} with the criterion stated in Theorem \ref{thm:cor_of_DeGiorgi-Letta}, we obtain the following.
\begin{proposition}\label{prop:ext_|Df|_am_to_meas}
Let \((\X,\sfd,\mm)\) be a metric measure space. Let \(f\in BV_{am}^0(\X)\) be given. Then the set-function
\(|\D f|_{am}^0\colon\tau(\X)\to[0,\infty)\) defined in \eqref{eq:def_|Df|_am} can be uniquely extended to
an outer regular Borel measure on \(\X\), which we still denote by \(|\D f|_{am}^0\).
\end{proposition}
\begin{proof}
We check that \(\phi\coloneqq|Df|_{am}^0\colon\tau(\X)\to[0,\infty)\) fulfils the assumptions of Theorem \ref{thm:cor_of_DeGiorgi-Letta}:
\begin{enumerate}[label=(\roman*)]
\item Theorem \ref{thm:cor_of_DeGiorgi-Letta}\ref{it:phi_empty} trivially holds.
\item If \(U,V\in\tau(\X)\) and \(U\subseteq V\), then \(({g_n}_{|U})_n\) is an \({\rm am}_0\)-bounding sequence for \(f\) on \(U\) when \((g_n)_n\) is
an \({\rm am}_0\)-bound for \(f\) on \(V\). This implies that \(|\D f|_{am}^0(U)\leq|\D f|_{am}^0(V)\), so that \(\phi\) satisfies
Theorem \ref{thm:cor_of_DeGiorgi-Letta}\ref{it:isoton}.
\item Theorem \ref{thm:cor_of_DeGiorgi-Letta}\ref{it:sigma-subadd} was proved in Lemma \ref{lem:sum_am-bounds}.
\item Fix two
sets \(U,V\in\tau(\X)\) with $\distd(U,V)>0$. 
Take an \({\rm am}_0\)-bounding sequence \((g_n)_n\) for \(f\) on \(U\cup V\) such that \(\int_{U\cup V}g_n\,\d\mm\to|\D f|_{am}^0(U\cup V)\).
Since \(({g_n}|_U)_n\) and \(({g_n}|_V)_n\) are \({\rm am}_0\)-bounding sequences for \(f\) on \(U\) and \(V\), respectively, we can estimate
\begin{align*}
|\D f|_{am}^0(U)+|\D f|_{am}^0(V)&\leq\varliminf_{n\to\infty}\int_U g_n\,\d\mm+\varliminf_{n\to\infty}\int_V g_n\,\d\mm\\
&=\lim_{n\to\infty}\int_{U\cup V}g_n\,\d\mm=|\D f|_{am}^0(U\cup V).
\end{align*}
This proves that \(\phi\) satisfies Theorem~\ref{thm:cor_of_DeGiorgi-Letta}\ref{it:superadd_on_disj}.
\end{enumerate}
Therefore, by applying Theorem \ref{thm:cor_of_DeGiorgi-Letta} we achieve the statement.
\end{proof}
It is not clear to us whether, in general, the set-function \(|Df|_{am}^1\colon\tau(\X)\to[0,\infty)\) is the restriction of a Borel measure for every
\(f\in BV_{am}^1(\X)\). However, this is the case at least when \((\X,\sfd)\) is locally complete (see Remark \ref{rmk:produce_meas_loc-compl}).
\begin{proposition}\label{prop:am_0_vs_am_0_pt}
Let \((\X,\sfd,\mm)\) be a metric measure space. Then \(BV_{am}^0(\X)\subseteq BV_{\AM}(\X)\) and
\[
|Df|_{\AM}\leq|Df|_{am}^0\quad\text{ for every }f\in BV_{am}^0(\X).
\]
\end{proposition}
\begin{proof}
To prove the statement, it suffices to show that each \({\rm am}_0\)-bounding sequence \((\rho_j)_j\) for \(f\) on an open set \(\Omega\subseteq\X\) is also
an AM-bounding sequence for \(f\) on \(\Omega\). Taking Remark~\ref{rem:localize_ambound} into account, we find an \(\mm\)-null
set \(N\subseteq\Omega\) such that the following holds: for \({\rm AM}_0\)-a.e.\ \(\gamma\in\mathscr R(\Omega)\), we have
\begin{equation}\label{eq:am_0_vs_am_0_pt}
|f(\gamma_b)-f(\gamma_a)|\leq\varliminf_{j\to\infty}\int_{\gamma_{|[a,b]}}\rho_j
\quad\text{ if }a_\gamma\leq a<b\leq b_\gamma
\text{ and }\gamma_a,\gamma_b\notin N.
\end{equation}
Also, Lemma~\ref{Nages1} gives that $\mathscr{H}^1(\gamma([a_\gamma,b_\gamma])\cap N)=0$
for \({\rm AM}_0\)-a.e.\ \(\gamma\in\mathscr R(\Omega)\). Combining this 
with~\eqref{eq:am_0_vs_am_0_pt}, with $N_\gamma=\gamma^{-1}(N)$
we conclude that \((\rho_j)_j\) is
an AM-bounding sequence for \(f\) on \(\Omega\), since $\gamma(N_\gamma)\subseteq N$.
\end{proof}
\subsection{BV functions via test plans}\label{sec:tpp}
In this section we introduce our definition of $BV$ space based on test plans. The original definition 
given in~\cite{Amb:DiMa:14} directly provides a notion of ``measure upper gradient''.
Its equivalence with the definition below is discussed in Appendix~\ref{sec:tppp} (equivalence holds 
at least in locally complete metric measure spaces). Here we opt for a simpler definition, more in the spirit
of the definitions adopted for Sobolev classes, as suggested in Remark~7.2 of  \cite{Amb:DiMa:14} and in
Di Marino's PhD thesis \cite{DiMaPhD:14}. We
only improve it slightly, using the difference of $f$ at the endpoints and not its modulus (see 
Lemma~\ref{lem:Federichi_Nobili} below, borrowed from \cite{Federichi_Nobili}).
This specific approach to defining the BV space had an essential role in the paper \cite{Nob:Pas:Sch:22}.
The notion of test plans is defined in subsection~\ref{SubSec:TestPlans} above; the collection of all
test plans on the metric space $\X$ is denoted there by $\Pi(\X)$.
\begin{definition}[The space \(BV_{tp}^\star(\X)\)]\label{def:BV-test-plan}
Let \((\X,\sfd,\mm)\) be a metric measure space and \(f\in L^1(\X)\). We declare that \(f\)
belongs to \(BV_{tp}^\star(\X)\) if there exists a constant \(C\geq 0\) such that
\begin{equation}\label{eq:def_BV_tp}
\int_{C([0,1]; \X)}\, f(\gamma_1)-f(\gamma_0)\,\d\ppi(\gamma)\leq{\rm Comp}(\ppi)\Lip(\ppi)\, C\quad\text{ for every }\ppi\in\Pi(\X).
\end{equation}
Given any \(f\in BV_{tp}^\star(\X)\) and \(\Omega\in\tau(\X)\), we define the quantity \(|Df|_{tp}^\star(\Omega)\in[0,\infty)\) as
the minimal constant \(C\geq 0\) for which \eqref{eq:def_BV_tp} holds with $\Omega$ playing the role of $\X$, namely
\[
|Df|_{tp}^\star(\Omega)\coloneqq\sup\bigg\{\frac{1}{{\rm Comp}(\ppi)\Lip(\ppi)}\int_{C([0,1];\Omega)} f(\gamma_1)-f(\gamma_0)\,\d\ppi(\gamma)\;\bigg|\;
\ppi\in\Pi(\Omega),\,\Lip(\ppi)\neq 0\bigg\}.
\]
\end{definition}

Note that the definition is well posed, as the very definition of test plan ensures that
$f(\gamma_i)=\tilde{f}(\gamma_i)$, \(i=0,1\) for $\ppi$-a.e.\ \(\gamma\) whenever $\{f\neq\tilde{f}\}$ is
$\mm$-negligible. An equivalent formulation of \eqref{eq:def_BV_tp} would involve the modulus of the
integral  \(\int f(\gamma_1)-f(\gamma_0)\,\d\ppi(\gamma)\), as all the quantities involved in the right
hand side are invariant under time inversion:
\begin{remark}{\rm
One could equivalently define \(BV_{tp}^\star(\X)\) by asking that for some \(C\geq 0\) it holds
\begin{equation}\label{eq:def_BV_tp_alt}
\bigg|\int_{C([0,1]; \X)}\, f(\gamma_1)-f(\gamma_0)\,\d\ppi(\gamma)\bigg|\leq{\rm Comp}(\ppi)\Lip(\ppi)\, C
\quad\text{ for every }\ppi\in\Pi(\X).
\end{equation}
Indeed, on the one hand it is clear that \eqref{eq:def_BV_tp_alt} implies \eqref{eq:def_BV_tp}. On the other hand, we have
\[
\bigg|\int f(\gamma_1)-f(\gamma_0)\,\d\ppi(\gamma)\bigg|=\max\bigg\{\int f(\gamma_1)-f(\gamma_0)\,\d\ppi(\gamma),
\int f(\gamma_1)-f(\gamma_0)\,\d({\sf rev}_\#\ppi)(\gamma)\bigg\}
\]
for every \(\ppi\in\Pi(\X)\), where the map \({\sf rev}\colon C([0,1];\X)\to C([0,1];\X)\) is given by
\({\sf rev}(\gamma)_t\coloneqq\gamma_{1-t}\). Since we have that \({\sf rev}_\#\ppi\in\Pi(\X)\),
\({\rm Comp}({\sf rev}_\#\ppi)={\rm Comp}(\ppi)\) and \(\Lip({\sf rev}_\#\ppi)=\Lip(\ppi)\), we deduce
that \eqref{eq:def_BV_tp} implies \eqref{eq:def_BV_tp_alt}.

Moreover, we can have a potentially stronger condition characterizing membership in $BV_{tp}^\star(\X)$, 
namely that
\begin{equation}\label{eq:Ambrosio-defn}
\int_{C([0,1]; \X)}\, |f(\gamma_1)-f(\gamma_0)|\,\d\ppi(\gamma) \leq{\rm Comp}(\ppi)\Lip(\ppi)\, C
\quad\text{ for every }\ppi\in\Pi(\X).
\end{equation}
To see this, set $\Gamma_+$ to be the collection of all curves $\gamma\in C([0,1];\Omega)$ for which
$f(\gamma_1)-f(\gamma_0)\ge 0$, and for $\Gamma\subseteq C([0,1];\Omega)$ we set the time-reversed family to be
\[
-\Gamma:=\{\gamma\in C([0,1];\Omega)\, :\, [0,1]\ni t\mapsto\gamma(1-t)\, \in \Gamma\}
\]
and $\ppi_-$ be the measure given by $\ppi_-(\Gamma):=\ppi(-\Gamma)$.
Then, given $\ppi\in\Pi(\Omega)$, we can set
\[
\widehat{\ppi}:=\ppi_{|{\Gamma_+}}+(\ppi_-)_{|{C([0,1]\Omega)\setminus\Gamma_+}},
\]
and note that $\widehat{\ppi}\in\Pi(\Omega)$. We can use this to get the equivalence with 
version given by~\eqref{eq:Ambrosio-defn}.
In general the norms associated to these definitions could differ, as ${\rm Comp}(\ppi)\le {\rm Comp}(\widehat{\ppi})\le 2 {\rm Comp}(\ppi)$. In the situation
considered in Theorem~\ref{thm:equiv_BV_loc-complete} the two norms are indeed proved to be equal. The condition~\eqref{eq:Ambrosio-defn},
in analogy with the one of Newtonian spaces (as Definition~\ref{def:BV-test-plan}, without the modulus), was suggested in Remark~7.2 of \cite{Amb:DiMa:14}.
The full equivalence with the even (a priori) stronger original definition of \cite{Amb:DiMa:14}, that has the advantage of providing directly a measure, is discussed in
Appendix~\ref{sec:tppp}.
\fr}\end{remark}

It is not clear to us whether, in general, \(|Df|_{tp}^\star\colon\tau(\X)\to[0,\infty)\) is the restriction of a Borel measure.
However, this is the case at least when \((\X,\sfd)\) is locally complete (see Remark \ref{rmk:produce_meas_loc-compl}),
because of the coincidence on open sets with other notions of total variations for which this property is known.
\medskip

Because $\Pi(E)\subseteq\Pi(\X)$ when a given
Borel set \(E\subseteq\X\) is equipped with the metric $\sfd_{|E\times E}$ and measure $\mm_{|\mathscr B(E)}$,
from the definitions it is directly verifiable that given a function \(f\in BV_{tp}^\star(\X)\),
\begin{equation}\label{eq:restr_BV_tp}
f_{|E}\in BV_{tp}^\star(E)\;\text{ and }\;|Df_{|E}|_{tp}^\star(E)\leq|Df|_{tp}^\star(\X)
\quad\text{ for every }f\in BV_{tp}^\star(\X).
\end{equation}
\begin{proposition}[Invariance properties of \(BV_{tp}^\star(\X)\)]\label{prop:invar_BV_tp}
Let \((\X,\sfd,\mm)\) be a metric measure space and \(f\in\mathcal L^1(\X)\). Then \(f\in BV_{tp}^\star(\X)\) if and only if
\(f_{|{\rm spt}(\mm)}\in BV_{tp}^\star({\rm spt}(\mm))\). Moreover,
\[
|Df|_{tp}^\star(\X)=|Df_{|{\rm spt}(\mm)}|_{tp}^\star({\rm spt}(\mm))\quad\text{ for every }f\in BV_{tp}^\star(\X).
\]
\end{proposition}

\begin{proof}
On the one hand, if \(f\in BV_{tp}^\star(\X)\), then \eqref{eq:restr_BV_tp} gives that \(f_{|{\rm spt}(\mm)}\in BV_{tp}^\star({\rm spt}(\mm))\)
and \(|Df_{|{\rm spt}(\mm)}|_{tp}^\star({\rm spt}(\mm))\leq|Df|_{tp}^\star(\X)\). On the other hand, assume \(f_{|{\rm spt}(\mm)}\in BV_{tp}^\star({\rm spt}(\mm))\)
and fix \(\ppi\in\Pi(\X)\). Remark \ref{rmk:tp_conc_on_spt} gives that \(\ppi\) is concentrated on \(C([0,1];{\rm spt}(\mm))\subseteq C([0,1];\X)\),
thus \(\ppi\) induces a test plan \(\tilde\ppi\in\Pi({\rm spt}(\mm))\) in a canonical way. Therefore, we can estimate
\begin{align*}
\int_{C([0,1]; \X)}\, |f(\gamma_1)-f(\gamma_0)|\,\d\ppi(\gamma)
&=\int_{C([0,1]; {\rm spt}(\mm))}\, \big|f_{|{\rm spt}(\mm)}(\gamma_1)-f_{|{\rm spt}(\mm)}(\gamma_0)\big|\,\d\tilde\ppi(\gamma)\\
&\leq{\rm Comp}(\tilde\ppi)\Lip(\tilde\ppi)|Df_{|{\rm spt}(\mm)}|_{tp}^\star({\rm spt}(\mm))\\
&={\rm Comp}(\ppi)\Lip(\ppi)|Df_{|{\rm spt}(\mm)}|_{tp}^\star({\rm spt}(\mm)),
\end{align*}
whence it follows that \(f\in BV_{tp}^\star(\X)\) and \(|Df|_{tp}^\star(\X)\leq|Df_{|{\rm spt}(\mm)}|_{tp}^\star({\rm spt}(\mm))\), as desired.
\end{proof}
The following lemma will be useful to reduce problems related to BV functions to problems related
to bounded functions in the proof of Theorem~\ref{thm:incl_BV_spaces}.
\begin{lemma}\label{lem:Federichi_Nobili}
Let \((\X,\sfd,\mm)\) be a metric measure space. Assume that \(f\in\mathcal L^1(\X)\) and \(C\geq 0\) satisfy
\[
\int_{C([0,1]; \X)}\,  f(\gamma_1)-f(\gamma_0)\,\d\ppi(\gamma)\leq{\rm Comp}(\ppi)\Lip(\ppi)\, C\quad\text{ for every }\ppi\in\Pi(\X).
\]
Then, for any \(a,b\in\R\) with \(a<b\), we have that 
\(f_{a,b}\coloneqq(f\vee a)\wedge b\in\mathcal L^1(\X)\cap\mathcal L^\infty(\X)\) satisfies
\[
\int_{C([0,1]; \X)}\,  f_{a,b}(\gamma_1)-f_{a,b}(\gamma_0)\,\d\ppi(\gamma)\leq{\rm Comp}(\ppi)\Lip(\ppi)\,C\quad\text{ for every }\ppi\in\Pi(\X).
\]
Conversely, if \(f\in\mathcal L^1(\X)\) and \(C\geq 0\) such that for each $a,b\in\R$ with $a<b$ we have
\[
\int_{C([0,1]; \X)}\,  f_{a,b}(\gamma_1)-f_{a,b}(\gamma_0)\,\d\ppi(\gamma)\leq{\rm Comp}(\ppi)\Lip(\ppi)\,C\quad\text{ for every }\ppi\in\Pi(\X),
\]
then
\[
\int_{C([0,1]; \X)}\,  f(\gamma_1)-f(\gamma_0)\,\d\ppi(\gamma)\leq{\rm Comp}(\ppi)\Lip(\ppi)\, C\quad\text{ for every }\ppi\in\Pi(\X).
\]
\end{lemma}

\begin{proof}
Define \(\Gamma_{a,b}\coloneqq\{\gamma\in C([0,1];\X):f_{a,b}(\gamma_1)\geq f_{a,b}(\gamma_0)\}\) and fix \(\ppi\in\Pi(\X)\).
If \(\ppi(\Gamma_{a,b})=0\), then \(f_{a,b}(\gamma_1)<f_{a,b}(\gamma_0)\) for \(\ppi\)-a.e.\ \(\gamma\) and thus
\[
\int_{C([0,1]; \X)}\,  f_{a,b}(\gamma_1)-f_{a,b}(\gamma_0)\,\d\ppi(\gamma)<0\leq{\rm Comp}(\ppi)\Lip(\ppi)\, C.
\]
If \(\ppi(\Gamma_{a,b})>0\), then note that
\(\ppi_{a,b}\coloneqq\ppi(\Gamma_{a,b})^{-1}\ppi\llcorner\Gamma_{a,b}\in\Pi(\X)\).
Observing that
\[
f_{a,b}(\gamma_1)-f_{a,b}(\gamma_0)=|f_{a,b}(\gamma_1)-f_{a,b}(\gamma_0)|\leq|f(\gamma_1)-f(\gamma_0)|
=f(\gamma_1)-f(\gamma_0)\quad\text{ for all }\gamma\in\Gamma_{a,b},
\]
and using the fact that \({\rm Comp}(\ppi_{a,b})\leq\ppi(\Gamma_{a,b})^{-1}{\rm Comp}(\ppi)\) and
\(\Lip(\ppi_{a,b})\leq\Lip(\ppi)\), we get
\[\begin{split}
\int_{C([0,1]; \X)}\,  f_{a,b}(\gamma_1)-f_{a,b}(\gamma_0)\,\d\ppi(\gamma)&\leq
\ppi(\Gamma_{a,b})\int_{C([0,1]; \X)}\,  f_{a,b}(\gamma_1)-f_{a,b}(\gamma_0)\,\d\ppi_{a,b}(\gamma)\\
&\leq\ppi(\Gamma_{a,b})\int_{C([0,1]; \X)}\,  f(\gamma_1)-f(\gamma_0)\,\d\ppi_{a,b}(\gamma)\\
&\leq\ppi(\Gamma_{a,b}){\rm Comp}(\ppi_{a,b})\Lip(\ppi_{a,b})C\leq{\rm Comp}(\ppi)\Lip(\ppi)C.
\end{split}\]
The converse direction is proved by noting that 
$C([0,1];\X)=\bigcup_{(a,b)\in\mathbb{Q}_P}\Gamma_{a,b}$ with $\mathbb{Q}_P$ the collection of all pairs of rational numbers
$a,b$ with $a<b$, and appealing to a monotone convergence theorem.
\end{proof}
\section{Equivalence of metric BV spaces}
The following theorem summarizes the discussion about the inclusions in \eqref{chain1}, \eqref{chain2}, \eqref{chain3} that
hold unconditionally at the level of general metric measure spaces. We only add here the short proofs of the
inclusions of $BV_N(\X)$ in $BV_{am}^0(\X)$ and of \(BV_{am}^1(\X)\) in \(BV_{tp}^\star(\X)\).
\begin{theorem}\label{thm:incl_BV_spaces}
Let \((\X,\sfd,\mm)\) be a metric measure space. Then
\[
BV^\star(\X)\subseteq BV(\X)\subseteq BV_N(\X)\subseteq BV_{am}^0(\X)\subseteq BV_{\AM}(\X)\cap BV_{am}^1(\X)\]
and finally
\[
BV_{\AM}(\X)\cup BV_{am}^1(\X)\subseteq BV_{tp}^\star(\X).
\]
Moreover, in correspondence to these inclusions, it holds that
\begin{align}
\label{eq:ineq_BV_star-BV}
|Df|(\X)\leq|Df|^*(\X)&\quad\forall f\in BV^\star(\X),\\
\label{eq:ineq_BV_BVN}
|Df|_N\leq|Df|&\quad\forall f\in BV(\X),\\
\label{eq:ineq_BV-BV_am}
|Df|_{am}^0\leq|Df|_N&\quad\forall f\in BV_N(\X),\\
\label{eq:ineq:BV_am0-BVAM}
|Df|_{\AM}\leq |Df|_{am}^0&\quad\forall f\in BV_{am}^0(\X),\\
\label{eq:ineq_BV_am0-am1}
 |Df|_{am}^1(\Omega)\leq|Df|_{am}^0(\Omega)&\quad\forall f\in BV_{am}^0(\X),\,\,\Omega\in\tau(\X),\\
\label{eq:ineq_BV_am-BV_tp}
|Df|_{tp}^\star(\Omega)\leq\min\left\{
|Df|_{am}^1(\Omega),|Df|_{\AM}(\Omega)\right\}&\quad\forall f\in BV_{am}^1(\X)\cup BV_{\AM}(\X),\,\,\Omega\in\tau(\X).
\end{align}
\end{theorem}

\begin{proof}
The inclusion \(BV^\star(\X)\subseteq BV(\X)\), as well as the inequality $|Df|(\X)\leq |Df|^*(\X)$ for any $f\in BV^\star(\X)$, have already been justified in \eqref{eq:triv_ineq_BV}. Similarly,
the inclusion \(BV(\X)\subseteq BV_N(\X)\), as well as the inequality $|Df|_N\leq |Df|$ for any $f\in BV(\X)$, have already been justified in \eqref{Mir_in_New}.  

In order to prove the inclusion \(BV_N(\X)\subseteq BV_{am}^0(\X)\) and \eqref{eq:ineq_BV-BV_am}, fix \(f\in BV_N(\X)\) and an open set \(\Omega\subseteq\X\). From Lemma \ref{lem:BV_in_BV_tildeam} 
we deduce that there exist a $\mm$-negligible Borel set $N$ and a sequence
\((f_n)_n\subseteq N^{1,1}(\Omega)\) such that \(\int_\Omega g_{f_n}\,\d\mm\to|Df|_N(\Omega)\)
and \eqref{eq:compitino} holds.  
Hence, $(g_{f_n})_n$ is an ${\rm AM}_0$-bounding sequence in $\Omega$ and 
\(|Df|_{am}^0(\Omega)\leq|Df|_N(\Omega)\). In particular, since $\Omega$ is
arbitrary, \(|\D f|_{am}^0\leq|\D f|_N\) by the outer regularity of \(|\D f|_N\) and $|\D f|_{am}^0$.

The inclusion \(BV^0_{am}(\X)\subseteq BV_{am}^1(\X)\) and \eqref{eq:ineq_BV_am0-am1} are discussed in the paragraph succeeding Definition \ref{def:BV_am}. Similarly, the inclusion
\(BV^0_{am}(\X)\subseteq BV_{\AM}(\X)\) and \eqref{eq:ineq:BV_am0-BVAM} are proved
in Proposition~\ref{prop:am_0_vs_am_0_pt}. 

In order to prove that \(BV_{am}^1(\X)\subseteq BV_{tp}^\star(\X)\) and the inequality
\(|Df|_{tp}^\star(\Omega)\leq |Df|_{am}^1(\Omega)\)
in \eqref{eq:ineq_BV_am-BV_tp}, fix any \(f\in BV_{am}^1(\X)\) and \(\Omega\in\tau(\X)\).
Given any \(\varepsilon>0\), there exists an \({\rm am}_1\)-bounding sequence \((g_n)_n\) for \(f\) on \(\Omega\) such that 
\begin{equation}\label{eq:choice-gn}
\varliminf_n\int_\Omega g_n\,\d\mm\leq|Df|_{am}^1(\Omega)+\varepsilon.
\end{equation}
In particular,
we can find a family of curves \(\mathcal N\subseteq\mathscr R(\X)\) with \({\rm AM}_1(\mathcal N)=0\) 
such that for every $\gamma\in\mathscr R(\X)\setminus\mathcal N$ with \(\gamma([a_\gamma,b_\gamma])\subseteq\Omega\) 
we have
\begin{equation}\label{eq:BV=BV_tildeAM_aux}
|f(\gamma_{b_\gamma})-f(\gamma_{a_\gamma})|\leq\varliminf_{n\to\infty}\int_\gamma g_n.
\end{equation}
Let \(\tilde{\mathcal N}\subseteq C([0,1];\X)\) be a Borel set associated to \(\mathcal N\cap C([0,1];\X)\) 
Note that any test plan $\ppi\in\Pi(\Omega)$ extends as a test plan in $\Pi(\X)$ by setting $\ppi(\Gamma)=0$ with
$\Gamma=C([0,1];\X)\setminus C([0,1];\Omega)$.
as in Proposition~\ref{prop:AM_vs_tp}, and fix any \(\ppi\in\Pi(\Omega)\). Then \(\ppi(\tilde{\mathcal N})=0\) 
by Proposition \ref{prop:AM_vs_tp}. It follows from~\eqref{eq:BV=BV_tildeAM_aux}  
that for curves $\gamma\in\mathscr R(\Omega)\setminus\mathcal N\subseteq\mathscr R(\X)\setminus\mathcal N$, we have
\begin{equation}\label{eq:essential}
f(\gamma_1)-f(\gamma_0)\le
|f(\gamma_1)-f(\gamma_0)|\leq\varliminf_{n\to\infty}\int_0^1 g_n(\gamma_t)|\dot\gamma_t|\,\d t\quad\text{ for $\ppi$-a.e. $\gamma$}.
\end{equation}
From the discussion following Definition~\ref{def:BV_am}, we can choose specific representatives for $f$, and we now
choose a Borel representative. Such a choice allows us to make sense of the integral
$\int_{C([0,1];\Omega)} f(\gamma_1)-f(\gamma_0)\,\d\ppi(\gamma)$ below; however, this integral is independent of the 
specific choice of Borel representative, because of the compression assumption for the test plans $\ppi$ (applied to the
specific choices of $t=0$ and $t=1$).
Integrating both sides of~\eqref{eq:essential} with respect to \(\ppi\) and using Fatou's lemma, we then deduce that
\begin{align*}
\int_{C([0,1];\Omega)} f(\gamma_1)-f(\gamma_0)\,\d\ppi(\gamma)
&\leq\int_{C([0,1];\Omega)}\left(\varliminf_{n\to\infty}\int_0^1 g_n(\gamma_t)|\dot\gamma_t|\,\d t\right)\,\d\ppi(\gamma)\\
&\leq\varliminf_{n\to\infty}\int_{C([0,1];\Omega)}\, \int_0^1 g_n(\gamma_t)|\dot\gamma_t|\,\d t\,\d\ppi(\gamma)\\
&\leq{\rm Comp}(\ppi)\Lip(\ppi)\varliminf_{n\to\infty}\int_\Omega g_n\,\d\mm\\
&\leq{\rm Comp}(\ppi)\Lip(\ppi)(|Df|_{am}^1(\Omega)+\varepsilon),
\end{align*}
where we used Tonelli's theorem together with the 
property of test plans in the penultimate step, and~\eqref{eq:choice-gn} in obtaining the final step above.
Letting \(\varepsilon\searrow 0\), we conclude that 
\[
\int_{C([0,1];\Omega)} f(\gamma_1)-f(\gamma_0)\,\d\ppi(\gamma)\leq{\rm Comp}(\ppi)\Lip(\ppi)|Df|_{am}^1(\Omega)
\]
for every \(\ppi\in\Pi(\Omega)\). Therefore, \(f\in BV_{tp}^\star(\X)\) and \(|Df|_{tp}^\star(\Omega)\leq|Df|_{am}^1(\Omega)\) for every \(\Omega\in\tau(\X)\).

Finally, the inclusion \(BV_{AM}(\X)\subseteq BV_{tp}^\star(\X)\) 
and the inequality \(|Df|_{tp}^\star(\Omega)\leq |Df|_{AM}(\Omega)\)
follow by combining \eqref{eq:Nages3} of Lemma~\ref{lem:Nages2} with Lemma~\ref{lem:Federichi_Nobili}, which provide 
\eqref{eq:essential}, thanks to the inequality \eqref{eq:tildeAM_vs_tp} between ${\rm AM}_0$-modulus and test plans.
Then, the rest of the proof proceeds as in the proof of the inclusion  of \(BV_{am}^1(\X)\) in \(BV_{tp}^\star(\X)\).
\end{proof}
Recall that the measure $\mm$ is boundedly finite when every bounded Borel subset of $\X$ has finite $\mm$-measure.
Note that we also assume the measure $\mm$ to be a Radon measure, see Definition~\ref{defn:mms}.

\begin{theorem}[Equivalence of BV spaces in complete metric spaces]\label{thm:equiv_BV_complete}
Let \((\X,\sfd,\mm)\) be a metric measure space such that \((\X,\sfd)\) is complete and \(\mm\)
is boundedly finite. Then
\[
BV^\star(\X)=BV(\X)=BV_N(\X)=BV_{am}^0(\X)=BV_{am}^1(\X)=BV_{\AM}(\X)=BV_{tp}^\star(\X).
\]
Moreover, for every \(f\in BV(\X)\) one has
\[
|Df|^\star(\X)=|Df|(\X)=|Df|_N(\X)=|Df|_{am}^0(\X)=|Df|_{am}^1(\X)=|Df|_{\AM}(\X)=|Df|_{tp}^\star(\X)
\]
and also $|Df|=|Df|_N=|Df|_{am}^0$ as measures.
\end{theorem}

\begin{proof}
Thanks to Theorem~\ref{thm:incl_BV_spaces}, to prove the first two chains of identities it suffices to prove that
$BV_{tp}^\star(\X)\subseteq BV^\star(\X)$ with the corresponding identity of the energy seminorms.

Let \(f\in BV_{tp}^\star(\X)\) be given.  As in \cite{Amb:DiMa:14}, 
we will apply  
Theorem~\ref{thm:equiv_BV_basic} of Appendix~\ref{sec:Cheeger1}. To do so,
we do a preliminary reduction to the
case of complete length spaces, with $\mm$ finite with bounded support and with $f$ bounded. 

Since \({\rm spt}(\mm)\) is separable by Proposition
\ref{prop:main_Radon}, there exists a complete separable length space \((\Y,\rho)\) containing
\({\rm spt}(\mm)\) (e.g., isometrically embed \({\rm spt}(\mm)\) into the Banach space \(\ell^\infty\)
via a Kuratowski embedding \(\iota\colon{\rm spt}(\mm)\hookrightarrow\ell^\infty\), and consider the closure
of the linear span of \(\iota({\rm spt}(\mm))\) in \(\ell^\infty\)).
We then define the Radon measure \(\mu\) on \(\Y\) as \(\mu(E)\coloneqq\mm(E\cap{\rm spt}(\mm))\)
for every Borel set \(E\subseteq\Y\). Since \({\rm spt}(\mm)\) is closed in \(\X\) by 
Proposition~\ref{prop:main_Radon} and \((\X,\sfd)\) is complete, also
\(({\rm spt}(\mm),\sfd)=({\rm spt}(\mm), \rho_{\vert {\rm spt}(\mm)})\) is complete. In particular, \({\rm spt}(\mm)\) is closed
in \(\Y\), thus \({\rm spt}(\mu)={\rm spt}(\mm)\) and $\mu$ is a Borel measure on $\Y$. Letting \(g\in\mathcal L^1(\Y)\) be the zero
extension of \(f_{|{\rm spt}(\mm)}\) to \(\Y\), we know from Proposition \ref{prop:invar_BV_tp} that
\(g\in BV_{tp}^\star(\Y)\) and \(|Dg|_{tp}^\star(\Y)=|Df|_{tp}^\star(\X)\). Define
\[
g_k\coloneqq(g\wedge k)\vee(-k)\in\mathcal L^1(\Y)\quad\text{ for every }k\in\N.
\]
Since \(g_k=\phi_k\circ g\), where \(\phi_k\) denotes the \(1\)-Lipschitz function
\(\R\ni t\mapsto(t\wedge k)\vee(-k)\in\R\), we know from Lemma~\ref{lem:Federichi_Nobili} that
\(g_k\in BV_{tp}^\star(\Y)\) and \(|Dg_k|_{tp}^\star(\Y)\leq|Dg|_{tp}^\star(\Y)=|Df|_{tp}^\star(\X)\).
Now, we fix some point \(\bar x\in{\rm spt}(\mm)\) and for a positive integer $n$ we consider the ball
$B_{n+2}(\bar x)=\{\bar y\in\Y\, :\, \rho(\bar y,\bar x)<n+2\}$. Let
\(\Y_n\coloneqq(\Y,\rho,\mu\llcorner B_{n+2}(\bar x))\) and
\[
\eta_n\coloneqq(1-\rho(\cdot,B_n(\bar x)))\vee 0\in\LIP_{bs}(\Y)\quad\text{ for every }n\in\N.
\]
From the definition of $BV_{tp}^\star$ found in Definition~\ref{def:BV-test-plan}, we see that \(g_k\in BV_{tp}^\star(\Y_n)\). 
To see that
\(|D_{\Y_n}g_k|_{tp}^\star(\Y)\leq|D_\Y g_k|_{tp}^\star(\Y)\),  first we note that $\Pi(Y_n)\subseteq\Pi(Y)$
with the constants ${\rm Comp}(\ppi)$ and $\Lip(\ppi)$ the same with respect to both $Y$ and $Y_n$
for $\ppi\in \Pi(Y_n)$, and so by the definition of $|D_{\Y_n}g_k|_{tp}^\star$ as a supremum from Definition~\ref{def:BV-test-plan}, this claim follows.

By Theorem \ref{thm:equiv_BV_basic}, we then have \(g_k\in BV^\star(\Y_n)\)
and \(|D_{\Y_n}g_k|^\star(\Y)\leq|D_{\Y_n}g_k|_{tp}^\star(\Y)\leq|Df|_{tp}^\star(\X)\).
Applying Lemma~\ref{lem:localisation_BV_star} and inequality~\eqref{eq:Leibniz_BV}, we obtain that
\(g_{k,n}\coloneqq\eta_n g_k\in BV^\star(\Y)\) for every \(k,n\in\N\) and
\begin{align*}
|D_\Y g_{k,n}|^\star(\Y)&=|D_{\Y_n}g_{k,n}|(\Y)\leq\int\eta_n\,\d|D_{\Y_n}g_k|+\int|g_k|\lip_a(\eta_n)\,\d\mu\\
&\leq|D_{\Y_n}g_k|^\star(\Y)+\int|g_k|\nchi_{\bar B_{n+1}(\bar x)\setminus B_n(\bar x)}\,\d\mu.
\end{align*}
Since \(\int|g_k|\nchi_{\bar B_{n+1}(\bar x)\setminus B_n(\bar x)}\,\d\mu\to 0\) as \(n\to\infty\) by the
dominated convergence theorem, we deduce that \(\limi_n|D_\Y g_{k,n}|^\star(\Y)\leq|Df|_{tp}^\star(\X)\) for
every \(k\in\N\). Since \(g_{k,n}\to g_k\) as \(n\to\infty\) in $\mathcal{L}^1(Y)$ and
\(g_k\to g\) as \(k\to\infty\)
in \(\mathcal L^1(\Y)\), the lower semicontinuity of \(|D_\Y\cdot|^\star(\Y)\) implies that \(g\in BV^\star(\Y)\) 
with
\[
|D_\Y g|^\star(\Y)\leq\limi_{k\to\infty}\limi_{n\to\infty}|D_\Y g_{k,n}|^\star(\Y)\leq|D f|_{tp}^\star(\X).
\]
Applying Theorem~\ref{thm:invar_BV}, we deduce that \(f\in BV^\star(\X)\) and
\(|Df|^\star(\X)=|D_\Y g|^\star(\Y)\leq|Df|_{tp}^\star(\X)\). Combining this fact with Theorem~\ref{thm:incl_BV_spaces},
we complete the proof.
\end{proof}
As explained in the introduction, in locally complete spaces the space $BV^\star$
should be dropped
from the equivalence list. With the remaining list, by a localization of the previous statement, the following result holds.
\begin{theorem}[Equivalence of BV spaces in locally complete spaces]\label{thm:equiv_BV_loc-complete}
Let \((\X,\sfd,\mm)\) be a metric measure space such that \((\X,\sfd)\) is locally complete. Then
\[
BV(\X)= BV_N(\X)=BV_{am}^0(\X)=BV_{am}^1(\X)=BV_{\AM}(\X)=BV_{tp}^\star(\X).
\]
Moreover, for all open sets $\Omega\subseteq\X$ and all functions $f\in BV(\X)$,
\[
|Df|(\Omega)=|Df|_N(\Omega)=|Df|_{am}^0(\Omega)=|Df|_{am}^1(\Omega)=|Df|_{\AM}(\X)=|Df|_{tp}^\star(\Omega).
\] 
In particular, \(|Df|=|Df|_N=|Df|_{am}^0\) for every \(f\in BV(\X)\).
\end{theorem}
\begin{proof}
Let \(f\in BV_{tp}^\star(\X)\) be given. By Remark \ref{rmk:good_cover}, we can find an increasing sequence
\((\Omega_n)_n\) of open sets in \(\X\) with \({\rm spt}(\mm)\subseteq\Omega\coloneqq\bigcup_{n\in\N}\Omega_n\)
such that \(\mm(\overline\Omega_n)<\infty\) and \((\overline\Omega_n,\sfd)\) is complete for all \(n\in\N\).
By \eqref{eq:restr_BV_tp}, we have that \(f_n\coloneqq f_{|\overline\Omega_n}\in BV_{tp}^\star(\overline\Omega_n)\)
and \(|Df_n|_{tp}^\star(\overline\Omega_n)\leq|Df|_{tp}^\star(\X)\) for all \(n\in\N\). Using
Theorem \ref{thm:equiv_BV_complete}, we obtain that \(f_n\in BV^\star(\overline\Omega_n)\)
and \(|Df_n|^\star(\overline\Omega_n)=|D f_n|_{tp}^\star(\overline\Omega_n)\) for all \(n\in\N\).
Taking Remark~\ref{rmk:local_property_BV} and~\eqref{eq:triv_ineq_BV} into account, we deduce that
\begin{align*}
|Df|(\Omega_n)=|Df_{|\Omega_n}|(\Omega_n)=|D{f_n}|(\Omega_n)
&\leq|Df_n|(\overline\Omega_n)\leq|Df_n|^\star(\overline\Omega_n)\\
&=|Df_n|_{tp}^\star(\overline\Omega_n)\leq|Df|_{tp}^\star(\X)\quad\text{ for every }n\in\N.
\end{align*}
The continuity from below of the Borel measure
\(|Df|\) then implies that 
\[
|Df|(\Omega)=\sup_n|Df|(\Omega_n)\leq|Df|_{tp}^\star(\X).
\]
By \eqref{eq:restr_BV} and Remark \ref{rmk:local_property_BV}, we get
\(|Df_{|{\rm spt}(\mm)}|({\rm spt}(\mm))\leq|Df_{|\Omega}|(\Omega)=|Df|(\Omega)\leq|Df|_{tp}^\star(\X)\).
In particular, \(f_{|{\rm spt}(\mm)}\in BV({\rm spt}(\mm))\). Thanks to Theorem \ref{thm:invar_BV}, we
deduce that \(f\in BV(\X)\) and \(|Df|(\X)=|Df_{|{\rm spt}(\mm)}|({\rm spt}(\mm))\leq|Df|_{tp}^\star(\X)\).
Combining this with Theorem \ref{thm:incl_BV_spaces}, we conclude that 
\(BV(\X)= BV_{am}^0(\X)=BV_{am}^1(\X)=BV_{tp}(\X)\)
and \(|Df|(\X)=|Df|_{am}^0(\X)=|Df|_{am}^1(\X)=|Df|_{tp}^\star(\X)\) for every \(f\in BV(\X)\). Finally, since open subsets
of locally-complete spaces are locally complete, we deduce that 
\[
|Df|(\Omega)=|Df_{|\Omega}|(\Omega)=|Df_{|\Omega}|_{tp}^\star(\Omega)=|Df|_{tp}^\star(\Omega)
\]
for every \(f\in BV(\X)\) and \(\Omega\in\tau(\X)\). 
Thus, we also have \(|Df|(\Omega)=|Df|_N(\Omega)=|Df|_{am}^0(\Omega)=|Df|_{am}^1(\Omega)\), 
again by Theorem~\ref{thm:incl_BV_spaces}.
Since \(|Df|\), \(|Df|_N\) and \(|Df|_{am}^0\) are outer regular, it also follows that \(|Df|=|Df|_{am}^0\).
\end{proof}
\begin{remark}\label{rmk:produce_meas_loc-compl}{\rm
Theorem \ref{thm:equiv_BV_loc-complete} implies that if \((\X,\sfd,\mm)\) is a metric measure space such that \((\X,\sfd)\) is locally
complete and \(f\in BV_{am}^1(\X)\) (resp.\ \(f\in BV_{tp}^\star(\X)\)) is given, then the set-function \(|Df|_{am}^1\colon\tau(\X)\to[0,\infty)\)
(resp.\ \(|Df|_{tp}^\star\colon\tau(\X)\to[0,\infty)\)) can be uniquely extended to an outer regular Borel measure on \(\X\) (which, in fact, coincides with \(|Df|\)).
\fr}\end{remark}
\section{Equivalence of \texorpdfstring{$BV$}{BV} and \texorpdfstring{$BV_{\AM}$}{BVAM} under Poincar\'e inequalities}\label{sec:doubling}
Note that Theorem~\ref{thm:incl_BV_spaces},
which establishes the inclusion $BV(\X)\subseteq BV_{\AM}(\X)$ with $|Df|(\X)\leq |Df|_{\AM}(\X)$, does not require much of
$\X$ beyond that the measure on $\X$ is Borel outer-regular. In this section we reproduce the proof of equivalence of $BV$ 
and $BV_{\AM}$ spaces originally given in~\cite{DuCa:ErBiq:Ko:Shan:19},
under the assumptions that the measure is doubling and supports a $1$-Poincar\'e 
inequality. In~\cite{DuCa:ErBiq:Ko:Shan:19} the requirement that $\mathscr{H}^1(\gamma(N_\gamma))=0$
in Definition~\ref{defAMbounding}
was replaced by the more restrictive condition that $\mathscr{L}^1(N_\gamma)=0$. This condition is 
too strong, see the discussion preceding Example~\ref{ex:strong_am_bound}, and the proof of the equivalence
as given in~\cite{DuCa:ErBiq:Ko:Shan:19} therefore has a gap. We correct this error in this section, since the
results of~\cite{DuCa:ErBiq:Ko:Shan:19} also connect the discussion of BV functions with a Semmes pencil of curves,
a topic of independent interest.

We say that a measure $\mm$ is doubling on a metric space $\X$ if there is a constant $C\ge 1$ such that 
$0<\mm(B_{2r}(x))\le C\, \mm(B_r(x))<\infty$ whenever $x\in \X$ and $r>0$.

Throughout this section, we will assume that:
\begin{equation}\label{eq:standing}
 \begin{cases}&\text{$(\X,\sfd)$ is a locally complete metric space,}\\ 
&\text{$\mm$ is a doubling measure with ${\rm spt}(\mm)=\X$}\\ 
&\text{and that for every $\eps>0$ and $x\in\X$ there exists $r_{\eps,x}>0$ such that}\\
&\text{$\mm(B_r(x))\le \eps r$ whenever $0<r\le r_{\eps,x}$.} 
\end{cases}
\end{equation}

We start with two notions of Poincar\'e inequality, and prove their equivalence with Semmes' notion of pencil of curves. 

We say that $(\X,\sfd,\mm)$ supports a \emph{$1$-Poincar\'e inequality} if there are constants $C_1>0$ and $\lambda>0$ such that 
for each ball $B=B_r(x)\subseteq\X$ and $(u,g)$ a function-upper gradient pair in $\X$, we have
\[
\fint_B|u-u_B|\, \d\mm\le C_1\, r\, \fint_{\lambda B}g\, \d\mm,
\]
where $\lambda B$ stands for the concentric ball $B_{\lambda r}(x)$, and $u_B$ denotes 
\[
\frac{1}{\mu(B)}\int_Bu\, \d\mm=\fint_B\, u\, \d\mm.
\]
A direct limiting argument with
monotone approximation from below tells us that the $1$-Poincar\'e inequality implies 
\[
\fint_B|u-u_B|\, \d\mm\le C_1\, r\, \frac{|Du|(\lambda B)}{\mm(\lambda B)}\qquad\forall u\in BV(\X).
\]

For $x,\,y\in\X$ we set $R_{xy}$ to be the function given by
\[
R_{xy}(z):=\frac{\sfd(x,z)}{\mm(B_{\sfd(x,z)}(x))}+\frac{\sfd(y,z)}{\mm(B_{\sfd(y,z)}(y))},
\]
and for $C\ge 1$ we set $\Gamma_{xy}^C$ to be the collection of all curves $\gamma\in \mathscr R(\X)$ with $\gamma_{a_\gamma}=x$,
$\gamma_{b_\gamma}=y$, and $\ell(\gamma)\le C\,\sfd(x,y)$.
We say that $(\X,\sfd,\mm)$ supports a \emph{Semmes pencil of curves} if there 
exist constants $C\ge 1$ and
$C_S\ge 1$ so that for each distinct pairs of
points $x,\,y\in\X$ there is a probability measure $\sigma_{xy}$ on $\Gamma_{xy}^C$
so that 
\[
\int_{\Gamma_{xy}^C}\ell(\gamma\cap A)\, \d\sigma_{xy}(\gamma)\le C_S \, \int_{A\cap C_S B_{xy}} R_{xy}\,\d\mm,
\]
where $C_SB_{xy}\coloneqq B(x, C_S\sfd(x,y))\cup B(y,C_S\sfd(x,y))$ and $\ell(\gamma\cap A)$ stands in as a short-hand notation
for $\mathscr{H}^1(A\cap\gamma([0,1]))$. In particular, by Cavalieri's principle, one has
\begin{equation}\label{eq:Cavalieri}
\int_{\Gamma_{xy}^C}\int_\gamma g\, \d\sigma_{xy}(\gamma)\le C_S \, \int_{C_S B_{xy}} g R_{xy}\,\d\mm,
\end{equation}
for any nonnegative Borel function $g:\X\to [0,\infty]$.

From the discussion in \cite{Heinonen,HKST}, we know that
if $(\X,\sfd,\mm)$ supports a Semmes pencil of curves, then it supports a $1$-Poincar\'e inequality. 
The 
proof of
result~\cite[Theorem~3.7]{DuCa:ErBiq:Ko:Shan:19}, as given there, is correct.

\begin{lemma}[{\cite[Theorem~3.7]{DuCa:ErBiq:Ko:Shan:19}}]
If $(\X,\sfd,\mm)$ supports a $1$-Poincar\'e inequality, then it supports a Semmes pencil of curves.
\end{lemma}

The following lemma, which provides a control as in the definition of AM-bounding sequence on \textit{all} curves in
$\mathscr R(\X)$, will be useful.
\begin{lemma}\label{lem:AMbound-vs-universal}
Suppose that $u\in BV_{\AM}(\X)$ and that $(\rho_k)_k$ is an AM-bounding sequence for $u$. Then for each $\eps>0$ there is an AM-bounding sequence $(g_k)_k$ for $u$ such that: 
\begin{enumerate}
\item[\rm{(a)}] For each $\gamma\in \mathscr R(\X)$ there exists $N_\gamma\subseteq [a_\gamma,b_\gamma]$ with 
$\mathscr{H}^1(\gamma(N_\gamma))=0$ and for each $a,b\in[a_\gamma,b_\gamma]\setminus N_\gamma$ with $a<b$,
\[
|u(\gamma_b)-u(\gamma_a)|\le \varliminf_k\, \int_{\gamma_{|[a,b]}}\, g_k;
\]
\item[\rm{(b)}] we have
\[
\varliminf_k\int_{\X} g_k\,\d\mm\le \eps+\varliminf_k\int_{\X} \rho_k\, \d\mm.
\]
\end{enumerate}
\end{lemma}
\begin{proof}
Given the AM-bounding sequence, we have a family $\Gamma\subseteq \mathscr R(\X)$ with ${\rm AM}_0(\Gamma)=0$ such that
for each $\gamma\in \mathscr R(\X)\setminus \Gamma$ the existence of $N_\gamma$ and the corresponding inequality pairing 
$u$ with $(\rho_k)_k$ are guaranteed. By \ref{Fuglede2} of Proposition~\ref{prop:properties_AM} we can find a family $(\widetilde{\rho}_k)_k$ such that
$\sup_k\int_\X \widetilde{\rho}_k\,\d\mm\leq 1$ and $\varliminf_k\int_\gamma \widetilde{\rho}_k=\infty$ for each $\gamma\in\Gamma$.
We enlarge the collection $\Gamma$ so that all curves $\gamma\in \mathscr R(\X)$ for which the above limit infimum is infinity belong to $\Gamma$.

Now, for each $\eps>0$, choosing $g_k=\rho_k+\eps\, \widetilde{\rho}_k$ proves the lemma. Indeed, if $\gamma\in\Gamma$ and
all sub-curves of $\gamma$ are also in $\Gamma$, then we can take $N_\gamma$ to be empty. If there are sub-curves of 
$\gamma$ that are not in $\Gamma$, then let $\mathcal{I}(\gamma)$ be the countable collection of all intervals 
$I=[a_0,b_0]\subseteq[a_\gamma,b_\gamma]$
with $a_0,b_0\in\mathbb{Q}$ such that $\varliminf_k\int_{\gamma_{|I}}\, \widetilde{\rho}_k<\infty$; then $\gamma_{|I}\not\in\Gamma$,
and so there is an associated null-set $N_I\subseteq[a_0,b_0]$ with $\mathscr{H}^1(\gamma(N_I))=0$. We can then set
$N_\gamma\coloneqq\bigcup_{I\in\mathcal{I}(\gamma)}N_I$ to obtain the validity of condition~(a) of the statement of the lemma.
\end{proof}

Analogously, we say that $(\X,\sfd,\mm)$ supports an AM-Poincar\'e inequality if there are constants $C_2\ge 1$
and $\lambda\ge 1$ such that
\[
\fint_B|u-u_B|\, \d\mm\le C_2\, r\, \frac{|Du|_{\AM}(\lambda B)}{\mm(\lambda B)}\qquad\forall u\in BV_{\AM}(\X)
\]
whenever $B$ is a ball in $\X$. From the direct inequalities of Theorem~\ref{thm:incl_BV_spaces}, which allow
to bound $|Du|_{AM}$ from above with $|Du|$, we know that if 
$(\X,\sfd,\mm)$ supports an AM-Poincar\'e inequality, then it supports a $1$-Poincar\'e 
inequality with $C_2=C_1$.

\begin{lemma}
Suppose that $(\X,\sfd,\mm)$ supports a Semmes pencil of curves. Then it supports an AM-Poincar\'e inequality.
\end{lemma}

The proof of this lemma follows along the lines of the proof of~\cite[Proposition~3.9]{DuCa:ErBiq:Ko:Shan:19}, but we provide the
complete proof here in order to rectify the typo mentioned at the beginning of this section.

\begin{proof} Without loss of generality, by a truncation argument, we assume
$u\in BV_{\AM}(\X)$ to be bounded,
and let $M$ be the collection of all $x\in\X$ for which
\[
\varlimsup_{r\to 0^+}\fint_{B_r(x)}|u-u(x)|\, \d\mm>0.
\]
By the Lebesgue differentiation theorem (which holds because $\mm$ is doubling), it follows that $\mm(M)=0$. 
Thanks to Lemma~\ref{lem:BVAM-measurable}, we can perturb functions in $BV_{\AM}(\X)$ arbitrarily on sets of
$\mm$-measure zero. Since $\mm$ is a Borel regular measure, we can choose a representative of $u\in BV_{\AM}(\X)$
so that $u$ is also Borel on $\X$. This choice is crucial for us when we use the Semmes pencil property below.

For $\eps>0$ and $y\in\X$, set $E_\eps(y)=\{z\in\X:\ |u(z)-u(y)|>\eps\}$.
As $x,\,y\not\in M$, it follows that 
\[
\lim_{r\to 0^+}\frac{\mm(B_r(x)\cap E_\eps(x))}{\mm(B_r(x))}=0,\qquad
\lim_{r\to 0^+}\frac{\mm(B_r(y)\cap E_\eps(y))}{\mm(B_r(y))}=0.
\]
We now fix $x, y\in\X\setminus M$, set $\Gamma_{xy}$ to be the associated Semmes pencil of curves and $\sigma_{xy}$ the probability measure,
and inductively choose a monotone decreasing sequence of positive real numbers $r_i$ with $r_1<\sfd(x,y)/4$ and 
$r_{i+1}<r_i/4$ for each $i\in\N$, such that 
\[
\frac{\mm(B_{r_i}(x)\cap E_\eps(x))}{\mm(B_{r_i}(x))}<\frac{2^{-i}}{4C_D}\ \text{ and }\ \frac{\mm(B_{r_i}(y))\cap E_\eps(y))}{\mm(B_{r_i}(y))}<\frac{2^{-i}}{4C_D},
\]
where $C_D$ is the constant associated with the doubling property of $\mm$. For each $i\in\N$ we also set $\Gamma_i[x,y]$
to be the collection of all curves $\gamma\in\Gamma_{xy}$ for which 
\[
\mathscr{H}^{1}\bigl(\gamma([0,1])\cap(B_{r_i}(x)\setminus [B_{r_i/2}(x)\cup E_\eps(x)])\bigr)=0,
\]
so that $\Gamma_i[y,x]$ is the collection of all $\gamma\in\Gamma_{yx}$ for which
\[
\mathscr{H}^{1}\bigl(\gamma([0,1])\cap(B_{r_i}(y)\setminus [B_{r_i/2}(y)\cup E_\eps(y)])\bigr)=0.
\]
Note that if $\gamma\in \Gamma_i[x,y]$, then in the annulus $B_{r_i}(x)\setminus B_{r_i/2}(x)$, for
$\mathscr{L}^1$-a.e.\ $t\in\hat{\gamma}^{-1}(B_{r_i}(x)\setminus B_{r_i/2}(x))$, we have that 
$\hat{\gamma}(t)\in E_\eps(x)$ and an analogous property holds for $\Gamma_i[y,x]$. Note also that $r_1<\sfd(x,y)/4$ and $\gamma\in\Gamma_{xy}$. 
It follows that choosing
$A_i=E_\eps(x)\cap B_{r_i}(x)\setminus B_{r_i/2}(x)$ in the Semmes pencil of curves we get
\[
\sigma_{xy}(\Gamma_i[x,y])\le \frac{2}{r_i}\, \int_{\Gamma_{xy}} \ell(\gamma\cap A_i)\,\d\sigma_{xy}(\gamma)
\le \frac{2C_S}{r_i}\, \int_{A_i} R_{xy}\, \d\mm.
\]
Note that for $z\in A_i$, we have 
\[
\frac34 \sfd(x,y)<\sfd(x,y)-\sfd(x,z)\le \sfd(y,z)\le \frac54\, \sfd(x,y).
\]
Therefore we have that
\[
\frac{\sfd(y,z)}{\mm(B_{\sfd(y,z)}(y))}\le \frac54\, \frac{\sfd(x,y)}{\mm(B_{3 \sfd(x,y)/4}(y))}.
\]
Now, from the assumptions~\eqref{eq:standing}, with $\eps=\tfrac85\, \frac{\mm(B_{3 \sfd(x,y)/4}(y))}{\sfd(x,y)}$,
we have that for positive integers $i$ for which $2^{1-i}\, \sfd(x,y)\le \min\{r_{\eps,x}, r_{\eps,y}\}$, 
and for $z\in A_i$, necessarily
\[
R_{xy}(z)\approx\frac{r_i}{\mm(B_{r_i}(x))}+\frac{\sfd(y,z)}{\mm(B_{\sfd(y,z)}(y))}\leq 2\, \frac{r_i}{\mm(B_{r_i}(x))}.
\]
Hence
\[
\sigma_{xy}(\Gamma_i[x,y])\lesssim \frac{\mm(A_i)}{\mu(B_{r_i}(x))}\lesssim 2^{-i}.
\]
Similarly, we have that $\sigma_{xy}(\Gamma_i[y,x])\lesssim 2^{-i}$. Hence, with 
\[
\Gamma_b(x,y)\coloneqq\left(\bigcap_{n\in\N}\, \bigcup_{i=n}^\infty\Gamma_i[x,y]\right) \cup \left(\bigcap_{m\in\N}\, \bigcup_{i=m}^\infty\Gamma_i[y,x]\right),
\]
we have that $\sigma_{xy}(\Gamma_b(x,y))=0$. If $\gamma\in\Gamma_{xy}\setminus \Gamma_b(x,y)$, then
for sufficiently large positive integers $i$ we can find $x_i\in B_{r_i}(x)\setminus B_{r_i/2}(x)$ so that $x_i\not\in E_\eps(x)$ and hence
$|u(x_i)-u(x)|<\eps$; a similar choice of $y_i\in B_{r_i}(y)\setminus B_{r_i/2}(y)$ such that $|u(y_i)-u(y)|<\eps$.
Moreover, we can choose $x_i, \,y_i\not\in\gamma(N_\gamma)$ as $\mathscr{H}^1(\gamma(N_\gamma))=0$.
Here we have choices of AM-bounding sequences $(g_k)_k$ as given by Lemma~\ref{lem:AMbound-vs-universal}.
Hence
\[
|u(x)-u(y)|-2\eps\le \varliminf_{i\to\infty} |u(x_i)-u(y_i)|\le \varliminf_k\int_\gamma g_k.
\]
Since $\sigma_{xy}$ is a probability measure, and the above inequality and sequences $x_i,y_i$ hold for $\sigma_{xy}$-almost every
$\gamma\in\Gamma_{xy}$, it follows by Fatou's lemma that 
\[
|u(x)-u(y)|-2\eps\le \int_{\Gamma_{xy}} \varliminf_k\int_\gamma g_k\,\d\sigma_{xy}(\gamma)
 \le \varliminf_k \int_{\Gamma_{xy}}\int_\gamma g_k\,\d\sigma_{xy}(\gamma),
\]
and so by the Semmes pencil condition, and by the fact that each $g_k$ is a Borel function, 
taking \eqref{eq:Cavalieri} into account, we have
\[
|u(x)-u(y)|-2\eps\le C_S \int_{C_S B_{xy}} g_k\, R_{xy}\,\d\mm.
\]
The remainder of the proof for obtaining the AM-Poincar\'e inequality follows
exactly as in~\cite[Proposition~3.9]{DuCa:ErBiq:Ko:Shan:19}.
\end{proof}

Combining the above discussions, we have the following theorem, see~\cite[Theorem~3.10]{DuCa:ErBiq:Ko:Shan:19}.

\begin{theorem}
Under the assumptions \eqref{eq:standing}, the following three conditions are equivalent:
\begin{enumerate}
\item[{\rm(a)}] the space supports a $1$-Poincar\'e inequality,
\item[\rm{(b)}] the space supports a Semmes pencil of curves,
\item[\rm{(c)}] the space supports an AM-Poincar\'e inequality.
\end{enumerate}
In the above set of equivalences, constants related to one condition can be chosen to depend solely on the constants related
to either of the other two conditions and to the constant associated with the doubling property of $\mm$.
\end{theorem}

The next theorem establishes the equivalence of the two spaces, and the proof can be found in~\cite{DuCa:ErBiq:Ko:Shan:19}.
The proof of \cite[Proposition~3.3]{DuCa:ErBiq:Ko:Shan:19} is correct as given there. The basic argument there
is that thanks to the AM-Poincar\'e inequality, we can find a sequence of locally Lipschitz functions belonging to
$N^{1,1}(\X)$ that approximate the given $u\in BV_{\AM}(\X)$ in $L^1(X)$ and pointwise $\mm$-almost everywhere,
such that the upper gradient energy of the approximating functions is majorized by a constant multiple of
$|Du|_{\AM}(\X)$, and so $u\in BV_N(\X)=BV(\X)$.

\begin{theorem}[{\cite[Proposition~3.3]{DuCa:ErBiq:Ko:Shan:19}}]
Under assumption \eqref{eq:standing}, if $(\X,\sfd,\mm)$ supports a $1$-Poincar\'e inequality, then 
$BV(\X)=BV_{\AM}(\X)$.
\end{theorem}

In addition, the proof provides the inequalities
\[
|Du|(\X)\approx |Du|_{\AM}(\X)\qquad\forall u\in BV(\X),
\]
with the comparison constants depending solely on the
constants related to the doubling property of $\mu$ and the constants associated with the $1$-Poincar\'e inequality.
\appendix
\section{Gradient flow of the Cheeger \texorpdfstring{\(1\)}{1}-energy}\label{sec:Cheeger1}
In this section we give results, relevant to our paper, that are from \cite{Amb:DiMa:14}, with
minor modifications.
The main goal of this section is to prove Theorem~\ref{thm:equiv_BV_basic}. This theorem is based on 
a notion called \emph{the gradient flow of ${\rm Ch}$}. The following theorem, Theorem~\ref{thm:grad_flow}, is from~\cite[Section~6]{Amb:DiMa:14},
and gives also the formulation of this notion. For this reason, we state Theorem~\ref{thm:grad_flow} before the statement of the
main theorem of this section, Theorem~\ref{thm:equiv_BV_basic}.
Let \((\X,\sfd,\mm)\) be a metric measure space with \(\mm(\X)<\infty\).
Following \cite{Amb:DiMa:14}, we define the \emph{Cheeger \(1\)-energy} functional \({\rm Ch}\colon L^2(\X)\to[0,\infty]\) as
\[
{\rm Ch}(f)\coloneqq|Df|^\star(\X)\quad\text{ for every }f\in L^2(\X),
\]
see Definition~\ref{def:BV-star}.
The finiteness domain \(D({\rm Ch})=BV^\star(\X)\cap L^2(\X)\) of \({\rm Ch}\) is dense in \(L^2(\X)\),
since it contains the set \(\{[f]_\mm:f\in\LIP_{bs}(\X)\}\), which is dense in \(L^2(\X)\) by Lemma~\ref{lem:dens_Lip}.
Moreover, it is easy to check that \({\rm Ch}\colon L^2(\X)\to[0,+\infty]\) is convex and lower semicontinuous.
Given any \(f\in L^2(\X)\), we denote by \(\partial^-{\rm Ch}(f)\) the \emph{subdifferential} of \({\rm Ch}\) at \(f\),
which is defined as
\[
\partial^-{\rm Ch}(f)\coloneqq\bigg\{\xi\in L^2(\X)\;\bigg|\;{\rm Ch}(g)\geq{\rm Ch}(f)+\int_{\X}(g-f)\xi\,\d\mm\text{ for every }g\in L^2(\X)\bigg\}.
\]
For any \(f\in D(\Delta_1)\coloneqq\{f\in D({\rm Ch}):\partial^-{\rm Ch}(f)\neq\varnothing\}\), we define the
\emph{\(1\)-Laplacian} of \(f\) as the unique \(\Delta_1 f\in L^2(\X)\) such that \(-\Delta_1 f\)
is the element of minimal \(L^2(\X)\)-norm of \(\partial^-{\rm Ch}(f)\).

The theorem below is from~\cite[Section 6]{Amb:DiMa:14}.
\begin{theorem}[Gradient flow of \({\rm Ch}\)]\label{thm:grad_flow}
Let \((\X,\sfd,\mm)\) be a metric measure space with \(\mm\) finite. Fix any \(f\in L^2(\X)\).
Then there exists a unique continuous curve \([0,\infty)\in t\mapsto f_t\in L^2(\X)\) with \(f_0=f\)
such that \((0,\infty)\ni t\mapsto f_t\in L^2(\X)\) is locally Lipschitz and
\[
\frac{\d}{\d t}f_t=\Delta_1 f_t\quad\text{ for a.e.\ }t>0.
\]
We say that \((f_t)_t\) is the \emph{gradient flow} of \({\rm Ch}\) starting from \(f\). Moreover, the following hold:
\begin{enumerate}[label=(\roman*)]
\item\label{it:Ch_decreas} The function \([0,\infty)\ni t\mapsto{\rm Ch}(f_t)\) is non-increasing.
\item\label{it:mass_preserv} \textsc{Mass preservation.} It holds that \(\int f_t\,\d\mm=\int f\,\d\mm\) for every \(t\geq 0\).
\item\label{it:maxim_principle} \textsc{Maximum/minimum principle.} If \(f\leq C\) holds \(\mm\)-a.e.\ for some constant 
\(C\in\R\) (resp.\ \(f\geq c\) holds \(\mm\)-a.e.\ for some constant \(c\in\R\)), then \(f_t\leq C\) holds \(\mm\)-a.e.\ for
every \(t\geq 0\) (resp.\ \(f_t\geq c\) holds \(\mm\)-a.e.\ for every \(t\geq 0\)).
\item\label{it:en_dissip} \textsc{Energy dissipation.} If there exist \(c,C>0\) such that \(c\leq f\leq C\) holds \(\mm\)-a.e.,
then the curve \((0,+\infty)\ni t\mapsto\int f_t^2\,\d\mm\in[0,\infty)\) is locally Lipschitz and it holds that
\[
-\frac{\d}{\d t}\int f_t^2\,\d\mm=2\,{\rm Ch}(f_t)\quad\text{ for a.e.\ }t>0.
\]
In particular, it holds that \(\int f_0^2\,\mm-\int f_t^2\,\d\mm=2\int_0^t{\rm Ch}(f_s)\,\d s\) for every \(t\geq 0\).
\end{enumerate}
\end{theorem}
The gradient flow of \({\rm Ch}\) has the following properties: if \((f_t)_t\) is the gradient flow of \({\rm Ch}\)
starting from \(f\in L^2(\X)\) and \(a\in\R\) is a given constant, then
\[
(f_t+a)_t\quad\text{ is the gradient flow of }{\rm Ch}\text{ starting from }f+a.
\]
The main goal of this section is to prove the following approximation result.
\begin{theorem}\label{thm:equiv_BV_basic}
Let \((\X,\sfd,\mm)\) be a metric measure space such that
\begin{equation}\label{eq:hp_equiv_BV_basic}\begin{split}
&(\X,\sfd)\text{ is a complete separable length space},\\
&\mm(\X)<\infty\text{ and }{\rm spt}(\mm)\text{ is bounded.}
\end{split}\end{equation}
Let \(f\in BV_{tp}^\star(\X)\cap L^\infty(\X)\) be a given function.
Let \((f_t)_t\subseteq BV^\star(\X)\) be the gradient flow of \({\rm Ch}\) starting from \(f\). Then
\[
\lim_{t\searrow 0}\|f_t-f\|_{L^1(\X)}=0,\qquad|Df|_{tp}^\star(\X)=\lim_{t\searrow 0}|Df_t|^\star(\X).
\]
In particular, it holds that \(f\in BV^\star(\X)\) and \(|Df|^\star(\X)=|Df|_{tp}^\star(\X)\).
\end{theorem}

In order to prove Theorem \ref{thm:equiv_BV_basic}, we first need to discuss several auxiliary definitions and tools.
Letting \((\X,\sfd)\) be a complete separable metric space, we denote by \(\mathscr M_\infty(\X)\) the space
of all finite Borel measures on \(\X\) whose support is bounded, and by $\mathscr P(\X)$ the subspace of probability measures. 
We endow \(\mathscr M_\infty(\X)\) with the
\emph{\(\infty\)-Wasserstein extended distance} \(W_\infty\), which is defined 
for $\mu,\nu\in \mathscr M_\infty(\X)$,
as
\[
W_\infty(\mu,\nu)\coloneqq\inf_\gamma\|\sfd\|_{\mathcal L^\infty(\gamma)}
=\inf_\gamma\, \sup\bigg\lbrace s\ge 0\, :\, \gamma(\{(x,y)\in \X\times\X\, :\, \sfd(x,y)<s\})=0\bigg\rbrace,
\]
where the infimum is taken over all finite Borel measures \(\gamma\) on \(\X\times\X\) whose marginals are
\(\mu\) and \(\nu\),~i.e.\ $\gamma(A\times\X)=\mu(A)$ and $\gamma(\X\times A)=\nu(A)$ for Borel sets $A\subseteq\X$.
Notice that $W_\infty(\mu,\nu)$
is finite if and only if $\mu$ and $\nu$ have the same mass. Moreover, if $W_\infty(\mu,\nu)$ is finite, then we have
$W_\infty( s\mu,s\nu)=W_\infty(\mu,\nu)$
for all $s>0$.

The next theorem is from~\cite[Proposition A.2]{Amb:DiMa:14}.

\begin{theorem}[Superposition principle]\label{thm:superpos_principle}
Let \((\X,\sfd)\) be a complete separable metric space. Let \([0,1]\ni t\mapsto\mu_t\in(\mathscr P_\infty(\X),W_\infty)\)
be Lipschitz. Then there exists a probability measure \(\ppi\) in $C([0,1];\X)$ that is concentrated on \(\LIP([0,1];\X)\),
satisfying \((\e_t)_\#\ppi=\mu_t\) for every \(t\in[0,1]\) and
\begin{equation}\label{eq:Lip_ppi_superpos}
\|\Lip(\cdot)\|_{\mathcal L^\infty(\sppi)}=\Lip(\mu_\cdot).
\end{equation}
\end{theorem}

The proof of the next lemma can be found in~\cite[Lemma 7.3]{Amb:DiMa:14}.

\begin{lemma}[Kuwada's lemma]\label{lem:Kuwada}
Let \((\X,\sfd,\mm)\) be a metric measure space satisfying \eqref{eq:hp_equiv_BV_basic}. Fix \(f\in L^2(\X)\). Assume that
\(c\leq f\leq C\) holds \(\mm\)-a.e.\ for some \(c,C>0\). Let \((f_t)_t\) be the gradient flow
of \({\rm Ch}\) starting from \(f\). Then
\[
[0,+\infty)\ni t\mapsto\mu_t\coloneqq f_t\mm\in(\mathscr M_\infty(\X),W_\infty)\quad\text{is a }\frac{1}{c}\text{-Lipschitz curve.}
\]
\end{lemma}
\begin{proposition}\label{prop:variation_int_f_squared}
Let \((\X,\sfd,\mm)\) be a metric measure space satifying \eqref{eq:hp_equiv_BV_basic}. Assume that
\[
[0,1]\ni t\mapsto\mu_t\coloneqq g_t\mm\in(\mathscr P_\infty(\X),W_\infty)\text{ is a Lipschitz curve,}
\]
for some map \([0,1]\ni t\mapsto g_t\in L^1(\X)\). Assume also that \(g_0\in BV_{tp}^\star(\X)\) and that
there exist constants \(c,C>0\) such that \(c\leq g_t\leq C\) holds \(\mm\)-a.e.\ for every \(t\in[0,1]\). Then it holds that
\[
\int g_0^2\,\d\mm-\int g_t^2\,\d\mm\leq 2C\,\Lip(\mu_\cdot)|Dg_0|_{tp}^\star(\X)t\quad\text{ for every }t\in[0,1].
\]
\end{proposition}
\begin{proof}
Let \(\ppi\in\mathscr P(C([0,1];\X))\) be associated to the curve \(t\mapsto\mu_t\) as in Theorem \ref{thm:superpos_principle}.
Note that \((\e_t)_\#\ppi=\mu_t=f_t\mm\leq C\mm\) for every \(t\in[0,1]\). Taking also \eqref{eq:Lip_ppi_superpos} into account,
we deduce that \(\ppi\) is a test plan on \(\X\) with \({\rm Comp}(\ppi)\leq C\) and
\(\Lip(\ppi)=\|\Lip(\cdot)\|_{\mathcal L^\infty(\sppi)}=\Lip(\mu_\cdot)\eqqcolon L\).
Fix \(t\in(0,1]\) and let
\(\ppi_t\coloneqq({\rm restr}_0^t)_\#\ppi\), we know from \eqref{eq:restr_tp} that \(\ppi_t\in\Pi(\X)\),
\({\rm Comp}(\ppi_t)\leq C\) and \(\Lip(\ppi_t)\leq Lt\). Given that \(a^2-b^2\leq 2a(a-b)\) for every \(a,b\in\R\), we can thus estimate
\begin{align*}
\int g_0^2\,\d\mm-\int g_t^2\,\d\mm&\leq 2\int g_0(g_0-g_t)\,\d\mm=2\int g\,\d\mu_0-2\int g\,\d\mu_t\\
&=2\int g(\gamma_0)- g(\gamma_t)\,\d\ppi(\gamma)=2\int g(\sigma_0)- g(\sigma_1)\,\d\ppi_t(\sigma)\\
&\leq 2\,{\rm Comp}(\ppi_t)\Lip(\ppi_t)|D\bar g|_{tp}^\star(\X)\leq 2CL|Dg_0|_{tp}^\star(\X)t.
\end{align*}
Consequently, the statement is achieved.
\end{proof}
We are now ready to prove Theorem \ref{thm:equiv_BV_basic}.
\begin{proof}[Proof of Theorem \ref{thm:equiv_BV_basic}] 
We set $m=\int f\,\d\mm$ and we assume $f\geq 0$ with no loss of generality.
Given any \(k\in\N\) with $k>\|f\|_\infty$, let us define \(f^k\coloneqq f+k\in BV_{tp}^\star(\X)\). Observe that
\[
k\leq f^k\leq k+\|f\|_{L^\infty(\X)}\eqqcolon C_k\quad\text{ holds }\mm\text{-a.e.\ on }\X.
\]
Given that \({\rm spt}(\mm)\) is bounded, we have that \(f^k\mm\in\mathscr M_\infty(\X)\) and we denote by $f^k_t$ the gradient flow
of ${\rm Ch}$ starting from $f^k$. We have that \(k\leq f^k_t\leq C_k\) holds \(\mm\)-a.e.\ for every \(t\geq 0\) by 
Theorem \ref{thm:grad_flow}\ref{it:maxim_principle} and that $\int f^k_t\,\d\mm=m+2k$ for every $t\geq 0$. 
In addition, Lemma \ref{lem:Kuwada} ensures that \([0,+\infty)\ni t\mapsto f^k_t\mm\in \mathscr M_\infty(\X)\) is 
\(k^{-1}\)-Lipschitz with respect to $W_\infty$.

Set now $g^k=f^k/(m+k)$ and $g^k_t=f^k_t/(m+k)$, so that $g^k_t\mm\in \mathscr P_\infty(\X)$ are probability measures and
\([0,+\infty)\ni t\mapsto g^k_t\mm\in\mathscr P_\infty(\X)\) is still \(k^{-1}\)-Lipschitz with respect to $W_\infty$.
By using Theorem \ref{thm:grad_flow}\ref{it:en_dissip}, Proposition \ref{prop:variation_int_f_squared} applied
to \(t\mapsto g^k_t\mm\) and the fact that
\[
(m+k)|Dg^k|_{tp}^\star(\X)=|Df^k|_{tp}^\star(\X)=|Df|_{tp}^\star(\X),
\]
for any given \(t\in(0,1)\) we obtain that
\begin{align*}
2\int_0^t{\rm Ch}(f^k_s)\,\d s&=2(m+k)\int_0^t{\rm Ch}(g^k_s)\,\d s=(m+k)\bigg(\int(g^k_0)^2\,\d\mm-\int(g^k_t)^2\,\d\mm\bigg)\\
&\leq 2(m+k)\frac{C_k}{k}|Dg^k|_{tp}^\star(\X)t=2\frac{k+\|f\|_{L^\infty(\X)}}{k}|Df|_{tp}^\star(\X)t.
\end{align*}
Letting \(k\to\infty\) and using the fact that \(t\,{\rm Ch}(f_t)\leq\int_0^t{\rm Ch}(f_s)\,\d s=\int_0^t{\rm Ch}(f^k_s)\,\d s\)
 (by Theorem \ref{thm:grad_flow}\ref{it:Ch_decreas}), we get
\[
2t\,{\rm Ch}(f_t)\leq 2\int_0^t{\rm Ch}(f_s)\,\d s\leq\lim_{k\to\infty}2\frac{k+\|f\|_{L^\infty(\X)}}{k}|Df|_{tp}^\star(\X)t=2|Df|_{tp}^\star(\X)t,
\]
so that \({\rm Ch}(f_t)\leq|Df|_{tp}^\star(\X)\) for every \(t\in(0,1)\). Given that \(f_t\to f\) in \(L^2(\X)\) as \(t\searrow 0\) and the functional \({\rm Ch}\)
is lower semicontinuous, we deduce that
\[
|Df|^\star(\X)={\rm Ch}(f)\leq\limi_{t\searrow 0}{\rm Ch}(f_t)\leq|Df|_{tp}^\star(\X),
\]
in particular \(f\in BV(\X)\). Recalling that \(|Df|_{tp}^\star(\X)\leq|Df|^\star(\X)\) by Theorem \ref{thm:incl_BV_spaces}, we finally
conclude that \(|Df|^\star(\X)=|Df|_{tp}^\star(\X)=\lim_{t\searrow 0}|Df_t|^\star(\X)\). Since \(\|f_t-f\|_{L^1(\X)}\to 0\) as \(t\searrow 0\) by H\"{o}lder's
inequality and the fact that \(\mm(\X)<\infty\), the statement is achieved.
\end{proof}
\section{Alternative notions of BV functions via test plans}\label{sec:tppp}
In this section we recall the original definition of \cite{Amb:DiMa:14}.
\begin{definition}[The space \({BV}_{tp}(\X)\)]
Let \((\X,\sfd,\mm)\) be a metric measure space and \(f\in L^1(\X)\). Then we define the Borel measure \(|Df|_{tp}\)
on \(\X\) as 
\begin{equation}\label{eq:def_|Df|_tp}
|Df|_{tp}\coloneqq\bigvee\bigg\{\frac{1}{{\rm Comp}(\ppi)\Lip(\ppi)}\int\gamma_\#{\rm eV}(f\circ\gamma,\cdot)\,\d\ppi(\gamma)\;\bigg|\;
\ppi\in\Pi(\X),\,\Lip(\ppi)\neq 0\bigg\},
\end{equation}
where $\bigvee$ stands for the supremum in the lattice of nonnegative Borel measures and
${\rm eV}(f\circ\gamma,\cdot)=|D(f\circ\gamma)|$.
Then we declare that \(f\) belongs to \(BV_{tp}(\X)\) provided \(|Df|_{tp}(\X)<\infty\).
\end{definition}
It is easily seen that this definition is well posed, i.e.\ invariant under the choice of a representative $f\in\mathcal L^1(\X)$: indeed,
if $N$ is a Borel $\mm$-neglibile set, then $\gamma_t\notin N$ $\ppi$-a.e.\ for all $t\in [0,1]$. Thanks to Fubini's theorem, it follows
that the set of curves $\gamma$ such that $\mathcal L^1(\{t\in (0,1):\ \gamma_t\in N\})>0$ is $\ppi$-negligible, so that for these curves the
essential variation of $f\circ\gamma$ does not depend on the choice of the representative.

We will need a mild continuity result of functions $f\circ\gamma$ at the extreme points of the curve, taken from \cite{Nob:Pas:Sch:22}.

\begin{proposition}[Boundary regularity of test plans]\label{prop:bdry_reg_tp}
Let \((\X,\sfd,\mm)\) be a metric measure space and \(\ppi\in\Pi(\X)\). Let \(f\in\mathcal L^1(\X)\) be such that \({\rm eV}(f\circ\gamma,(0,1))<\infty\) for \(\ppi\)-a.e.\ \(\gamma\). Then
\[
\lim_{\varepsilon\searrow 0}\frac{1}{\varepsilon}\int_0^\varepsilon|f(\gamma_t)-f(\gamma_0)|\,\d t=
\lim_{\varepsilon\searrow 0}\frac{1}{\varepsilon}\int_{1-\varepsilon}^1|f(\gamma_t)-f(\gamma_1)|\,\d t=0\quad\text{ for }\ppi\text{-a.e.\ }\gamma.
\]
In particular, it holds that
\[
|f(\gamma_1)-f(\gamma_0)|\leq{\rm eV}(f\circ\gamma,(0,1))\quad\text{ for }\ppi\text{-a.e.\ }\gamma.
\]
\end{proposition}
\begin{proof}
Let \(\mathcal N\subseteq C([0,1];\X)\) be a \(\ppi\)-negligible Borel set such that \({\rm eV}(f\circ\gamma,(0,1))<\infty\) for every \(\gamma\in C([0,1];\X)\setminus\mathcal N\).
Fix any \(\gamma\in C([0,1];\X)\setminus\mathcal N\). Thanks to \cite[Theorem 3.28]{Amb:Fus:Pal:00}, there exist (unique) numbers \(c^\ell_\gamma,c^r_\gamma\in\R\) such that
\begin{equation}\label{eq:bdry_tp_1}
|f(\gamma_t)-c^\ell_\gamma|\leq{\rm eV}(f\circ\gamma,(0,t)),\quad|f(\gamma_t)-c^r_\gamma|\leq{\rm eV}(f\circ\gamma,[t,1))\quad\text{ for a.e.\ }t\in(0,1).
\end{equation}
In particular, we can estimate
\[
\varlimsup_{\varepsilon\searrow 0}\frac{1}{\varepsilon}\int_0^\varepsilon|f(\gamma_t)-c^\ell_\gamma|\,\d t\leq
\varlimsup_{\varepsilon\searrow 0}\frac{1}{\varepsilon}\int_0^\varepsilon{\rm eV}(f\circ\gamma,(0,t))\,\d t\leq
\lim_{\varepsilon\searrow 0}{\rm eV}(f\circ\gamma,(0,\varepsilon))=0
\]
and similarly \(\varlimsup_{\varepsilon\searrow 0}\frac{1}{\varepsilon}\int_{1-\varepsilon}^1|f(\gamma_t)-c^r_\gamma|\,\d t\leq\lim_{\varepsilon\searrow 0}{\rm eV}(f\circ\gamma,[1-\varepsilon,1))=0\).
All in all, we have
\begin{equation}\label{eq:bdry_tp_2}
\lim_{\varepsilon\searrow 0}\frac{1}{\varepsilon}\int_0^\varepsilon|f(\gamma_t)-c^\ell_\gamma|\,\d t=
\lim_{\varepsilon\searrow 0}\frac{1}{\varepsilon}\int_{1-\varepsilon}^1|f(\gamma_t)-c^r_\gamma|\,\d t=0\quad\text{ for all }\gamma\in C([0,1];\X)\setminus\mathcal N.
\end{equation}
Now, fix a sequence \((\varepsilon_n)_n\subseteq(0,1/2)\) such that \(\varepsilon_n\searrow 0\). Given \(\delta>0\), by Lemma \ref{lem:dens_Lip}
we can find a function  \(f_\delta\in\LIP_{bs}(\X)\cap\mathcal L^1(\X)\) such that \(\|f_\delta-f\|_{\mathcal L^1(\X)}\leq\delta\). Therefore, we can estimate
\[\begin{split}
&\frac{1}{\varepsilon_n}\int\!\!\!\int_0^{\varepsilon_n}|f(\gamma_t)-f(\gamma_0)|\,\d t\,\d\ppi(\gamma)\\
\leq\,&\frac{1}{\varepsilon_n}\int\!\!\!\int_0^{\varepsilon_n}|f-f_\delta|\circ{\sf e}_t\,\d t\,\d\ppi+
\frac{1}{\varepsilon_n}\int\!\!\!\int_0^{\varepsilon_n}|f_\delta(\gamma_t)-f_\delta(\gamma_0)|\,\d t\,\d\ppi(\gamma)+\int|f-f_\delta|\circ{\sf e}_0\,\d\ppi\\
\leq\,&2\,{\rm Comp}(\ppi)\int|f-f_\delta|\,\d\mm+\frac{\Lip(f_\delta)}{\varepsilon_n}\int\!\!\!\int_0^{\varepsilon_n}\sfd(\gamma_t,\gamma_0)\,\d t\,\d\ppi(\gamma)\\
\leq\,&2\,{\rm Comp}(\ppi)\delta+\Lip(f_\delta)\varepsilon_n\int\Lip(\gamma)\,\d\ppi(\gamma)\leq 2\,{\rm Comp}(\ppi)\delta+\Lip(f_\delta)\Lip(\ppi)\varepsilon_n.
\end{split}\]
Letting first \(n\to\infty\) and then \(\delta\searrow 0\), we obtain that
\(\lim_n\frac{1}{\varepsilon_n}\int\!\!\!\int_0^{\varepsilon_n}|f(\gamma_t)-f(\gamma_0)|\,\d t\,\d\ppi(\gamma)=0\).
Likewise, one can prove that \(\lim_n\frac{1}{\varepsilon_n}\int\!\!\!\int_{1-\varepsilon_n}^1|f(\gamma_t)-f(\gamma_1)|\,\d t\,\d\ppi(\gamma)=0.\)
Hence, we can find a subsequence \((n_k)_k\) and a \(\ppi\)-negligible Borel set \(\tilde{\mathcal N}\subseteq C([0,1];\X)\) containing \(\mathcal N\) such that
\begin{equation}\label{eq:bdry_tp_3}
\lim_{k\to\infty}\frac{1}{\varepsilon_{n_k}}\int_0^{\varepsilon_{n_k}}|f(\gamma_t)-f(\gamma_0)|\,\d t
=\lim_{k\to\infty}\frac{1}{\varepsilon_{n_k}}\int_{1-\varepsilon_{n_k}}^1|f(\gamma_t)-f(\gamma_1)|\,\d t=0
\end{equation}
for every \(\gamma\in C([0,1];\X)\setminus\tilde{\mathcal N}\). Combining \eqref{eq:bdry_tp_3} with \eqref{eq:bdry_tp_2}, we deduce that
\(c^\ell_\gamma=f(\gamma_0)\) and \(c^r_\gamma=f(\gamma_1)\) for every \(\gamma\in C([0,1];\X)\setminus\tilde{\mathcal N}\), along with
\[
\lim_{\varepsilon\searrow 0}\int_0^\varepsilon|f(\gamma_t)-f(\gamma_0)|\,\d t=
\lim_{\varepsilon\searrow 0}\int_{1-\varepsilon}^1|f(\gamma_t)-f(\gamma_1)|\,\d t=0\quad\text{ for all }\gamma\in C([0,1];\X)\setminus\tilde{\mathcal N}.
\]
This proves the first part of the statement. Finally, taking also \eqref{eq:bdry_tp_1} into account, we see that for any \(\gamma\in C([0,1];\X)\setminus\tilde{\mathcal N}\)
we can find \(t_0\in(0,1)\) such that
\[\begin{split}
|f(\gamma_1)-f(\gamma_0)|&=|c^r_\gamma-c^\ell_\gamma|\leq|c^r_\gamma-f(\gamma_{t_0})|+|f(\gamma_{t_0})-c^\ell_\gamma|\\
&\leq{\rm eV}(f\circ\gamma,[t_0,1))+{\rm eV}(f\circ\gamma,(0,t_0))={\rm eV}(f\circ\gamma,(0,1)),
\end{split}\]
which gives the last part of the statement.
\end{proof}

\begin{theorem}\label{thm:BV_along_curves}
Let \((\X,\sfd,\mm)\) be a metric measure space. Then it holds that
\[
BV(\X)\subseteq BV_{tp}(\X)\subseteq BV_{tp}^\star(\X).
\]
Moreover, \(|Df|_{tp}\leq|Df|\) for all \(f\in BV(\X)\), and
\(|Df|_{tp}^\star(\Omega)\leq|Df|_{tp}(\Omega)\) for all \(f\in BV_{tp}(\X)\) and any open set \(\Omega\subseteq\X\).
In particular, if \((\X,\sfd)\) is locally complete, then \(BV_{tp}(\X)=BV(\X)\) and
\[
|Df|_{tp}=|Df|\quad\text{ for every }f\in BV(\X).
\]
\end{theorem}
\begin{proof}
Let \(f\in BV(\X)\) and \(\ppi\in\Pi(\X)\) be given. Fix an open set \(\Omega\subseteq\X\). Take a sequence \((f_n)_n\subseteq\LIP_{b,loc}(\Omega)\cap\mathcal L^1(\Omega)\)
with \(f_n\to f_{|\Omega}\) in \(\mathcal L^1(\Omega)\) and \(\int_\Omega\lip_a(f_n)\,\d\mm\to|Df|(\Omega)\). Since
\[
\int\!\!\!\int_{\gamma^{-1}(\Omega)}|f_n(\gamma_t)-f(\gamma_t)|\,\d t\,\d\ppi(\gamma)\leq\int_0^1\!\!\!\int(\nchi_\Omega|f_n-f|)\circ{\sf e}_t\,\d\ppi\,\d t
\leq{\rm Comp}(\ppi)\int_\Omega|f_n-f|\,\d\mm,
\]
we see (up to passing to a non-relabelled subsequence in \(n\)) that \(f_n\circ\gamma\to f\circ\gamma\) in \(\mathcal L^1(\gamma^{-1}(\Omega))\) for
\(\ppi\)-a.e.\ \(\gamma\). Using the lower semicontinuity of \(\mathcal L^1(\gamma^{-1}(\Omega))\ni g\mapsto{\rm eV}(g,\gamma^{-1}(\Omega))\), the fact that
\({\rm eV}(f_n\circ\gamma,\gamma^{-1}(\Omega))\leq\int_{\gamma^{-1}(\Omega)}\lip_a(f_n)(\gamma_t)|\dot\gamma_t|\,\d t\) and Fatou's lemma, we deduce that
\[\begin{split}
\int\gamma_\#{\rm eV}(f\circ\gamma,\Omega)\,\d\ppi(\gamma)&=\int{\rm eV}(f\circ\gamma,\gamma^{-1}(\Omega))\,\d\ppi(\gamma)
\leq\int\varliminf_{n\to\infty}{\rm eV}(f_n\circ\gamma,\gamma^{-1}(\Omega))\,\d\ppi(\gamma)\\
&\leq\varliminf_{n\to\infty}\int{\rm eV}(f_n\circ\gamma,\gamma^{-1}(\Omega))\,\d\ppi(\gamma)\\
&\leq\varliminf_{n\to\infty}\int\!\!\!\int_{\gamma^{-1}(\Omega)}\lip_a(f_n)(\gamma_t)|\dot\gamma_t|\,\d t\,\d\ppi(\gamma)\\
&\leq\varliminf_{n\to\infty}\int\!\!\!\int_0^1(\nchi_\Omega\lip_a(f_n))(\gamma_t)|\dot\gamma_t|\,\d t\,\d\ppi(\gamma)\\
&\leq{\rm Comp}(\ppi)\Lip(\ppi)\lim_{n\to\infty}\int_\Omega\lip_a(f_n)\,\d\mm={\rm Comp}(\ppi)\Lip(\ppi)|Df|(\Omega).
\end{split}\]
Due to the arbitrariness of \(\Omega\), it follows that
\[
\frac{1}{{\rm Comp}(\ppi)\Lip(\ppi)}\int\gamma_\#{\rm eV}(f\circ\gamma,\cdot)\,\d\ppi(\gamma)\leq|Df|\quad\text{ for every }\ppi\in\Pi(\X)\text{ with }\Lip(\ppi)>0,
\]
so that accordingly \(|Df|_{tp}\leq|Df|\) and thus in particular \(f\in BV_{tp}(\X)\).

Conversely, assume \(f\in BV_{tp}(\X)\). For any \(\Omega\subseteq\X\) open and \(\ppi\in\Pi(\X)\), we have
\begin{equation}\label{eq:alt_BV_1}
\int\gamma_\#{\rm eV}(f\circ\gamma,\Omega)\,\d\ppi(\gamma)\leq{\rm Comp}(\ppi)\Lip(\ppi)|Df|_{tp}(\Omega).
\end{equation}
Choosing \(\Omega=\X\), we get that \(\int{\rm eV}(f\circ\gamma,(0,1))\,\d\ppi(\gamma)=\int\gamma_\#{\rm eV}(f\circ\gamma,\X)\,\d\ppi(\gamma)<\infty\),
thus in particular \({\rm eV}(f\circ\gamma,(0,1))<\infty\) for \(\ppi\)-a.e.\ \(\gamma\in C([0,1];\X)\). Therefore, if we assume in addition that \(\ppi\in\Pi(\Omega)\),
then Proposition \ref{prop:bdry_reg_tp} and \eqref{eq:alt_BV_1} ensure that
\[\begin{split}
\int|f(\gamma_1)-f(\gamma_0)|\,\d\ppi(\gamma)&\leq\int{\rm eV}(f\circ\gamma,(0,1))\,\d\ppi(\gamma)=\int\gamma_\#{\rm eV}(f\circ\gamma,\Omega)\,\d\ppi(\gamma)\\
&\leq{\rm Comp}(\ppi)\Lip(\ppi)|Df|_{tp}(\Omega),
\end{split}\]
whence it follows that \(f\in\mathcal{BV}_{tp}^\star(\X)\) and \(|Df|_{tp}^\star(\Omega)\leq|Df|_{tp}(\Omega)\) for every open set \(\Omega\subseteq\X\).

The last part of the statement is a direct consequence of Theorem~\ref{thm:equiv_BV_loc-complete}.
\end{proof}
 Using the previous theorem together with Proposition~\ref{prop:bdry_reg_tp} we can obtain a characterization of the
measure $|Df|$ based on test plans, valid in any locally complete metric measure space.
\begin{theorem}
Let \((\X,\sfd,\mm)\) be a metric measure space such that \((\X,\sfd)\) is locally complete. Then \(f\in L^1(\X)\) belongs to \(BV(\X)\) if and only if
\[
\sup\bigg\{\frac{1}{{\rm Comp}(\ppi)\Lip(\ppi)}\int{\rm eV}(f\circ\gamma,(0,1))\,\d\ppi(\gamma)\;\bigg|\;\ppi\in\Pi(\X),\,\Lip(\ppi)>0\bigg\}<\infty.
\]
Moreover, if \(f\in BV(\X)\), then for any open set \(\Omega\subseteq\X\) we have that
\begin{equation}\label{eq:equiv_BV_tp_bis}
|Df|(\Omega)=\sup\bigg\{\frac{1}{{\rm Comp}(\ppi)\Lip(\ppi)}\int{\rm eV}(f\circ\gamma,(0,1))\,\d\ppi(\gamma)\;\bigg|\;\ppi\in\Pi(\Omega),\,\Lip(\ppi)>0\bigg\}.
\end{equation}
\end{theorem}
\begin{proof}
For brevity, let us denote by \(V(\Omega)\in[0,\infty]\) the right-hand side of \eqref{eq:equiv_BV_tp_bis}.
On the one hand, assuming that \(f\in BV(\X)\), we have \(f\in {BV}_{tp}(\X)\) and
\(|Df|_{tp}\leq|Df|\) by Theorem \ref{thm:BV_along_curves}, thus for any open set \(\Omega\subseteq\X\)
and any test plan \(\ppi\in\Pi(\Omega)\) with \(\Lip(\ppi)>0\) we can estimate
\[\begin{split}
\frac{1}{{\rm Comp}(\ppi)\Lip(\ppi)}\int{\rm eV}(f\circ\gamma,(0,1))\,\d\ppi(\gamma)
&=\frac{1}{{\rm Comp}(\ppi)\Lip(\ppi)}\int\gamma_\#{\rm eV}(f\circ\gamma,\Omega)\,\d\ppi(\gamma)\\
&\leq|Df|_{tp}(\Omega)=|Df|(\Omega).
\end{split}\]
Taking the supremum over all \(\ppi\in\Pi(\Omega)\) with \(\Lip(\ppi)>0\), we obtain that \(V(\Omega)\leq|Df|(\Omega)\).

On the other hand, assume that \(V(\X)<\infty\). In particular, for any \(\ppi\in\Pi(\X)\) it holds that
\(\int{\rm eV}(f\circ\gamma,(0,1))\,\d\ppi(\gamma)\leq{\rm Comp}(\ppi)\Lip(\ppi)V(\X)<\infty\) and thus
\({\rm eV}(f\circ\gamma,(0,1))<\infty\) for \(\ppi\)-a.e.\ \(\gamma\). Hence, Proposition~\ref{prop:bdry_reg_tp}
ensures that for any \(\Omega\subseteq\X\) open and \(\ppi\in\Pi(\Omega)\) we have
\[
\int|f(\gamma_1)-f(\gamma_0)|\,\d\ppi(\gamma)\leq\int{\rm eV}(f\circ\gamma,(0,1))\,\d\ppi(\gamma)
\leq{\rm Comp}(\ppi)\Lip(\ppi)V(\Omega).
\]
This shows that \(f\in {BV}_{tp}^\star(\X)\) and \(|Df|_{tp}^\star(\Omega)\leq V(\Omega)\) for
every \(\Omega\in\tau(\X)\). Taking Theorem \ref{thm:equiv_BV_loc-complete} into account, we conclude
that \(f\in {BV}(\X)\) and \(|Df|(\Omega)\leq V(\Omega)\) for every \(\Omega\in\tau(\X)\).
\end{proof}
\section{Relaxation of integral functionals depending on 
\texorpdfstring{\(\lip\)}{lip} and \texorpdfstring{\(\lip_a\)}{lip-a}}\label{sec:genrel}

The focus of this appendix is the comparison of the two notions of slopes of a Lipschitz function $f$, namely
$\lip_a(f)$ and $\lip(f)$, as given in~\eqref{eq:def_slope} and~\eqref{asylip}.
From Theorem~\ref{thm:DiMa_Gig_Pra} we know the value of considering $\lip_a(f)$, but $\lip(f)\le \lip_a(f)$ everywhere with
$\lip(f)$ an upper gradient of the Lipschitz function $f$. Indeed, when $(\X,\sfd,\mm)$ supports a doubling measure $\mm$ and
a $1$-Poincar\'e inequality, it follows from the work of Cheeger~\cite{Cheeger} that for Lipschitz functions $f$ the function
$\lip(f)$ is the minimal $1$-weak upper gradient of $f$. Therefore it is useful to consider how the two quantities $\lip(f)$ and
$\lip_a(f)$ are related. Let \((\X,\sfd)\) be a metric space. 
\begin{remark}{\rm
The asymptotic Lipschitz constant of \(f\in\LIP_{loc}(\X)\) can also be written as
\[
\lip_a(f)(x)=\lims_{y,z\to x}\frac{|f(y)-f(z)|}{\sfd(y,z)}\quad\text{ for every accumulation point }x\in\X
\]
and \(\lip_a(f)(x)=0\) for every isolated point \(x\in\X\). Consequently, we have that
\[
\lip(f)(x)\leq\lip_a(f)(x)\quad\text{ for every }x\in\X.
\]
On the other hand, since \(\lip_a(f)\) is always upper semicontinuous, while \(\lip(f)\) is not, it is
easy to construct examples where \(\lip(f)\neq\lip_a(f)\). The example below shows that \(\lip(f)=\lip_a(f)\)
can fail even in the a.e.\ sense.
\fr}\end{remark}
\begin{example}{\rm
Equip \(\R\) with the Euclidean distance \(\sfd\) and the Lebesgue measure \(\mathscr L^1\). Fix a Cantor
set \(C\subseteq[0,1]\) with \(\mathscr L^1(C)>0\) and define \(g\coloneqq\nchi_{\R\setminus C}\in\mathcal L^\infty(\R)\).
Defining \(f\in\LIP(\R)\) as \(f(t)\coloneqq\int_0^t g(s)\,\d s\) for every \(t\in\R\), we have that
\(1\geq\lip(f)(t)=|f'(t)|=g(t)\) for a.e.\ \(t\in\R\). It is then easy to check that the upper semicontinuous
envelope of \(\lip(f)\) is the constant function equal to \(1\), thus in particular
\(\lip_a(f)(t)\geq 1>0=\lip(f)(t)\) for a.e.\ \(t\in C\).
\fr}\end{example}
In the following lemma, we consider a Borel measure $\mu$ on the metric space $(\X, \sfd)$, without requiring
it to be a Radon measure; for this reason we use the notation $\mu$ as opposed to $\mm$, which is reserved for Radon measures.

\begin{lemma}\label{lem:lip_a_vs_lip}
Let \((\X,\sfd)\) be a metric space and \(\mu\) a finite Borel measure on \(\X\). Fix any \(f\in\LIP_{loc}(\X)\)
and \(\varepsilon>0\). Then there exists a Borel set \(E_\varepsilon\subseteq\X\) such that
\(\mu(\X\setminus E_\varepsilon)\leq\varepsilon\) and
\[
\lip_a(f_{|E_\varepsilon})(x)\leq\lip(f)(x)\quad\text{ for every }x\in E_\varepsilon.
\]
\end{lemma}
\begin{proof}
Note that \(\lip(f)(x)=\lim_{n\to\infty}\Phi_n(x)\) for every \(x\in\X\), where we set
\[
\Phi_n(x)\coloneqq\sup\bigg\{\frac{|f(x)-f(y)|}{\sfd(x,y)}\;\bigg|\;y\in\X\setminus\{x\},\,\sfd(x,y)<\frac{2}{n}\bigg\}
\quad\text{ for every }n\in\N\text{ and }x\in\X.
\]
Applying Lusin's theorem and Egorov's theorem, we find a Borel set \(E_\varepsilon\subseteq\X\) with \(\mu(\X\setminus E_\varepsilon)\leq\varepsilon\)
such that \(\lip(f)_{|E_\varepsilon}\) is continuous and
\[
\delta_n\coloneqq\sup_{x\in E_\varepsilon}|\Phi_n(x)-\lip(f)(x)|\to 0\quad\text{ as }n\to\infty.
\]
Therefore, for any \(x\in E_\varepsilon\) we have that
\begin{align*}
\lip_a(f_{|E_\varepsilon})(x)&=\lim_{n\to\infty}\Lip(f_{|E_\varepsilon\cap B_{1/n}(x)})\\
&=\lim_{n\to\infty}\sup\bigg\{\frac{|f(y)-f(z)|}{\sfd(y,z)}\;\bigg|\;y,z\in E_\varepsilon,\,y\neq z,\,\sfd(x,y),\sfd(x,z)<\frac{1}{n}\bigg\}\\
&\leq\varliminf_{n\to\infty}\sup_{E_\varepsilon\cap B_{1/n}(x)}\Phi_n
\leq\lim_{n\to\infty}\bigg(\delta_n+\sup_{E_\varepsilon\cap B_{1/n}(x)}\lip(f)\bigg)=\lip(f)(x),
\end{align*}
proving the statement.
\end{proof}

Given a metric measure space \((\X,\sfd,\mm)\) and \(p\in[1,\infty)\), we denote by
\({\rm sc}^-\Phi\colon L^p(\X)\to[0,\infty]\) the lower semicontinuous envelope of a
given functional \(\Phi\colon L^p(\X)\to[0,\infty]\).
\begin{theorem}[Relaxation of functionals depending on \(\lip\) and \(\lip_a\)]\label{thm:relax_funct_lip}
Let \((\X,\sfd,\mm)\) be a metric measure space and \(p\in[1,\infty)\). Assume \(\mm\) is boundedly finite.
Let \(\phi\colon[0,+\infty)\to[0,+\infty)\) be a non-decreasing continuous function with \(\phi(0)=0\).
Given any map \(\varrho\colon\LIP_{bs}(\X)\to\mathcal L^\infty(\mm)^+\), we define the functional
\(\Phi_\varrho\colon L^p(\X)\to[0,\infty]\) as
\[
\Phi_\varrho(f)\coloneqq\inf\bigg\{\int_{\X}\phi(\varrho(\bar f)(x))\,\d\mm(x)\;\bigg|\;\bar f\in\LIP_{bs}(\X),\,[\bar f]_\mm=f\bigg\}
\quad\text{ for every }f\in L^p(\X).
\]
Then it holds that \({\rm sc}^-\Phi_\lip={\rm sc}^-\Phi_{\lip_a}\).
\end{theorem}
\begin{proof}
Since \(\lip\leq\lip_a\), we have that \(\Phi_\lip\leq\Phi_{\lip_a}\) and thus \({\rm sc}^-\Phi_\lip\leq{\rm sc}^-\Phi_{\lip_a}\).
To prove the converse inequality, fix \(f\in L^p(\X)\). Recall that we assume $\mm$ to be a Radon measure. Therefore, we can pick
\((\bar f_n)_n\subseteq\LIP_{bs}(\X)\) such that
\(f_n\coloneqq[\bar f_n]_\mm\to f\) in \(L^p(\X)\) and \(\int\phi\circ\lip(\bar f_n)\,\d\mm\to{\rm sc}^-\Phi_\lip(f)\).
Fix \(\bar x\in\X\) and define \(\eta_k\coloneqq(1-\sfd(\cdot,B_k(\bar x)))\vee 0\in\LIP_{bs}(\X)\) for every \(k\in\N\).
Now, fix any \(n\in\N\). Since \(\phi\) is non-decreasing, we can estimate
\[
\phi\circ(\eta_k\lip(\bar f_n)+|\bar f_n|\lip_a(\eta_k))\leq\phi\circ\big((\Lip(\bar f_n)+\|\bar f_n\|_{C_b(\X)})\nchi_{{\rm spt}(\bar f_n)}\big).
\]
Since \(\mm\) is boundedly finite and \(\phi(0)=0\), we have that
\[
\int_{\X}\phi\circ\big((\Lip(\bar f_n)+\|\bar f_n\|_{C_b(\X)})\nchi_{{\rm spt}(\bar f_n)}\big)\,\d\mm
=\phi\bigg(\Lip(\bar f_n)+\|\bar f_n\|_{C_b(\X)}\bigg)\mm({\rm spt}(\bar f_n))<\infty.
\]
Moreover, we have that \(\lim_k(\eta_k\lip(\bar f_n)+|\bar f_n|\lip_a(\eta_k))(x)=\lip(\bar f_n)(x)\) for every \(x\in\X\). Using the continuity of \(\phi\)
and the dominated convergence theorem, we deduce that 
\[
Q_n^k\coloneqq\int\phi\circ(\eta_k\lip(\bar f_n)+|\bar f_n|\lip_a(\eta_k))\,\d\mm=\int\phi\circ\lip(\bar f_n)\,\d\mm\quad\text{ as }n\to\infty.
\]
Hence, we can find \(k(n)\in\N\) such that \({\rm spt}(\bar f_n)\subseteq B_{k(n)}(\bar x)\) and \(\big|Q_n-\int\phi\circ\lip(\bar f_n)\,\d\mm\big|\leq 1/n\),
where we set \(Q_n\coloneqq Q_n^{k(n)}\). In particular, we have \(Q_n\to{\rm sc}^-\Phi_\lip(f)\). By virtue of Theorem \ref{thm:DiMa_Gig_Pra} and Lemma
\ref{lem:lip_a_vs_lip}, for any \(n\in\N\) we can find a Borel set \(E_n\subseteq B_n\coloneqq\bar B_{k(n)+1}(\bar x)\) such that
\[
\mm(B_n\setminus E_n)\leq\frac{1}{n}\min\bigg\{\frac{1}{2^p\|\bar f_n\|_{C_b(\X)}^p+1},\frac{1}{\phi(\Lip(\bar f_n)+\|\bar f_n\|_{C_b(\X)}+1)+1}\bigg\}
\]
along with a function \(h_n\in\LIP_b(\X)\) such that
\[
h_n|_{E_n}=\bar f_n|_{E_n},\quad\|h_n\|_{C_b(\X)}\leq\|\bar f_n\|_{C_b(\X)},\quad\Lip(h_n)\leq\Lip(\bar f_n)+1,\quad\lip_a(h_n)\leq\lip(\bar f_n)\text{ on }E_n.
\]
Define \(\bar g_n\coloneqq\eta_{k(n)} h_n\in\LIP_{bs}(\X)\). Observe that \(\bar g_n|_{E_n}=\bar f_n|_{E_n}\), \(\|\bar g_n\|_{C_b(\X)}\leq\|\bar f_n\|_{C_b(\X)}\) and
\[
\Lip(\bar g_n)\leq\Lip(h_n)\|\eta_{k(n)}\|_{C_b(\X)}+\Lip(\eta_{k(n)})\|h_n\|_{C_b(\X)}\leq\Lip(\bar f_n)+\|\bar f_n\|_{C_b(\X)}+1.
\]
Moreover, we have that \(\int_{\X\setminus B_n}\phi\circ\lip_a(\bar g_n)\,\d\mm=0\) and \(\int_{E_n}\phi\circ\lip_a(\bar g_n)\,\d\mm\leq Q_n\), since
\[
\phi\circ\lip_a(\bar g_n)\leq\phi\circ\big(\eta_{k(n)}\lip_a(h_n)+|h_n|\lip_a(\eta_{k(n)})\big)\leq\phi\circ\big(\eta_{k(n)}\lip(\bar f_n)+|\bar f_n|\lip_a(\eta_{k(n)})\big)\text{ on }E_n.
\]
Therefore, we have that \(g_n\coloneqq[\bar g_n]_\mm\in L^p(\X)\) satisfies
\begin{align*}
\|g_n-f_n\|_{L^p(\X)}^p&=\int_{E_n}(1-\eta_{k(n)})^p|\bar f_n|^p\,\d\mm+\int_{B_n\setminus E_n}|\bar g_n-\bar f_n|^p\,\d\mm\\
&=\int_{B_n\setminus E_n}|\bar g_n-\bar f_n|^p\,\d\mm\leq 2^p\|\bar f_n\|_{C_b(\X)}^p\mm(B_n\setminus E_n)\leq\frac{1}{n},\\
\Phi_{\lip_a}(g_n)&\leq\int_{\X}\phi\circ\lip_a(\bar g_n)\,\d\mm\leq Q_n+\int_{B_n\setminus E_n}\phi\circ\big(\eta_{k(n)}\lip_a(h_n)+|h_n|\lip_a(\eta_{k(n)})\big)\,\d\mm\\
&\leq Q_n+\phi\big(\Lip(\bar f_n)+\|\bar f_n\|_{C_b(\X)}+1\big)\mm(B_n\setminus E_n)\leq Q_n+\frac{1}{n}.
\end{align*}
By letting \(n\to\infty\), we finally conclude that \(g_n\to f\) in \(L^p(\mm)\) and thus accordingly
\[
{\rm sc}^-\Phi_{\lip_a}(f)\leq\varliminf_{n\to\infty}\Phi_{\lip_a}(g_n)\leq\lim_{n\to\infty}\bigg(Q_n+\frac{1}{n}\bigg)={\rm sc}^-\Phi_\lip(f),
\]
which yields the statement.
\end{proof}
Choosing $p=1$ and \(\phi(t)\coloneqq t\) for every \(t\geq 0\) in Theorem~\ref{thm:relax_funct_lip} we obtain the following corollary.
Notice that, since Theorem~\ref{thm:relax_funct_lip} has been stated under general growth assumptions, similar results of invariance in the
choice of the pseudo-gradient in the relaxation procedure could be obtained in the Sobolev setting as well.
\begin{corollary}
Let \((\X,\sfd,\mm)\) be a metric measure space such that \(\mm\) is boundedly finite. Then
\[
|Df|^\star(\X)=\inf\left\{\limi_{n\to\infty}\int\lip(f_n)\,\d\mm\;\middle|\;(f_n)_n\subseteq\LIP_{bs}(\X),\,f_n\to f\text{ in }L^1(\X)\right\}
\]
for every \(f\in L^1(\X)\).
\end{corollary}
%
%
%
%

\begin{thebibliography}{10}

\bibitem{Amb:DiMa:14}
{\sc L.~Ambrosio and S.~Di~Marino}, {\em Equivalent definitions of {$BV$} space and of total variation on metric measure spaces}, J. Funct. Anal., 266 (2014), pp.~4150--4188.

\bibitem{ADS}
{\sc L.~Ambrosio, S.~Di~Marino, and G.~Savar\'e}, {\em On the duality between {$p$}-modulus and probability measures}, J. Eur. Math. Soc. (JEMS), 17 (2015), pp.~1817--1853.

\bibitem{Amb:Fus:Pal:00}
{\sc L.~Ambrosio, N.~Fusco, and D.~Pallara}, {\em Functions of bounded variation and free discontinuity problems}, Oxford Mathematical Monographs, The Clarendon Press, Oxford University Press, New York, 2000.

\bibitem{Density}
{\sc L.~Ambrosio, N.~Gigli, and G.~Savar\'e}, {\em Density of {L}ipschitz functions and equivalence of weak gradients in metric measure spaces}, Rev. Mat. Iberoam., 29 (2013), pp.~969--996.

\bibitem{Inventiones}
\leavevmode\vrule height 2pt depth -1.6pt width 23pt, {\em Calculus and heat flow in metric measure spaces and applications to spaces with {R}icci bounds from below}, Invent. Math., 195 (2014), pp.~289--391.

\bibitem{AmbrosioIkonenPasqualetto}
{\sc L.~Ambrosio, T.~Ikonen, D.~Lu\v{c}i\'{c}, and E.~Pasqualetto}, {\em Metric {S}obolev spaces {I}: {E}quivalence of definitions}, Milan J. Math., 92 (2024), pp.~255--347.

\bibitem{Ambrosio_Tilli}
{\sc L.~Ambrosio and P.~Tilli}, {\em Topics on analysis in metric spaces}, vol.~25 of Oxford Lecture Series in Mathematics and its Applications, Oxford University Press, Oxford, 2004.

\bibitem{Cheeger}
{\sc J.~Cheeger}, {\em Differentiability of {L}ipschitz functions on metric measure spaces}, Geom. Funct. Anal., 9 (1999), pp.~428--517.

\bibitem{Cob:Mic:Nic:19}
{\sc {\c{S}}.~Cobza{\c{s}}, R.~Miculescu, and A.~Nicolae}, {\em Lipschitz {F}unctions}, Lecture Notes in Mathematics, Springer International Publishing, 2019.

\bibitem{DiMaPhD:14}
{\sc S.~Di~Marino}, {\em Recent advances on {BV} and {S}obolev spaces in metric measure spaces}, PhD thesis, Scuola Normale Superiore (Pisa), 2014.
\newblock CVGMT preprint.

\bibitem{DiMa:Gig:Pra:21}
{\sc S.~Di~Marino, N.~Gigli, and A.~Pratelli}, {\em Global {L}ipschitz extension preserving local constants}, Rend. Lincei Mat. Appl., 31 (2020), pp.~757--765.

\bibitem{DuCa:ErBiq:Ko:Shan:19}
{\sc E.~Durand-Cartagena, S.~Eriksson-Bique, R.~Korte, and N.~Shanmugalingam}, {\em Equivalence of two {B}{V} classes of functions in metric spaces, and existence of a {S}emmes family of curves under a $1$-{P}oincar\'{e} inequality}, Advances in Calculus of Variations, 14 (2021), pp.~231--245.

\bibitem{Gut:20}
{\sc V.~Gutev}, {\em Lipschitz extensions and approximations}, Journal of Mathematical Analysis and Applications, 491 (2020), p.~124242.

\bibitem{Hajlasz}
{\sc P.~Haj{\l}asz}, {\em Sobolev spaces on metric-measure spaces}, in Heat kernels and analysis on manifolds, graphs, and metric spaces ({P}aris, 2002), vol.~338 of Contemp. Math., Amer. Math. Soc., Providence, RI, 2003, pp.~173--218.

\bibitem{Heinonen}
{\sc J.~Heinonen}, {\em Nonsmooth calculus}, Bull. Amer. Math. Soc. (N.S.), 44 (2007), pp.~163--232.

\bibitem{HKST}
{\sc J.~Heinonen, P.~Koskela, N.~Shanmugalingam, and J.~T. Tyson}, {\em Sobolev spaces on metric measure spaces. An approach based on upper gradients}, vol.~27 of New Mathematical Monographs, Cambridge University Press, Cambridge, 2015.

\bibitem{Plans_modulus_duality}
{\sc V.~Honzlov\'a{}~Exnerov\'a, O.~r. F.~K. Kalenda, J.~Mal\'y, and O.~Martio}, {\em Plans on measures and {$AM$}-modulus}, J. Funct. Anal., 281 (2021), pp.~Paper No. 109205, 35.

\bibitem{HMMartio-1}
{\sc V.~Honzlov\'a{}~Exnerov\'a, J.~Mal\'y, and O.~Martio}, {\em Functions of bounded variation and the {$AM$}-modulus in {$\mathbb R^n$}}, Nonlinear Anal., 177 (2018), pp.~553--571.

\bibitem{Koskela_MacManus}
{\sc P.~Koskela and P.~MacManus}, {\em Quasiconformal mappings and {S}obolev spaces}, Studia Math., 131 (1998), pp.~1--17.

\bibitem{Martio}
{\sc O.~Martio}, {\em Functions of bounded variation and curves in metric measure spaces}, Adv. Calc. Var., 9 (2016), pp.~305--322.

\bibitem{Martio-ConfGeomDyn}
{\sc O.~Martio}, {\em The space of functions of bounded variation on curves in metric measure spaces}, Conform. Geom. Dyn., 20 (2016), pp.~81--96.

\bibitem{McShane}
{\sc E.~J. McShane}, {\em Extension of range of functions}, Bull. Amer. Math. Soc., 40 (1934), pp.~837--842.

\bibitem{Miranda}
{\sc M.~Miranda, Jr.}, {\em Functions of bounded variation on ``good'' metric spaces}, J. Math. Pures Appl. (9), 82 (2003), pp.~975--1004.

\bibitem{Nob:Pas:Sch:22}
{\sc F.~Nobili, E.~Pasqualetto, and T.~Schultz}, {\em On master test plans for the space of {BV} functions}, Adv. Calc. Var., 16 (2023), pp.~1061--1092.

\bibitem{Federichi_Nobili}
{\sc F.~Nobili, F.~Renzi, and F.~Vitillaro}, {\em Mosco-convergence of {C}heeger energies on varying spaces satisfying curvature dimension conditions}, Arxiv 2511.13320,  (2025).

\bibitem{Pas:Sod:25}
{\sc E.~Pasqualetto and G.~E. Sodini}, {\em Functions of bounded variation and {L}ipschitz algebras in metric measure spaces}, ESAIM:COCV, 32 (2026).

\bibitem{RoyFitz}
{\sc H.~L. Royden and P.~M. Fitzpatrick}, {\em Real Analysis}, Pearson, 2010.

\bibitem{Sav:22}
{\sc G.~Savar\'{e}}, {\em Sobolev spaces in extended metric-measure spaces}, in New trends on analysis and geometry in metric spaces, vol.~2296 of Lecture Notes in Math., Springer, Cham, 2022, pp.~117--276.

\bibitem{Srivastava}
{\sc S.~M. Srivastava}, {\em A course on {B}orel sets}, vol.~180 of Graduate Texts in Mathematics, Springer-Verlag, New York, 1998.

\bibitem{Sto:48}
{\sc A.~H. Stone}, {\em Paracompactness and product spaces}, Bulletin of the American Mathematical Society, 54 (1948), pp.~977--982.

\end{thebibliography}
%
%

%
%
%
%
\end{document}